\documentclass[11pt]{article}

\usepackage{amssymb, amsmath, amsthm, latexsym}
\usepackage{fullpage, color}
\usepackage{url}

\usepackage{amsmath,amsfonts,amsbsy,amstext,amscd,amsxtra,multicol}
\usepackage{pifont}
\usepackage{algorithm}
\usepackage{algorithmic}

\usepackage{forloop}
\usepackage{graphicx}
\usepackage{tabularx}
\usepackage{paralist}
\usepackage{mathtools}
\usepackage{tcolorbox}
\usepackage{xcolor}

\usepackage{bbm} 

\usepackage{makecell}
\usepackage{multirow}
\usepackage{booktabs}

\usepackage{nicefrac}       

\usepackage[numbers]{natbib}
\usepackage{sidecap}

\usepackage{pifont}
\usepackage{subfigure}

\newcommand{\R}{\mathbb{R}}

\def\<#1,#2>{\left\langle #1,#2\right\rangle}

\usepackage{mdframed} 
\usepackage{thmtools}

\definecolor{shadecolor}{gray}{0.9}
\declaretheoremstyle[
headfont=\normalfont\bfseries,
notefont=\mdseries, notebraces={(}{)},
bodyfont=\normalfont,
postheadspace=0.5em,
spaceabove=1pt,
mdframed={
  skipabove=8pt,
  skipbelow=8pt,
  hidealllines=true,
  backgroundcolor={shadecolor},
  innerleftmargin=4pt,
  innerrightmargin=4pt}
]{shaded}

\declaretheorem[style=shaded,within=section]{definition}
\declaretheorem[style=shaded,sibling=definition]{theorem}

\declaretheorem[style=shaded,sibling=definition]{lemma}

\newcommand{\argmin}{\mathop{\arg\!\min}}

\newcommand{\circledOne}{\text{\ding{172}}}
\newcommand{\circledTwo}{\text{\ding{173}}}
\newcommand{\circledThree}{\text{\ding{174}}}
\newcommand{\circledFour}{\text{\ding{175}}}

\usepackage[colorinlistoftodos,bordercolor=orange,backgroundcolor=orange!20,linecolor=orange,textsize=scriptsize]{todonotes}

\newcommand{\EE}{\mathbf{E}}

\def\R{\mathbb{R}}

\def\R{\mathbb R}

\def\EE{\mathbb E}

\def\e{\varepsilon}
\def\la{\langle}
\def\ra{\rangle}

\def\rev2#1{{\color{black}#1}} 

\def\tf{\widetilde{f}}

\def\tnmf{\widetilde{\nabla}^m f}

\def \NN {\mathbb N}

\def \EE {\mathbb E}

\usepackage{hyperref}

\usepackage{accents}
\newlength{\dhatheight}

\begin{document}
\title{An Accelerated Directional Derivative Method for Smooth Stochastic Convex Optimization}
\author{Pavel Dvurechensky \quad Eduard Gorbunov \quad Alexander Gasnikov
\thanks{This paper was published in European Journal of Operational Research (DOI: \url{https://doi.org/10.1016/j.ejor.2020.08.027}).\newline
P.~Dvurechensky (\textit{pavel.dvurechensky@wias-berlin.de}) is with Weierstrass Institute for Applied Analysis and Stochastics and 
Institute for Information Transmission Problems RAS. E.~Gorbunov (\textit{eduard.gorbunov@phystech.edu}, \href{https://eduardgorbunov.github.io/}{eduardgorbunov.github.io}) is with Moscow Institute of Physics and Technology and National Research University Higher School of Economics. A.~Gasnikov (\textit{gasnikov@yandex.ru}) is with Moscow Institute of Physics and Technology, National Research University Higher School of Economics and Institute for Information Transmission Problems RAS}}
\date{August 20, 2020}
\maketitle

\begin{abstract}
We consider smooth stochastic convex optimization problems in the context of algorithms which are based on directional derivatives of the objective function. This context can be considered as an intermediate one between derivative-free optimization and gradient-based optimization. 
We assume that at any given point and for any given direction, a stochastic approximation for the directional derivative of the objective function at this point and in this direction is available with some additive noise. 
The noise is assumed to be of an unknown nature, but bounded in the absolute value. We underline that we consider directional derivatives in \textit{any} direction, as opposed to coordinate descent methods which use only derivatives in coordinate directions.
For this setting, we propose a non-accelerated and an accelerated directional derivative method and provide their complexity bounds. 
Our non-accelerated algorithm has a complexity bound which is similar to the gradient-based algorithm, that is, without any dimension-dependent factor.
Our accelerated algorithm has a complexity bound which coincides with the complexity bound of the accelerated gradient-based algorithm up to a factor of square root of the problem dimension. We extend these results to strongly convex problems.
\end{abstract}

\section{Introduction}
\label{sec:intro}
Zero-order or derivative-free optimization considers problems of minimization of a function using only, possibly noisy, observations of its values. This area of optimization has a long history, starting as early as in 1960 \cite{rosenbrock1960automatic,fabian1967stochastic}, see also \cite{brent1973algorithms,spall2003introduction,conn2009introduction}. 
Even an older area of optimization, which started in 19th century \cite{cauchy1847methode}, considers first-order methods which use the information about the gradient of the objective function. 
In this paper, we choose an intermediate class of problems. 
Namely, we assume that at any given point and for any given direction, a noisy stochastic approximation for the directional derivative of the objective function at this point in this direction is available. 
We underline that we consider directional derivatives in \textit{any} direction, as opposed to coordinate descent methods which rely only on derivatives in coordinate directions.
We refer to the class of optimization methods, which use directional derivatives of the objective function, as \textit{directional derivative methods}. 
Unlike well developed areas of derivative-free and first-order stochastic optimization methods, the area of directional derivative optimization methods for stochastic optimization problems is not sufficiently covered in the literature.
This class of optimization methods can be motivated by at least three situations.

The first one is connected to Automatic Differentiation \cite{wengert1964simple}. Assume that the objective function is given as a computer program, which performs elementary arithmetic operations and elementary functions evaluations. 
Automatic Differentiation allows to calculate the gradient of this objective function and the additional computational cost is no more than five times larger than the cost of the evaluation of the objective value. The drawback of this approach is that it requires to store in memory the result of all the intermediate operations, which can require large memory amount. On the contrary, calculation of the directional derivative is easier than the calculation of the full gradient and requires the same memory amount as the calculation of the value of the objective  \cite{kim1984efficient}. 
Since a random vector can be a part of the program input or some randomness can be used during the program execution, stochastic optimization problems can also be considered. 

Importantly, automatic calculation of the directional derivative does not require the objective function to be smooth. This fact motivates the study of directional derivative methods in connection to Deep Learning. Indeed, learning problem is often stated as a problem of minimization of a loss function. A non-smooth activation function, called rectifier, is frequently used in Deep Learning as a building block for the loss function. Formally speaking, this non-smoothness does not allow to use Automatic Differentiation in the form of backpropagation to calculate the gradient of the objective function. At the same time, directional derivatives can be calculated by properly modified backpropagation.

The second motivating situation is connected to quasi-variational inequalities, which are used in modelling of different phenomena, such as sandpile formation and growth \cite{prigozhin1996variational}, determination of lakes and river networks \cite{barett2014lakes}, and superconductivity \cite{barett2010quasi-variational}. It happens that directional derivatives can be calculated for such problems \cite{mordukhovich2007coderivative} as a solution to some auxiliary problem. Since this subproblem can not always be solved exactly, the noise in the directional derivative naturally arises. If the considered physical phenomenon takes place in some random media, stochastic optimization can be a natural approach to use. 

The third motivating situation is connected to derivative-free stochastic optimization. In this situation a gradient approximation, based on the difference of stochastic approximations for the values of the objective in two close points, can be considered as a noisy directional derivative in the direction given by the difference of these two points \cite{dvurechensky2017randomized}. In this case, derivative-free stochastic optimization can be considered as a particular case of directional derivative stochastic optimization.

Motivated by potential presence of non-stochastic noise in the problem,
we assume that the noise in the directional derivative consists of two parts. Similar to stochastic optimization problems, the first part is of a stochastic nature. On the opposite, the second part is an additive noise of an unknown nature, but bounded in the absolute value. 
More precisely, we consider the following optimization problem
\begin{equation}
\label{eq:PrSt}
   \min_{x\in\R^n} \left\{ f(x) := \EE_{\xi}[F(x,\xi)] = \int_{\mathcal{X}}F(x,\xi)dP(x) \right\}, 
\end{equation}
where $\xi$ is a random vector with probability distribution $P(\xi)$, $\xi \in \mathcal{X}$, and for $P$-almost every $\xi \in \mathcal{X}$, the function $F(x,\xi)$ is closed  and  convex. Moreover, we assume that, for $P$ almost every $\xi$, the function $F(x,\xi)$ has gradient $g(x,\xi)$, which is $L(\xi)$-Lipschitz continuous with respect to the Euclidean norm and there exists $L_2\geqslant 0$ such that $\sqrt{\EE_{\xi} L(\xi)^2 } \leqslant L_2 < +\infty$. Under this assumptions, $\EE_{\xi}g(x,\xi) = \nabla f(x)$ and $f$ has $L_2$-Lipschitz continuous gradient with respect to the Euclidean norm. Also we assume that
\begin{equation}
\label{stoch_assumption_on_variance}
    \EE_{\xi}[\|g(x,\xi) - \nabla f(x)\|_2^2] \leqslant \sigma^2,
\end{equation}
where $\|\cdot\|_2$ is the Euclidean norm.

Finally, we assume that an optimization procedure, given a point $x \in \R^{n}$, direction $e \in S_2(1)$ and $\xi$ independently drawn from $P$, can obtain a noisy stochastic approximation $\tf'(x,\xi,e)$ for the directional derivative $\la g(x,\xi),e \ra$: 
\begin{align}
\tf'(x,\xi,e) &= \la g(x,\xi),e \ra + \zeta(x,\xi,e) + \eta(x,\xi,e), \notag \\ \EE_{\xi}(\zeta(x,\xi,e))^2 &\leqslant \Delta_{\zeta}, \;\forall x \in \R^n, \forall e \in S_2(1),  \notag \\ |\eta(x,\xi,e)| & \leqslant  \Delta_{\eta}, \;\forall x \in \R^n, \forall e \in S_2(1), \; \text{a.s. in } \xi,
\label{eq:tf_def}
\end{align}
where $S_2(1)$ is the Euclidean sphere or radius one with the center at the point zero and the values $\Delta_{\zeta}$, $\Delta_{\eta}$ are controlled and can be made as small as it is desired. Note that we use the smoothness of $F(\cdot,\xi)$ to write the directional derivative as $\la g(x,\xi),e \ra$, but we \textit{do not assume} that the whole stochastic gradient $g(x,\xi)$ is available.

It is well-known \cite{lan2012optimal,devolder2011stochastic,dvurechensky2016stochastic,gasnikov2016stochasticInter} that, if the stochastic approximation $g(x,\xi)$ for the gradient of $f$ is available, an accelerated gradient method has complexity bound $O\left(\max\left\{\sqrt{L_2/\e},\sigma^2/\e^2\right\}\right)$, where $\e$ is the target optimization error. The question, to which we give a positive answer in this paper, is as follows.

\textit{Is it possible to solve a smooth stochastic optimization problem with the same $\e$-dependence in the complexity and only noisy observations of the directional derivative?} 


\subsection{Related work}
We first consider the related work on directional derivative optimization methods and, then, a closely related class of derivative-free methods with two-point feedback, the latter meaning that an optimization method uses two function value evaluations on each iteration. Since all the considered methods are randomized, we compare oracle complexity bounds in terms of expectation, that is, a number of directional derivatives or function values evaluations which is sufficient to achieve an error $\e$ in the expected optimization error $\EE f(\hat{x}) - f^*$, where $\hat{x}$ is the output of an algorithm and $f^*$ is the optimal value of $f$.  

\subsubsection{Directional derivative methods}
\textbf{Deterministic smooth optimization problems.}
In \cite{nesterov2017random}, the authors consider the Euclidean case and propose a non-accelerated and an accelerated directional derivative method for smooth convex problems with complexity bounds $O(nL_2/\e)$ and $O(n\sqrt{L_2/\e})$ respectively. 
Also they propose a non-accelerated and an accelerated method for problems with $\mu$-strongly convex objective and prove complexity bounds $O(nL_2/\mu\log_2(1/\e))$ and $O(n\sqrt{L_2/\mu}\log_2(1/\e))$ respectively.
For a more general case of problems with additional bounded noise in directional derivatives, but also for the Euclidean case, an accelerated directional derivative method was proposed in \cite{dvurechensky2017randomized} and a bound $O(n\sqrt{L_2/\e})$ was proved.

We also should mention coordinate descent methods. In the seminal paper \cite{nesterov2012efficiency}, a random coordinate descent for smooth convex and $\mu$-strongly convex optimization problems were proposed and $O(L/\e)$ and $O(L/\mu\log_2(1/\e))$ complexity bounds were proved, where $L$ is an effective Lipschitz constant of the gradient varying from $n$ to some average over coordinates coordinate-wise Lipschitz constant. In the same paper, an accelerated version of random coordinate descent was proposed for convex problems and $O(n\sqrt{L/\e})$ complexity bound was proved. Papers \cite{lee2013efficient,fercoq2015accelerated,lin2014accelerated,shalev-shwartz2014accelerated} generalize accelerated random coordinate descent for different settings, including $\mu$-strongly convex problems, and \cite{nesterov2017efficiency,allen2016even,gasnikov2016accelerated} provide a $O(\sqrt{L/\e})$ and $O(\sqrt{L/\mu}\log_2(1/\e))$ complexity bounds, where $L$ is an effective Lipschitz constant of the gradient varying from $n$ to some average over coordinates coordinate-wise Lipschitz constant, and, in the best case, is dimension-independent.
An accelerated random coordinate descent with inexact coordinate-wise derivatives was proposed in \cite{dvurechensky2017randomized} with $O(n\sqrt{L/\e})$ complexity bound and also a unified view on directional derivative methods, coordinate descent and derivative-free methods.  

\textbf{Stochastic optimization problems.}
A directional derivative method for non-smooth stochastic convex optimization problems was introduced in \cite{nesterov2017random} with a complexity bound $O(n^2/\e^2)$. 
A random coordinate descent method for non-smooth stochastic convex and $\mu$- strongly convex optimization problems were introduced in \cite{dang2015stochastic} with complexity bounds $O(n/\e^2)$ and $O(n/\mu\e)$ respectively. 

\subsubsection{Derivative-free methods}
\textbf{Deterministic smooth optimization problems.}
A non-accelerated and an accelerated derivative-free method for this type of problems were proposed in \cite{nesterov2017random} for the Euclidean case with the bounds $O(nL_2/\e)$ and $O(n\sqrt{L_2/\e})$ respectively. 
The same paper proposed a non-accelerated and an accelerated method for $\mu$-strongly convex problems with complexity bounds $O(nL_2/\mu\log_2(1/\e))$ and $O(n\sqrt{L_2/\mu}\log_2(1/\e))$ respectively.
A non-accelerated derivative-free method for deterministic problems with additional bounded noise in function values was proposed in \cite{bogolubsky2016learning} together with $O(nL_2/\e)$ bound and application to learning parameter of a parametric PageRank model, see also \cite{gasnikov2017about,gasnikov2018about}. Deterministic problems with additional bounded noise in function values were also considered in \cite{dvurechensky2017randomized}, where several accelerated derivative-free methods, including Derivative-Free Block-Coordinate Descent, were proposed and a bound $O(n\sqrt{L/\e})$ was proved, where $L$ depends on the method and, in some sense, characterizes the average over blocks of coordinates Lipschitz constant of the derivative in the block. Mixed first-order/zero-order setting is considered in \cite{beznosikov2020derivative-Free}.
\rev2{After our paper appeared as a preprint, the papers \cite{berahas2019derivative-free,bollapragada2019adaptive} studied derivative-free quasi-Newton methods for problems with noisy function values, and the paper \cite{berahas2019theoretical} reported theoretical and empirical comparison of different gradient approximations for zero-order methods.}

\textbf{Stochastic optimization problems.}
Most of the authors in this group solve a more general problem of bandit convex optimization and obtain bounds on the so-called regret. It is well known \cite{cesa-bianchi2002generalization} that a bound on the regret can be converted to a bound on the expected optimization error.
Non-smooth stochastic optimization problems were considered in \cite{nesterov2017random}, where an $O(n^2/\e^2)$ complexity bound was proved for a derivative-free method. This bound was improved by \cite{duchi2015optimal,gasnikov2016gradient-free,gasnikov2017stochastic,shamir2017optimal,bayandina2017gradient-free,hu16bandit} to\footnote{$\widetilde{O}$ hides polylogarithmic factors $(\ln n)^c$, $c>0$.} $\widetilde{O}(n^{2/q}R_p^2/\e^2)$, where $p\in \{1,2\}$, $\frac{1}{p}+\frac{1}{q} = 1$ and $R_p$ is the radius of the feasible set in the  $p$-norm $\|\cdot\|_p$. 
For non-smooth $\mu_p$-strongly convex  w.r.t. to $p$-norm problems, the authors of \cite{gasnikov2017stochastic,bayandina2017gradient-free} proved a bound $\widetilde{O}(n^{2/q}/(\mu_p\e))$. A version of these methods for non-smooth saddle-point problems is developed in \cite{beznosikov2020gradient-free}.

Intermediate, partially smooth problems with a restrictive assumption of boundedness of $\EE~\|g(x,\xi)\|^2$, were considered in \cite{duchi2015optimal}, where it was proved that a proper modification of Mirror Descent algorithm with derivative-free approximation of the gradient gives a bound  $O(n^{2/q}R_p^2/\e^2)$ for convex problems, improving upon the bound $\widetilde{O}(n^2/\e^2)$ of \cite{agarwal2010optimal}. 
For strongly convex w.r.t $2$-norm problems, the authors of \cite{agarwal2010optimal} obtained a bound $\widetilde{O}(n^2/\e)$, which was later extended for $\mu_p$-strongly convex problems and improved to $\widetilde{O}(n^{2/q}/(\mu_p\e))$ in \cite{gasnikov2017stochastic}.

In the fully smooth case, without the assumption that $\EE \|g(x,\xi)\|^2 < +\infty$, papers \cite{ghadimi2016mini-batch,ghadimi2013stochastic} proposed a derivative-free algorithm for the Euclidean case with the bound 
$$
\widetilde{O}\left(\max\left\{\frac{nL_2R_2}{\e},\frac{n\sigma^2}{\e^2}\right\}\right).
$$
In \cite{dvurechensky2018accelerated}, the authors proposed a non-accelerated and an accelerated derivative-free method with the bounds 
$$
\widetilde{O}\left(\max\left\{\frac{n^{\frac{2}{q}}L_2R_p^2}{\e},\frac{n^{\frac{2}{q}}\sigma^2R_p^2}{\e^2}\right\}\right), \quad \widetilde{O}\left(\max\left\{n^{\frac12+\frac{1}{q}}\sqrt{\frac{L_2R_p^2}{\e}},\frac{n^{\frac{2}{q}}\sigma^2R_p^2}{\e^2}\right\}\right)
$$
respectively, where $R_p$ characterizes the distance in $p$-norm between the starting point of the algorithm and a solution to \eqref{eq:PrSt}, $p\in \{1,2\}$ and $q \in \{2,\infty\}$ is the conjugate to $p$, given by the identity $\frac{1}{p}+\frac{1}{q} = 1$. 

\rev2{The authors of \cite{chen2020accelerated} combine accelerated derivative-free optimization with accelerated variance reduction technique for finite-sum convex problems in the Euclidean setup.}

\rev2{
\textbf{Other works.} For a recent review of derivative-free optimization see \cite{larson2019derivative-free} and for a review of stochastic optimization, including derivative-free optimization, see \cite{powell2019unified}.
}

\subsection{Our contributions}
As we have seen above, only two results on directional derivative methods for non-smooth stochastic convex optimization are available in the literature, and, to the best of our knowledge, nothing is known about directional derivative methods for smooth stochastic convex optimization, even in the well-developed area of random coordinate descent methods. Our main contribution consists in closing this gap in the theory of directional derivative methods for stochastic optimization and considering even more general setting with additional noise of an unknown nature in the directional derivative. 

Our methods are based on two proximal setups \cite{ben-tal2015lectures} characterized by the value\footnote{Strictly speaking, we are able to consider all the intermediate cases $p\in [1,2]$, but we are not aware of any proximal setup which is compatible with $p \notin \{1,2\}$} $p\in \{1,2\}$ and its conjugate $q \in \{2,\infty\}$, given by the identity $\frac{1}{p}+\frac{1}{q} = 1$. The case $p=1$ corresponds to the choice of $1$-norm in $\R^n$ and corresponding prox-function which is strongly convex with respect to this norm (we provide the details below). The case $p=2$ corresponds to the choice of the Euclidean $2$-norm in $\R^n$ and squared Euclidean norm as the prox-function.
As our main contribution, we propose an Accelerated Randomized Directional Derivative (ARDD) algorithm for smooth stochastic optimization based on noisy observations of directional derivative of the objective. Our method has the complexity bound 
\begin{equation}
\label{eq:ARDFDSComplInformal}
\widetilde{O}\left(\max\left\{n^{\frac12+\frac{1}{q}}\sqrt{\frac{L_2R_p^2}{\e}},\frac{n^{\frac{2}{q}}\sigma^2R_p^2}{\e^2}\right\}\right),
\end{equation}
where $R_p$ characterizes the distance in $p$-norm between the starting point of the algorithm and a solution to \eqref{eq:PrSt}. 

As our second contribution, we propose a non-accelerated Randomized Directional Derivative (RDD) algorithm with the complexity bound
\begin{equation}
\label{eq:RDFDSComplInformal}
\widetilde{O}\left(\max\left\{\frac{n^{\frac{2}{q}}L_2R_p^2}{\e},\frac{n^{\frac{2}{q}}\sigma^2R_p^2}{\e^2}\right\}\right).
\end{equation}
Interestingly, \rev2{for this method when $p=1$ and $q=\infty$, we obtain  complexity bound which depends on the dimension $n$ only logarithmically despite we use only noisy directional derivative observations.} 
\rev2{Let us comment on the comparison between the accelerated and non-accelerated method. In the regime of small variance $\sigma^2$ in both bounds the dominating term is the first one. If $p=1$, $q=\infty$ and $L_2R_p^2 < n\e$, then the bound for the non-accelerated method is smaller than that of for the accelerated. In this regime it is preferred to use the non-accelerated method.}

Note that, in the case of \eqref{eq:PrSt} having a sparse solution, our bounds for $p=1$ allow to gain a factor of $\sqrt{n}$ in the complexity of the accelerated method and a factor of $n$ in the complexity of the non-accelerated method in comparison to the Euclidean case $p=2$. 
Indeed, sparsity of a solution $x^*$ means that $\|x^*\|_1 = O(1) \cdot \|x^*\|_2$ and, if the starting point is zero, we obtain $R_1^2 = \|x^*\|_1^2 = O(1) \cdot \|x^*\|_2^2 = O(1) R_2^2$. Hence, the bounds for $p=1$ and $p=2$ can be compared only based on the corresponding powers of $n$, the latter being smaller for the case $p=1$, $q=\infty$.

We underline here that our methods are based on random directions drawn from the uniform distribution on the unit Euclidean sphere and our results for $p=1$ can not be obtained by random coordinate descent.

As our third contribution, we extend the above results to the case when the objective function is additionally known to be $\mu_p$-strongly convex w.r.t. $p$-norm. For this case, we propose an accelerated and a non-accelerated algorithm which respectively have complexity bounds
\begin{equation}
    \label{eq:SCBoundsInformal}
    \widetilde{O}\left(\max\left\{n^{\frac12+\frac{1}{q}}\sqrt{\frac{L_2}{\mu_p}}\log_2\frac{\mu_pR_p^2}{\e},\frac{n^{\frac{2}{q}}\sigma^2}{\mu_p\e}\right\}\right), \quad \widetilde{O}\left(\max\left\{\frac{n^{\frac{2}{q}}L_2}{\mu_p}\log_2\frac{\mu_pR_p^2}{\e},\frac{n^{\frac{2}{q}}\sigma^2}{\mu_p\e}\right\}\right).
\end{equation}
\rev2{In the regime of small variance $\sigma^2$ in both bounds the dominating term is the first one. If $p=1$, $q=\infty$ and $\frac{L_2}{\mu_p} < n$, then the bound for the non-accelerated method is smaller than that of for the accelerated. In this regime of relatively well-conditioned problems it is preferred to use the non-accelerated method.}

As our final contribution, we consider derivative-free smooth stochastic convex optimization with inexact values of the stochastic approximations for the function values as a particular case of optimization using noisy directional derivatives. This allows us to obtain the complexity bounds of \cite{dvurechensky2018accelerated} as a straightforward corollary of our results in this paper. At the same time we obtain new complexity bounds for the strongly convex case which, to the best of our knowledge, were not known in the literature.




Note that our results for accelerated and non-accelerated methods are somewhat similar to the finite-sum minimization problems of the form
$$
\min_{x \in \R^n} \sum_{i=1}^mf_i(x),
$$
where $f_i$ are convex smooth functions. For such problems accelerated methods have complexity $\widetilde{O}(m+\sqrt{mL/\e})$ and non-accelerated methods have complexity $\widetilde{O}(m+L/\e)$ (see, e.g. \cite{allen2016katyusha} for a nice review on the topic). As we see, acceleration allows to take the square root of the second term but for the price of $\sqrt{m}$ and the two bounds can not be directly compared without additional assumptions on the value of $m\e$.

{
\textbf{Special note on \cite{dvurechensky2018accelerated,vorontsova2019accelerated}.}
One of the novelties and insights in the approach of this paper in comparison to \cite{dvurechensky2018accelerated,vorontsova2019accelerated} is to realize that gradient-free methods are a particular case of directional derivative methods with inexact oracle.
Unlike these papers, in the current paper we need to account for two types of inexactness. One is stochastic with bounded second moment and the second is bounded a.s. This is a more complicated assumption than the one in \cite{dvurechensky2018accelerated,vorontsova2019accelerated} and we have to assume that the error values can be controlled, unlike \cite{dvurechensky2018accelerated,vorontsova2019accelerated}. Moreover, since the oracle returns different information, we have to construct our stochastic approximation of the gradient differently, which also changes the proof technique.
We also analyze in this paper the case of strongly convex objective values, which was not done in \cite{dvurechensky2018accelerated,vorontsova2019accelerated}.
}

\subsection{Paper organization}
The rest of the paper is organized as follows. In Section \ref{S:Algo_theor}, both for convex and strongly convex problems, we introduce our algorithms, state their convergence rate theorems and corresponding complexity bounds. Section \ref{S:ARDDProofs} is devoted to proof of the convergence rate theorem for our accelerated method and convex objective functions. Section \ref{sec:rdfds} is devoted to proof of the convergence rate theorem for our non-accelerated method and convex objective functions. \rev2{In Section \ref{S:SCProofs} we provide the proofs for the case of strongly convex objective function. Finally, in Section \ref{sec:numerical_experiments} we provide numerical experiments with two types of objective functions: worst case functions for first-order methods \cite{nesterov2004introduction} and least squares problem.
}

\section{Algorithms and main results}
\label{S:Algo_theor}
In this section, we provide our non-accelerated and accelerated directional derivative methods both for convex and strongly convex problems together with convergence theorems and corresponding complexity bounds. The proofs are rather technical and postponed to next sections.
\subsection{Preliminaries}
We start by introducing necessary objects and technical results.

\noindent \textbf{Proximal setup.}
Let $p\in[1,2]$ and $\|x\|_p$ be the $p$-norm in $\R^n$ defined as 
$$
\|x\|_p^p = \sum\limits_{i=1}^n|x_i|^p, \quad x \in \R^n,
$$
$\|\cdot\|_{q}$ be its dual, defined by $\|g\|_{q} = \max\limits_{x} \big\{ \la g, x \ra, \| x \|_p \leq 1 \big\}$, where $q \in [2,\infty]$ is the conjugate number to $p$, given by $\frac{1}{p} + \frac{1}{q} = 1$, and, for $q = \infty$, by definition $\|x\|_\infty = \max\limits_{i=1,\ldots,n}|x_i|$.

We choose a \textit{prox-function} $d(x)$ which is continuous, convex on $\R^n$ and is $1$-strongly convex on $\R^n$ with respect to $\|\cdot\|_p$, i.e., for any $x, y \in \R^n$ $d(y)-d(x) -\la \nabla d(x) ,y-x \ra \geq \frac12\|y-x\|_p^2$.
Without loss of generality, we assume that $\min\limits_{x\in \R^n} d(x) = 0$.
We define also the corresponding \textit{Bregman divergence} $V[z] (x) = d(x) - d(z) - \la \nabla d(z), x - z \ra$, $x, z \in \R^n$. Note that, by the strong convexity of $d$,
\begin{equation}
\label{eq:VStrConv}
V[z] (x) \geq \frac{1}{2}\|x-z\|_p^2, \quad x, z \in \R^n.
\end{equation}
For the case $p=1$, we choose the following prox-function \cite{ben-tal2015lectures}
\begin{equation}
\label{eq:dp1}
d(x) = \frac{{\rm e}n^{(\kappa-1)(2-\kappa)/\kappa}\ln n}{2} \|x\|_\kappa^2, \quad \kappa=1 + \frac{1}{\ln n}
\end{equation}
and, for the case $p=2$, we choose the prox-function to be the squared Euclidean norm
\begin{equation}
\label{eq:dp2}
d(x) = \frac{1}{2}\|x\|_2^2.
\end{equation}

\noindent \textbf{Main technical lemma.}
In our proofs of complexity bounds, we rely on the following lemma. The proof is rather technical and is provided in the appendix. 
\begin{lemma}
		\label{Lm:MainTechLM}
    Let $e \in RS_2(1)$, i.e be a random vector uniformly distributed on the surface of the unit Euclidean sphere in $\R^n$, $p\in[1,2]$ and $q$ be given by $\frac{1}{p}+\frac{1}{q} = 1$. Then, for $n \geqslant8$ and $\rho_n = \min\{q-1,\,16\ln n - 8\}n^{\frac{2}{q}-1}$,		
		\begin{equation}\label{assumption_e_norm}
        \EE_e\|e\|_q^2 \leq \rho_n,
    \end{equation}
    \begin{equation}\label{assumption_e_product}
        \EE_e\left(\la s,e \ra^2 \|e\|_q^2 \right) \leq \frac{6\rho_n}{n} \|s\|_2^2,     \quad \forall s \in \R^n.
    \end{equation}
		
\end{lemma}

\noindent \textbf{Stochastic approximation of the gradient.} 
Based on the noisy stochastic observations \eqref{eq:tf_def} of the directional derivative, we form the following stochastic approximation of $\nabla f(x)$
\begin{equation}
\label{eq:MiniBatcStocGrad}
\tnmf(x) = \frac{1}{m}\sum\limits_{i=1}^m\tf'(x,\xi_i,e)e,
\end{equation}
where $e \in RS_2(1)$, $\xi_i$, $i=1,...,m$ are independent realizations of $\xi$, $m$ is the \textit{batch size}.

\subsection{Algorithms and main results for convex problems}
Our Accelerated Randomized Directional Derivative (ARDD) method is listed as Algorithm~\ref{Alg:ARDS}.
\begin{algorithm}
	\caption{Accelerated Randomized Directional Derivative (ARDD) method}
	\label{Alg:ARDS}
	\begin{algorithmic}[1]
		\REQUIRE $x_0$~---starting point; $N \geqslant 1$~--- number of iterations; $m \geqslant 1$~--- batch size.
		\ENSURE point $y_N$.
		\STATE $y_0 \leftarrow x_0, \, z_0 \leftarrow x_0$.
		\FOR{$k=0,\, \dots, \, N-1$.}
		\STATE $\alpha_{k+1} \leftarrow \frac{k+2}{96n^2\rho_n L_2}, \, \tau_{k} \leftarrow \frac{1}{48\alpha_{k+1}n^2\rho_n L_2} = \frac{2}{k+2}$.
		\STATE Generate $e_{k+1} \in RS_2(1)$ independently from previous iterations and $\xi_i$, $i=1,...,m$ -- independent realizations of $\xi$. 
		\STATE Calculate 
		$$
		\tnmf(x_{k+1})= \frac{1}{m}\sum\limits_{i=1}^m\tf'(x_{k+1},\xi_i,e)e.
		$$
		\STATE $x_{k+1} \leftarrow \tau_kz_k + (1-\tau_k)y_k $.
		\STATE $y_{k+1} \leftarrow x_{k+1}-\frac{1}{2L_2}\tnmf(x_{k+1})$.
		\STATE $z_{k+1} \leftarrow \argmin\limits_{z\in\R^n} \left\{ {\alpha_{k+1} n \left\langle \tnmf(x_{k+1}), \, z-z_k \right\rangle +V[z_{k}] \left( z \right)}\right\}$.
		\ENDFOR
		\RETURN $y_N$
	\end{algorithmic}
\end{algorithm}
%
\begin{theorem}
\label{Th:ARDFDSConv}
    Let ARDD method be applied to solve problem \eqref{eq:PrSt}. Then
    \begin{equation}\label{eq:ARDFDSConv}
        \begin{array}{rl}
        \EE[f(y_N)] - f(x^*)
        \leqslant \frac{384\Theta_p n^2\rho_nL_2}{N^2} 
        + \frac{4N}{nL_2}\cdot\frac{\sigma^2}{m} + \frac{61N}{24L_2}\Delta_\zeta + \frac{122N}{3L_2}\Delta_\eta^2\\
        + \frac{12\sqrt{2n\Theta_p}}{N^2}\left(\frac{\sqrt{\Delta_\zeta}}{2}+ 2\Delta_\eta\right) + \frac{N^2}{12n\rho_nL_2} \left(\frac{\sqrt{\Delta_\zeta}}{2} 
        + 2\Delta_\eta\right)^2,
    \end{array}
    \end{equation}
    where $\Theta_p = V[z_0](x^*)$ is defined by the chosen proximal setup and  $\EE[\cdot] = \EE_{e_1,\ldots,e_N,\xi_{1,1},\ldots,\xi_{N,m}}[\cdot]$.
\end{theorem}
Before we proceed to the non-accelerated method, we give the appropriate choice of the ARDD method parameters $N$, $m$, and accuracy of the directional derivative evaluation $\Delta_\zeta$, $\Delta_\eta$. These values are chosen such that the r.h.s. of \eqref{eq:ARDFDSConv} is smaller than $\e$. For simplicity we omit numerical constants and summarize the obtained values of the algorithm parameters in Table 1 below. The last row represents the total number $Nm$ of oracle calls, that is, the number of directional derivative evaluations, which was advertised in \eqref{eq:ARDFDSComplInformal}. \rev2{Note that the bound \eqref{eq:ARDFDSConv} allows also to choose the accuracy of the directional derivative evaluation $\Delta_\zeta$, $\Delta_\eta$ decreasing with $N$. This is done by making each term with  $\Delta_\zeta$ or $\Delta_\eta$ in the r.h.s. to be of the same order as the first term.}
{
\renewcommand{\arraystretch}{2}
\begin{table}[h]
    \centering
    {\small
    \begin{tabular}{|c|c|c|}
        \hline
         & $p=1$ & $p=2$ \\
        \hline
        $N$ & $O\left(\sqrt{\frac{ n\ln nL_2 \Theta_1}{\varepsilon}}\right)$ & $O\left(\sqrt{\frac{ n^{2}L_2\Theta_2}{\varepsilon}}\right)$ \\
        \hline
        $m$ & $O\left(\max\left\{1,\sqrt{\frac{\ln n}{n}}\cdot\frac{\sigma^2}{\varepsilon^{3/2}}\cdot\sqrt{\frac{\Theta_1}{L_2}}\right\}\right)$ & $O\left(\max\left\{1,\frac{\sigma^2}{\varepsilon^{3/2}}\cdot\sqrt{\frac{\Theta_2}{L_2}}\right\}\right)$\\
        \hline
        $\Delta_\zeta$ & $O\left(\min\left\{n(\ln n)^2L_2^2\Theta_1,\, \frac{\varepsilon^2}{n\Theta_1},\, \frac{\varepsilon^{\frac{3}{2}}}{\sqrt{n\ln n}}\cdot\sqrt{\frac{L_2}{\Theta_1}}\right\}\right)$ & $O\left(\min\left\{n^3L_2^2\Theta_2,\, \frac{\varepsilon^2}{n\Theta_2},\, \frac{\varepsilon^{\frac{3}{2}}}{n}\cdot\sqrt{\frac{L_2}{\Theta_2}}\right\}\right)$\\
        \hline
        $\Delta_\eta$ & $O\left(\min\left\{\sqrt{n}\ln nL_2\sqrt{\Theta_1},\, \frac{\varepsilon}{\sqrt{n\Theta_1}},\, \frac{\varepsilon^{\frac{3}{4}}}{\sqrt[4]{n\ln n}}\cdot\sqrt[4]{\frac{L_2}{\Theta_1}}\right\}\right)$ & $O\left(\min\left\{n^{\frac{3}{2}}L_2\sqrt{\Theta_2},\, \frac{\varepsilon}{\sqrt{n\Theta_2}},\, \frac{\varepsilon^{\frac{3}{4}}}{\sqrt{n}}\cdot\sqrt[4]{\frac{L_2}{\Theta_2}}\right\}\right)$\\
        \hline
        O-le calls & $O\left(\max\left\{\sqrt{\frac{ n\ln nL_2\Theta_1}{\varepsilon}}, \frac{\sigma^2\Theta_1 \ln n}{\varepsilon^2}\right\}\right)$ & $O\left(\max\left\{\sqrt{\frac{ n^{2}L_2\Theta_2}{\varepsilon}}, \frac{\sigma^2\Theta_2 n}{\varepsilon^2}\right\}\right)$ \\
        \hline
    \end{tabular}}
    \caption{Algorithm~\ref{Alg:ARDS} parameters for the cases $p=1$ and $p=2$.}
    \label{tab:ARDD}
\end{table}
}

Our Randomized Directional Derivative (RDD) method is listed as Algorithm \ref{Alg:RDFDS}.
\begin{algorithm}
	\caption{Randomized Directional Derivative (RDD) method}
	\label{Alg:RDFDS}
	\begin{algorithmic}[1]
		\REQUIRE $x_0$~---starting point; $N \geqslant 1$~--- number of iterations; $m \geqslant 1$~--- batch size.
		\ENSURE point $\bar{x}_N$.
		\FOR{$k=0,\, \dots, \, N-1$.}
		\STATE $\alpha \leftarrow \frac{1}{48n\rho_n L_2}$.
		\STATE Generate $e_{k+1} \in RS_2\left( 1 \right)$ independently from previous iterations and $\xi_i$, $i=1,...,m$ -- independent realizations of $\xi$. 
		\STATE Calculate 
		$$
		\tnmf(x_{k})= \frac{1}{m}\sum\limits_{i=1}^m\tf'(x_{k},\xi_i,e)e.
		$$
		\STATE $x_{k+1} \leftarrow \argmin\limits_{x\in\R^n} \left\{ {\alpha n \left\langle \tnmf(x_{k}), \, x-x_k \right\rangle +V[x_{k}] \left( x \right)}\right\}$.
		\ENDFOR
		\RETURN $\bar{x}_N \leftarrow \frac{1}{N}\sum\limits_{k=0}^{N-1}x_k$
	\end{algorithmic}
\end{algorithm}

\begin{theorem}\label{theorem_convergence_mini_gr_free_non_acc}
    Let RDD method 
    be applied to solve problem \eqref{eq:PrSt}. Then
    \begin{equation}\label{theo_main_result_mini_gr_free}
        \begin{array}{rl}
            \EE[f(\bar{x}_N)] - f(x_*) \leqslant \frac{384n\rho_nL_2\Theta_p}{N} + \frac{2}{L_2}\frac{\sigma^2}{m}+ \frac{n}{12L_2}\Delta_\zeta + \frac{4n}{3L_2}\Delta_\eta^2 + \frac{8\sqrt{2n\Theta_p}}{N}\left(\frac{\sqrt{\Delta_\zeta}}{2} + 2\Delta_\eta\right) \\
            + \frac{N}{3L_2\rho_n}\left(\frac{\sqrt{\Delta_\zeta}}{2} + 2\Delta_\eta\right)^2, 
        \end{array}
    \end{equation}
     where $\Theta_p = V[z_0](x^*)$ is defined by the chosen proximal setup and  $\EE[\cdot] = \EE_{e_1,\ldots,e_N,\xi_{1,1},\ldots,\xi_{N,m}}[\cdot]$.
\end{theorem}
Before we proceed, we give the appropriate choice of the RDD method parameters $N$, $m$, and accuracy of the directional derivative evaluation $\Delta_\zeta$, $\Delta_\eta$. These values are chosen such that the r.h.s. of \eqref{theo_main_result_mini_gr_free} is smaller than $\e$. For simplicity we omit numerical constants and summarize the obtained values of the algorithm parameters in Table 2 below. The last row represents the total number $Nm$ of oracle calls, that is, the number of directional derivative evaluations, which was advertised in \eqref{eq:RDFDSComplInformal}. \rev2{Note that the bound \eqref{theo_main_result_mini_gr_free} allows also to choose the accuracy of the directional derivative evaluation $\Delta_\zeta$, $\Delta_\eta$ decreasing with $N$. This is done by making each term with  $\Delta_\zeta$ or $\Delta_\eta$ in the r.h.s. to be of the same order as the first term.}
{
\renewcommand{\arraystretch}{2}
\begin{table}[h]
    \centering
    \begin{tabular}{|c|c|c|}
        \hline
         & $p=1$ & $p=2$ \\
        \hline
        $N$ & $O\left(\frac{ L_2\Theta_1 \ln n}{\varepsilon}\right)$ & $O\left(\frac{n L_2\Theta_2}{\varepsilon}\right)$ \\
        \hline
        $m$ & $O\left(\max\left\{1,\frac{\sigma^2}{\varepsilon L_2}\right\}\right)$ & $O\left(\max\left\{1,\frac{\sigma^2}{\varepsilon L_2}\right\}\right)$\\
        \hline
        $\Delta_\zeta$ & $O\left(\min\left\{\frac{(\ln n)^2}{n}L_2^2\Theta_1,\, \frac{\varepsilon^2}{n\Theta_1},\, \frac{\varepsilon L_2}{n}\right\}\right)$ & $O\left(\min\left\{nL_2^2\Theta_2,\, \frac{\varepsilon^2}{n\Theta_2},\, \frac{\varepsilon L_2}{n}\right\}\right)$\\
        \hline
        $\Delta_\eta$ & $O\left(\min\left\{\frac{\ln n}{\sqrt{n}}L_2\sqrt{\Theta_1},\, \frac{\varepsilon}{\sqrt{n\Theta_1}},\, \sqrt{\frac{\varepsilon L_2}{n}}\right\}\right)$ & $O\left(\min\left\{\sqrt{n}L_2\sqrt{\Theta_2},\, \frac{\varepsilon}{\sqrt{n\Theta_2}},\, \sqrt{\frac{\varepsilon L_2}{n}}\right\}\right)$\\
        \hline
        O-le calls & $O\left(\max\left\{\frac{L_2\Theta_1 \ln n}{\varepsilon}, \frac{\sigma^2\Theta_1\ln n}{\varepsilon^2}\right\}\right)$ & $O\left(\max\left\{\frac{nL_2\Theta_2}{\varepsilon}, \frac{n\sigma^2\Theta_2 }{\varepsilon^2}\right\}\right)$\\
        \hline
    \end{tabular}
    \caption{Algorithm~\ref{Alg:RDFDS} parameters for the cases $p=1$ and $p=2$.}
    \label{tab:RDD}
\end{table}
}

\subsection{Extensions for strongly convex problems}
In this subsection, we assume additionally that $f$ is $\mu_p$-strongly convex w.r.t. $p$-norm.
Our algorithms and proofs rely on the following fact. Let $x_*$ be some fixed point and $x$ be a random point such that $\EE_x \big[ \| x-x_* \|_p^2 \big] \leqslant R_p^2$, then
\begin{equation}
    \EE_x  d\left( \frac{x-x_*}{R_p} \right)  \leqslant \frac{\Omega_p}{2},
		\label{eq:expdUpBound}
\end{equation}
where $\EE_x$ denotes the expectation with respect to random vector $x$ and $\Omega_p$ is defined as follows.
For $p=1$ and our choice of the prox-function \eqref{eq:dp1}, $\Omega_p = {\rm e}n^{(\kappa-1)(2-\kappa)/\kappa}\ln n = O(\ln n)$ for our choice of $\kappa=1 + \frac{1}{\ln n}$, see \cite{nemirovsky1983problem,juditsky2014deterministic}.
For $p=2$ and our choice of the prox-function \eqref{eq:dp2}, $\Omega_p = 1$.
Our Accelerated Randomized Directional Derivative method for strongly convex problems (ARDDsc)  is listed as Algorithm~\ref{ACDS_sc}.
\begin{algorithm}
	\caption{Accelerated Randomized Directional Derivative method for strongly convex functions  (ARDDsc)}\label{ACDS_sc}
	\begin{algorithmic}[1]
		\REQUIRE $x_0$~---starting point s.t. $\| x_0 - x_* \|_p^2 \leq R_p^2$; $K \geqslant 1$~--- number of iterations; $\mu_p$ -- strong convexity parameter.
		\ENSURE point $u_K$.		
		\STATE Set 
			\begin{equation}
				\label{eq:N0Def1}
				N_0 = \left\lceil \sqrt{\frac{8 a L_2\Omega_p }{\mu_p}}\right\rceil, 
			\end{equation}
			where $a=384n^2\rho_n$.
		\FOR{$k=0,\, \dots, \, K-1$}		
		\STATE Set
					  \begin{align}
							& m_k := \max \left\{1, \left\lceil \frac{8b \sigma^2 N_0 2^{k}}{L_2\mu_p R_p^2 } \right\rceil \right\}, \quad R_k^2 := R_p^2 2^{-k} + \frac{4 \Delta}{\mu_p} \left(1-2^{-k} \right),
						\end{align}
				where $b=\frac{4}{n}$.
		\STATE Set $d_k(x) = R_k^2d\left(\frac{x-u_k}{R_k}\right)$.
		\STATE Run ARDD with starting point $u_k$ and prox-function $d_k(x)$ for $N_0$ steps with batch size $m_k$. 
		\STATE Set $u_{k+1}=y_{N_0}$, $k=k+1$.		
		\ENDFOR
		\RETURN $u_K$
	\end{algorithmic}
\end{algorithm}

\begin{theorem}
\label{Th:ACDS_sc_rate}
Let $f$ in problem \eqref{eq:PrSt} be $\mu_p$-strongly convex and ARDDsc method be applied to solve this problem. Then
    \begin{equation}\label{eq:ARDFDSSConv}
    \begin{array}{rl}
    \EE f(u_K) - f^* \leqslant 
    \frac{\mu_p  R_p^2}{2} \cdot 2^{-K} + 2 \Delta .
    \end{array}
    \end{equation}
    where 
    $\Delta = \frac{61N_0}{24L_2}\Delta_\zeta + \frac{122N_0}{3L_2}\Delta_\eta^2
        + \frac{12\sqrt{2nR_p^2\Omega_p}}{N_0^2}\left(\frac{\sqrt{\Delta_\zeta}}{2}+ 2\Delta_\eta\right)
        + \frac{N_0^2}{12n\rho_nL_2} \left(\frac{\sqrt{\Delta_\zeta}}{2} 
        + 2\Delta_\eta\right)^2$.
    Moreover, under an appropriate choice of $\Delta_\zeta$ and $\Delta_\eta$ s.t. $2 \Delta \leqslant \e/2$, the oracle complexity to achieve $\e$-accuracy of the solution is
    $$
\widetilde{O}\left(\max\left\{n^{\frac12+\frac{1}{q}}\sqrt{\frac{L_2\Omega_p }{\mu_p}}\log_2 \frac{\mu_p R_p^2 }{ \e},\frac{n^{\frac{2}{q}}\sigma^2 \Omega_p}{\mu_p \e}\right\}\right).
    $$
\end{theorem}
\rev2{Despite we have linear convergence in terms of the iterations number, the number of the oracle evaluations corresponds to sublinear convergence. The reason is that we consider general stochastic optimization problem, rather than finite-sum problems for which the linear convergence rate is achievable in terms of the oracle evaluations \cite{allen2016katyusha}.
Our oracle complexity corresponds to the lower complexity bounds \cite{nemirovsky1983problem} for general stochastic convex optimization.}

Before we proceed to the non-accelerated method, we give the appropriate choice of the accuracy of the directional derivative evaluation $\Delta_\zeta$, $\Delta_\eta$ for ARDDsc to achieve an accuracy $\e$ of the solution. These values are chosen such that the r.h.s. of \eqref{eq:ARDFDSSConv} is smaller than $\e$. For simplicity we omit numerical constants and summarize the obtained values of the algorithm parameters in Table 3 below. The last row represents the total number of oracle calls, that is, the number of directional derivative evaluations, which was stated in \eqref{eq:SCBoundsInformal}.
{
\renewcommand{\arraystretch}{2}
\begin{table}[h]
\label{tab:ACDS_sc}
    \centering
    {\scriptsize
    \begin{tabular}{|c|c|c|}
        \hline
         & $p=1$ & $p=2$ \\
        \hline
        $\Delta_\zeta$ & $O\left(\min\left\{\varepsilon\sqrt{\frac{L_2\mu_1}{n\ln n\Omega_1}},\, \varepsilon^2\frac{n(\ln n)^2 L_2^2\Omega_1}{R_1^2\mu_1^2},\, \varepsilon\cdot\frac{\mu_1}{n\Omega_1}\right\}\right)$ & $O\left(\min\left\{\varepsilon\sqrt{\frac{L_2\mu_2}{n^2\Omega_2}},\, \varepsilon^2\frac{n^3 L_2^2\Omega_2}{R_2^2\mu_2^2},\, \varepsilon\cdot\frac{\mu_2}{n\Omega_2}\right\}\right)$\\
        \hline
        $\Delta_\eta$ & $O\left(\min\left\{\sqrt{\varepsilon}\sqrt[4]{\frac{L_2\mu_1}{n\ln n\Omega_1}},\, \varepsilon\frac{\sqrt{n}\ln n L_2\sqrt{\Omega_1}}{R_1\mu_1},\, \sqrt{\varepsilon}\cdot\sqrt{\frac{\mu_1}{n\Omega_1}}\right\}\right)$ & $O\left(\min\left\{\sqrt{\varepsilon}\sqrt[4]{\frac{L_2\mu_2}{n^2\Omega_2}},\, \varepsilon\frac{\sqrt{n^3} L_2\sqrt{\Omega_2}}{R_2\mu_2},\, \sqrt{\varepsilon}\cdot\sqrt{\frac{\mu_2}{n\Omega_2}}\right\}\right)$\\
        \hline
        O-le calls & $\widetilde{O}\left(\max\left\{\sqrt{\frac{n \ln n L_2\Omega_1 }{\mu_1}}\log_2 \frac{\mu_1 R_1^2 }{ \e},\frac{\sigma^2 \Omega_1 \ln n}{\mu_1 \e}\right\}\right)$ & $\widetilde{O}\left(\max\left\{n\sqrt{\frac{L_2\Omega_2 }{\mu_2}}\log_2 \frac{\mu_2 R_2^2 }{ \e},\frac{n\sigma^2 \Omega_2}{\mu_2 \e}\right\}\right)$\\
        \hline
    \end{tabular}}
    \caption{Algorithm~\ref{ACDS_sc} parameters for the cases $p=1$ and $p=2$.}
\end{table}
}

Our Randomized Directional Derivative method for strongly convex problems (RDDsc)  is listed as Algorithm~\ref{CDS_sc}.
\begin{algorithm}
	\caption{Randomized Directional Derivative method for strongly convex functions  (RDDsc)}\label{CDS_sc}
	\begin{algorithmic}[1]
		\REQUIRE $x_0$~---starting point s.t. $\| x_0 - x_* \|_p^2 \leq R_p^2$; $K \geqslant 1$~--- number of iterations; $\mu_p$ -- strong convexity parameter.
		\ENSURE point $u_K$.		
		\STATE Set 
			\begin{equation}
				\label{eq:N0Def2}
				N_0 = \left\lceil \frac{8 a L_2\Omega_p }{\mu_p}\right\rceil, 
			\end{equation}
		where $a = 384n\rho_n$.
		\FOR{$k=0,\, \dots, \, K-1$}		
		\STATE Set
					  \begin{align}
							& m_k := \max \left\{1, \left\lceil \frac{8b \sigma^2  2^{k}}{L_2\mu_p R_p^2 } \right\rceil \right\}, \quad R_k^2 := R_p^2 2^{-k} + \frac{4 \Delta}{\mu_p} \left(1-2^{-k} \right),
						\end{align}
						where $b=2$
		\STATE Set $d_k(x) = R_k^2d\left(\frac{x-u_k}{R_k}\right)$.
		\STATE Run RDD with starting point $u_k$ and prox-function $d_k(x)$ for $N_0$ steps with batch size $m_k$. 
		\STATE Set $u_{k+1}=y_{N_0}$, $k=k+1$.		
		\ENDFOR
		\RETURN $u_K$
	\end{algorithmic}
\end{algorithm}

\begin{theorem}
\label{Th:CDS_sc_rate}
Let $f$ in problem \eqref{eq:PrSt} be $\mu_p$-strongly convex and RDDsc method be applied to solve this problem. Then
    \begin{equation}\label{eq:RDFDSSConv}
    \begin{array}{rl}
    \EE f(u_K) - f^* \leqslant 
    \frac{\mu_p  R_p^2}{2} \cdot 2^{-K} + 2 \Delta .
    \end{array}
    \end{equation}
    where $\Delta = \frac{n}{12L_2}\Delta_\zeta + \frac{4n}{3L_2}\Delta_\eta^2 + \frac{8\sqrt{2nR_p^2\Omega_p}}{N_0}\left(\frac{\sqrt{\Delta_\zeta}}{2} + 2\Delta_\eta\right) + \frac{N_0}{3L_2\rho_n}\left(\frac{\sqrt{\Delta_\zeta}}{2} + 2\Delta_\eta\right)^2$.
    Moreover, under an appropriate choice of $\Delta_\zeta$ and $\Delta_\eta$ s.t. $2 \Delta \leqslant \e/2$, the oracle complexity to achieve $\e$-accuracy of the solution is
    $$
\widetilde{O}\left(\max\left\{\frac{n^{\frac{2}{q}}L_2\Omega_p }{\mu_p}\log_2 \frac{\mu_p R_p^2 }{ \e},\frac{n^{\frac{2}{q}}\sigma^2 \Omega_p}{\mu_p \e}\right\}\right).
    $$
\end{theorem}
\rev2{Despite we have linear convergence in terms of the iterations number, the number of the oracle evaluations corresponds to sublinear convergence. The reason is that we consider general stochastic optimization problem, rather than finite-sum problems for which the linear convergence rate is achievable in terms of the oracle evaluations \cite{allen2016katyusha}.
Our oracle complexity corresponds to the lower complexity bounds \cite{nemirovsky1983problem} for general stochastic convex optimization.}

Before we proceed, we give the appropriate choice of the accuracy of the directional derivative evaluation $\Delta_\zeta$, $\Delta_\eta$ for RDDsc to achieve an accuracy $\e$ of the solution. These values are chosen such that the r.h.s. of \eqref{eq:RDFDSSConv} is smaller than $\e$. For simplicity we omit numerical constants and summarize the obtained values of the algorithm parameters in Table 4 below. The last row represents the total number of oracle calls, that is, the number of directional derivative evaluations, which was stated in \eqref{eq:SCBoundsInformal}.
{
\renewcommand{\arraystretch}{2}
\begin{table}[h]
\label{tab:CDS_sc}
    \centering
    \begin{tabular}{|c|c|c|}
        \hline
         & $p=1$ & $p=2$ \\
        \hline
        $\Delta_\zeta$ & $O\left(\min\left\{\frac{\varepsilon L_2}{n},\, \varepsilon^2\frac{(\ln n)^2L_2^2}{nR_1^2\mu_1^2},\, \varepsilon\frac{\mu_1}{n\Omega_1}\right\}\right)$ & $O\left(\min\left\{\frac{\varepsilon L_2}{n},\, \varepsilon^2\frac{nL_2^2}{R_2^2\mu_2^2},\, \varepsilon\frac{\mu_2}{n\Omega_2}\right\}\right)$ \\
        \hline
        $\Delta_\eta$ & $O\left(\min\left\{\sqrt{\frac{\varepsilon L_2}{n}},\, \varepsilon\frac{\ln nL_2}{\sqrt{n}R_1\mu_1},\, \sqrt{\varepsilon\frac{\mu_1}{n\Omega_1}}\right\}\right)$ & $O\left(\min\left\{\sqrt{\frac{\varepsilon L_2}{n}},\, \varepsilon\frac{\sqrt{n}L_2}{R_2\mu_2},\, \sqrt{\varepsilon\frac{\mu_2}{n\Omega_2}}\right\}\right)$ \\
        \hline
        O-le calls & $\widetilde{O}\left(\max\left\{\frac{L_2\Omega_1 \ln n }{\mu_1}\log_2 \frac{\mu_1 R_1^2 }{ \e},\frac{\sigma^2 \Omega_1}{\mu_1 \e}\right\}\right)$ & $\widetilde{O}\left(\max\left\{\frac{nL_2\Omega_2 }{\mu_2}\log_2 \frac{\mu_2 R_2^2 }{ \e},\frac{n\sigma^2 \Omega_2}{\mu_2 \e}\right\}\right)$\\
        \hline
    \end{tabular}
    \caption{Algorithm~\ref{CDS_sc} parameters for the cases $p=1$ and $p=2$.}
\end{table}
}

\subsection{Corollaries for derivative-free optimization}
In this subsection, following \cite{dvurechensky2018accelerated}, we consider derivative-free smooth stochastic optimization in the two-point feedback situation.
We assume that an optimization procedure, given a pair of points $(x,y) \in \R^{2n}$ , can obtain a pair of noisy stochastic realizations $(\tf(x,\xi),\tf(y,\xi))$ of the objective value $f$, where
\begin{equation}
\label{eq:tf_def_2}
\tf(x,\xi) = F(x,\xi) + \Xi(x,\xi), \quad |\Xi(x,\xi)| \leqslant  \Delta, \;\forall x \in \R^n, \; \text{a.s. in } \xi,
\end{equation}
and $\xi$ is independently drawn from $P$.

Based on these observations of the objective value, we form the following stochastic approximation of $\nabla f(x)$
\begin{equation}
\label{eq:FinDiffStocGrad}
\tnmf^t(x) = \frac{1}{m}\sum\limits_{i=1}^m\frac{\tf(x+te,\xi_i)-\tf(x,\xi_i)}{t}e 
= \left(\left\la g^m(x,\vec{\xi_{m}}),e \right\ra + \frac{1}{m}\sum\limits_{i=1}^m(\zeta(x,\xi_i,e) + \eta(x,\xi_i,e))\right)e,
\end{equation}
where $e \in RS_2(1)$, $\xi_i$, $i=1,...,m$ are independent realizations of $\xi$, $m$ is the \textit{batch size}, $t$ is some small positive parameter which we call \textit{smoothing parameter}, $g^m(x,\vec{\xi_{m}}) := \frac{1}{m}\sum\limits_{i=1}^mg(x,\xi_i)$, and
\begin{equation}
\notag
\zeta(x,\xi_i,e) = \frac{F(x+te,\xi_i) - F(x,\xi_i)}{t} - \la g(x,\xi_i),\, e\ra, \quad \eta(x,\xi_i,e)= \frac{\Xi(x+te,\xi_i) - \Xi(x,\xi_i)}{t}, \quad i=1,...,m. 
\end{equation}


By Lipschitz smoothness of $F(\cdot,\xi)$, we have $|\zeta(x,\xi,e)| \leqslant \frac{L(\xi)t}{2}$ for all $x \in \R^n$ and $e \in S_2(1)$. Hence, $\EE_{\xi}(\zeta(x,\xi,e))^2 \leqslant \frac{L_2^2t^2}{4}$ for all $x \in \R^n$ and $e \in S_2(1)$. At the same time, from \eqref{eq:tf_def_2}, we have that $|\eta(x,\xi,e)| \leqslant \frac{2\Delta}{t}$ for all $x \in \R^n$, $e \in S_2(1)$ and a.s. in $\xi$. 
Applying Theorem \ref{Th:ARDFDSConv} and Theorem \ref{theorem_convergence_mini_gr_free_non_acc} with $\Delta_{\zeta} = \frac{L_2^2t^2}{4}$ and $\Delta_{\eta} = \frac{2\Delta}{t}$, we reproduce respectively the result of Theorem 2 and Theorem 3 in \cite{dvurechensky2018accelerated}. 
Applying Theorem \ref{Th:ACDS_sc_rate} and Theorem \ref{Th:CDS_sc_rate} with $\Delta_{\zeta} = \frac{L_2^2t^2}{4}$ and $\Delta_{\eta} = \frac{2\Delta}{t}$, we obtain also complexity bounds \eqref{eq:SCBoundsInformal} for derivative-free smooth stochastic strongly convex optimization, which was not yet done in the literature. 

\section{Proof of main result for ARDD method}
\label{S:ARDDProofs}
We divide the proof of Theorem \ref{Th:ARDFDSConv} into two large steps. First, to simplify the derivations, we prove this theorem assuming two additional inequalities which connect noisy stochastic approximation of the gradient \eqref{eq:MiniBatcStocGrad} with the true gradient and function values. This result is stated as Lemma  \ref{Lm:ARDFDSLikeConv}. Then, in Lemma \ref{Lm:ARDFDSDeltaMeaning}, we show that our approximation of the gradient \eqref{eq:MiniBatcStocGrad} indeed satisfies these two inequalities.

\begin{lemma}
\label{Lm:ARDFDSLikeConv}
Let $\{x_k,y_k,z_k\}$, $k \geqslant 0$ be generated by ARDD method. Assume that there exist numbers $\delta_1>0$,$\delta_2>0$ such that, for all $k\geqslant 0$
    \begin{equation}\label{eq:AssumInPr}
        \EE\left[\left\la \tnmf(x_{k+1}),\, z_k-x_* \right\ra\right] \geqslant \frac{1}{n}\EE\left[\left\la\nabla f(x_{k+1}),\, z_k-x_* \right\ra\right] - \delta_1\EE\left[\|z_k-x_*\|\right]
    \end{equation}
    and
    \begin{equation}\label{eq:AssumQNorm}
        \EE\left[\|\tnmf(x_{k+1})\|_q^2\right] \leqslant 96\rho_nL_2\left(\EE[f(x_{k+1})]-\EE[f(y_{k+1})]\right) + \delta_2,
    \end{equation}
    where expectation is taken w.r.t. all randomness and $x^*$ is a solution to \eqref{eq:PrSt}.
Then
    \begin{equation}\label{eq:ARDFDSLikeConv}
        \begin{array}{rl}
            \EE[f(y_N)] - f(x^*) 
        \leqslant \frac{384\Theta_p n^2\rho_nL_2}{N^2} + \frac{12n\sqrt{2\Theta_p}}{N^2}\delta_1+ \frac{N}{24\rho_nL_2} \delta_2 + \frac{N^2}{12\rho_nL_2} \delta_1^2,
        \end{array}
    \end{equation}
    where $\Theta_p = V[z_0](x^*)$ is defined by the chosen proximal setup and the expectation is taken w.r.t. all randomness.
\end{lemma}
This result is proved below in subsection \ref{S:ARDD_pr_1}.

\begin{lemma}
\label{Lm:ARDFDSDeltaMeaning}
Let $\{x_k,y_k,z_k\}$, $k \geqslant 0$ be generated by ARDD method.
    Then \eqref{eq:AssumInPr} and \eqref{eq:AssumQNorm} hold with
    \begin{equation}\label{eq:ARDFDSDelta1}
        \delta_1 = \frac{\sqrt{\Delta_\zeta}}{2\sqrt{n}} + \frac{2\Delta_\eta}{\sqrt{n}}
    \end{equation}
    and
    \begin{equation}\label{eq:ARDFDSDelta2}
        \delta_2 = \frac{96\rho_n}{n}\cdot\frac{\sigma^2}{m}+ 61\rho_n\Delta_\zeta + 976\rho_n\Delta_\eta^2.
    \end{equation}
\end{lemma}
This result is proved below in subsection \ref{S:ARDD_pr_2}.

\begin{proof}[Proof of Theorem \ref{Th:ARDFDSConv}.]
    Combining Lemma~\ref{Lm:ARDFDSLikeConv} and Lemma~\ref{Lm:ARDFDSDeltaMeaning}, we obtain \eqref{eq:ARDFDSConv}.
\end{proof}

\subsection{Proof Lemma \ref{Lm:ARDFDSLikeConv}}
\label{S:ARDD_pr_1}

The following lemma estimates the progress in step 8 of ARDD method (and in step 5 of RDD method), which is a Mirror Descent step.
\begin{lemma}
\label{Lm:MDStep}
    Assume that $z_+ = \argmin\limits_{v\in\R^n} \left\{ {\alpha n \left\langle \tnmf(x), \, v-z \right\rangle +V[z] \left( v \right)}\right\}$.
		Then, for any fixed $u \in \R^n$,
		\begin{equation}		
		\label{eq:MDStepProgr}
		\begin{array}{rl}
		\alpha n\EE\left[\langle \tnmf(x), \, z-u\rangle\right]  \leqslant \frac{\alpha^2n^2}{2}\EE\left[\| \tnmf(x)\|_q^2\right] + \EE\left[V[z](u)\right] - \EE\left[V[z_{+}](u)\right],
		\end{array}
		\end{equation}	
		where expectation is taken w.r.t. all randomness.
\end{lemma}
\begin{proof}
For all $u\in \R^n$, we have
\begin{equation}	
\label{lemma110:basic_estimations}
	\begin{array}{rl}
		\alpha n\langle \tnmf(x), \, z-u\rangle
		=  \alpha n \langle \tnmf(x), \,  z-z_{+}\rangle + \alpha n \langle \tnmf(x), \, z_{+}-u\rangle\\
		\overset{\circledOne}{\leqslant}  \alpha n \langle \tnmf(x), \, z-z_{+}\rangle + \langle -\nabla V[z](z_{+}), \,  z_{+}-u\rangle
		\overset{\circledTwo}{=}  \alpha n \langle \tnmf(x), \, z-z_{+}\rangle\\
		+ V[z](u) - V[z_{+}](u) - V[z](z_{+})
		\overset{\circledThree}{\leqslant} \left(\alpha n \langle \tnmf(x), \,  z-z_{+}\rangle - \frac{1}{2}\|z-z_{+}\|_p^2\right)\\ + V[z](u) - V[z_+](u)
		\overset{\circledFour}{\leqslant} \frac{\alpha^2n^2}{2}\| \tnmf(x)\|_q^2 + V[z](u) - V[z_{+}](u),
	\end{array}
\end{equation}
where $\circledOne$ follows from the definition of $z_+$, whence $\langle \nabla V[z](z_{+}) +  \alpha n \tnmf^t(x), \, u - z_{+}\rangle \geqslant 0$ for all $u\in \R^n$; $\circledTwo$ follows from the "`magic identity"' Fact 5.3.3 in \cite{ben-tal2015lectures} for the Bregman divergence; $\circledThree$ follows from \eqref{eq:VStrConv}; and $\circledFour$ follows from the Fenchel inequality $\zeta\la s,z \ra - \frac{1}{2}\|z\|_p^2 \leq \frac{\zeta^2}{2}\|s\|_q^2$. Taking full expectation we get \eqref{eq:MDStepProgr}.
\end{proof}

Now we prove the following lemma which estimates the one-iteration progress of the whole algorithm.
\begin{lemma}
\label{Lm:OneStep}
		Let $\{x_k,y_k,z_k,\alpha_k,\tau_k\}$, $k \geqslant 0$ be generated by ARDD method. Then, under assumptions of Lemma \ref{Lm:ARDFDSLikeConv},
		\begin{equation}
		\label{eq:OneStep}
        \begin{array}{rl}
            48 n^2 \rho_n L_2 \alpha_{k+1}^2\EE[f(y_{k+1})] - (48 n^2 \rho_n L_2 \alpha_{k+1}^2-\alpha_{k+1})\EE\left[f(y_k)\right]\\ - \EE\left[V[z_k](x_*)\right] + \EE[V[z_{k+1}](x_*)] - \alpha_{k+1}\delta_1n\EE\left[\|z_k-x_*\|_p\right] - \frac{\alpha_{k+1}^2n^2}{2}\delta_2  \leqslant \alpha_{k+1}f(x_*),
        \end{array}
    \end{equation}
	where expectation is taken w.r.t. all randomness, $x^*$ is a solution to \eqref{eq:PrSt}.
\end{lemma}
\begin{proof}
Combining \eqref{eq:AssumInPr}, \eqref{eq:AssumQNorm} and \eqref{eq:MDStepProgr}, we obtain
\begin{equation}		
		\label{eq:OneStepPr2}
		\begin{array}{rl}
		\alpha_{k+1} \EE\left[\la \nabla f(x_{k+1}),z_k-x_*\ra\right] \leqslant 48\alpha^2n^2 \rho_n L_2 \left(\EE\left[f(x_{k+1})\right] - \EE\left[f(y_{k+1})\right]\right) \\
		+ \EE\left[V_{z_k}(x_*)\right] - \EE[V[z_{k+1}](x_*)]
		+ \alpha_{k+1}\delta_1n\EE\left[\|z_k-x_*\|_p\right] + \frac{\alpha_{k+1}^2n^2}{2}\delta_2.
		\end{array}
\end{equation}
Further,
\begin{equation*}
    \begin{array}{rl}
        \alpha_{k+1}\left(\EE\left[f(x_{k+1})\right] - f(x_*)\right) \leqslant \alpha_{k+1}\EE\left[\langle \nabla f(x_{k+1}) ,\, x_{k+1} - x_* \rangle\right]\\
        = \alpha_{k+1}\EE\left[\langle\nabla f(x_{k+1}), \, x_{k+1}-z_k\rangle\right] + \alpha_{k+1}\EE\left[\langle\nabla f(x_{k+1}), \, z_{k}-x_*\rangle\right]\\
	    \overset{\circledOne}{=} \frac{(1-\tau_k)\alpha_{k+1}}{\tau_k}\EE\left[\langle\nabla f(x_{k+1}), \, y_{k}-x_{k+1}\rangle\right]+\alpha_{k+1}\EE\left[\langle\nabla f(x_{k+1}), \, z_{k}-x_*\rangle\right]\\
		\overset{\circledTwo}{\leqslant} \frac{(1-\tau_k)\alpha_{k+1}}{\tau_k} \left(\EE\left[f(y_k)\right]-\EE\left[f(x_{k+1})\right]\right) + \alpha_{k+1}\EE\left[\langle\nabla f(x_{k+1}), \, z_{k}-x_*\rangle\right] \\
		\overset{\eqref{eq:OneStepPr2}}{\leqslant} \frac{(1-\tau_k)\alpha_{k+1}}{\tau_k} \left(\EE\left[f(y_k)\right]-\EE\left[f(x_{k+1})\right]\right) +  48\alpha^2n^2 \rho_n L_2 \left(\EE\left[f(x_{k+1})\right] - \EE\left[f(y_{k+1})\right]\right) \\
		+ \EE\left[V_{z_k}(x_*)\right] - \EE[V[z_{k+1}](x_*)]
		+ \alpha_{k+1}\delta_1n\EE\left[\|z_k-x_*\|_p\right] + \frac{\alpha_{k+1}^2n^2}{2}\delta_2 \\
		\overset{\circledThree}{=} (48 \alpha_{k+1}^2 n^2 \rho_n L_2-\alpha_{k+1})\EE\left[f(y_k)\right] - 48 \alpha_{k+1}^2 n^2 \rho_n L_2\EE[f(y_{k+1})]\\
		+\alpha_{k+1}\EE[f(x_{k+1})] + \EE\left[V_{z_k}(x_*)\right] - \EE[V[z_{k+1}](x_*)]
		+ \alpha_{k+1}\delta_1n\EE\left[\|z_k-x_*\|_p\right] + \frac{\alpha_{k+1}^2n^2}{2}\delta_2.
    \end{array}
\end{equation*}
Here $\circledOne$ is since $x_{k+1} := \tau_kz_k+(1-\tau_k)y_k\; \Leftrightarrow\; \tau_k(x_{k+1}-z_k) = (1-\tau_k)(y_k - x_{k+1})$, $\circledTwo$ follows from the convexity of $f$ and the inequality $1-\tau_k\geqslant0$ and $\circledThree$ is since $\tau_k = \frac{1}{48\alpha_{k+1}n^2\rho_n L_2}$.
Rearranging the terms, we obtain the statement of the lemma.
\end{proof}

We are now ready to finish the proof of Lemma \ref{Lm:ARDFDSLikeConv}.
\begin{proof}[Proof of Lemma \ref{Lm:ARDFDSLikeConv}.]

Note that $48n^2\rho_nL_2\alpha_{k+1}^2 - \alpha_{k+1} + \frac{1}{192n^2\rho_nL_2} = 48n^2\rho_nL_2\alpha_{k}^2$. That is,
\begin{equation*}
    \begin{array}{rl}
        48n^2\rho_nL_2\alpha_{k+1}^2 - \alpha_{k+1} + \frac{1}{192n^2\rho_nL_2} = \frac{(k+2)^2}{192n^2\rho_nL_2} - \frac{k+2}{96n^2\rho_nL_2} + \frac{1}{192n^2\rho_nL_2}\\
        = \frac{k^2+4k+4-2k-4+1}{192n^2\rho_nL_2} = \frac{(k+1)^2}{192n^2\rho_nL_2} = 48n^2\rho_nL_2\alpha_{k}^2.
    \end{array}
\end{equation*}
Telescoping \eqref{eq:OneStep} for $k=0,1,2,\ldots,l-1$ for $l\leqslant N$ we have\footnote{Note that $\alpha_1 = \frac{2}{96n^2\rho_nL_2} = \frac{1}{48n^2\rho_nL_2}$ and therefore $48n^2\rho_nL_2\alpha_1^2 - \alpha_1 = 0$.}
\begin{equation}
\label{theo_after_telescoping_mini_gr_free}
    \begin{array}{rl}
        48n^2\rho_nL_2\alpha_{l}^2\EE[f(y_l)] + \sum\limits_{k=1}^{l-1}\frac{1}{192n^2\rho_nL_2}\EE[f(y_k)] - V[z_0](x_*) + \EE[V[z_l](x_*)]  \\
				- \zeta_1\sum\limits_{k=0}^{l-1}\alpha_{k+1}\EE[\|u-z_k\|_p] -\zeta_2\sum\limits_{k=0}^{l-1}\alpha_{k+1}^2   \leqslant \sum\limits_{k=0}^{l-1}\alpha_{k+1}f(u),
    \end{array}
\end{equation}
where we denoted 
\begin{equation}
\label{eq:zeta1zeta2Def}
\zeta_1:= \delta_1n, \quad \zeta_2 := \frac{n^2}{2}\delta_2.
\end{equation}
We define $\Theta := V[z_0](x^*)$, $R_k:=\EE[\|x^*-z_k\|_p]$. Also, from \eqref{eq:VStrConv}, we have that $\zeta_1\alpha_1R_0 \leq \frac{\sqrt{2\Theta}\zeta_1}{48n^2\rho_nL_2}$. To simplify the notation, we define $B_l := \zeta_2\sum\limits_{k=0}^{l-1}\alpha_{k+1}^2 + \Theta + \frac{\sqrt{2\Theta}\zeta_1}{48n^2\rho_nL_2}$ . Since $\sum\limits_{k=0}^{l-1}\alpha_{k+1} = \frac{l(l+3)}{192n^2\rho_nL_2}$ and, for all $i=1,\ldots,N$, $f(y_i) \leqslant f(x^*)$, we obtain from \eqref{theo_after_telescoping_mini_gr_free}
\begin{equation}
\label{theo_basic_inequality_mini_gr_free}
    \begin{array}{rl}
        \frac{(l+1)^2}{192n^2\rho_nL_2}\EE[f(y_l)] \leqslant f(x^*)\left(\frac{(l+3)l}{192n^2\rho_nL_2} - \frac{l-1}{192n^2\rho_nL_2}\right)
        + B_l - \EE[V[z_l](x^*)] + \zeta_1\sum\limits_{k=1}^{l-1}\alpha_{k+1}R_k,\\
        0 \leqslant \frac{(l+1)^2}{192n^2\rho_nL_2}\left(\EE[f(y_l)] - f(x^*)\right) \leqslant B_l - \EE[V[z_l](x^*)] + \zeta_1\sum\limits_{k=1}^{l-1}\alpha_{k+1}R_k,
    \end{array}
\end{equation}
which gives 
\begin{equation}\label{theo_bregman_divergence_estim_2_mini_gr_free}
    \begin{array}{rl}
        \EE[V[z_l](x^*)] \leqslant B_l + \zeta_1\sum\limits_{k=1}^{l-1}\alpha_{k+1}R_k.
    \end{array}
\end{equation}
Moreover,
\begin{equation}\label{theo_distance_estimation_mini_gr_free}
    \begin{array}{rl}
        \frac{1}{2}\left(\EE[\|z_l-x^*\|_p]\right)^2 \leqslant \frac{1}{2}\EE[\|z_l-x^*\|_p^2] \leqslant \EE[V[z_l](x^*)] \overset{\eqref{theo_bregman_divergence_estim_2_mini_gr_free}}{\leqslant} B_l + \zeta_1\sum\limits_{k=1}^{l-1}\alpha_{k+1}R_k,
    \end{array}
\end{equation}
whence,
\begin{equation}\label{theo_distance_estimation_2_mini_gr_free}
    \begin{array}{rl}
        R_l \leqslant \sqrt{2}\cdot\sqrt{B_l + \zeta_1\sum\limits_{k=1}^{l-1}\alpha_{k+1}R_k}.
    \end{array}
\end{equation}
Applying Lemma~\ref{stoh:technical_lemma} for $a_0 = \zeta_2\alpha_1^2 + \Theta + \frac{\sqrt{2\Theta}\zeta_1}{48n^2\rho_nL_2}, a_{k} = \zeta_2\alpha_{k+1}^2, b = \zeta_1$ for $k=1,\ldots,N-1$, we obtain
\begin{equation}\label{theo_inequality_for_induction_mini_gr_free}
    \begin{array}{rl}
        B_l + \zeta_1\sum\limits_{k=1}^{l-1}\alpha_{k+1}R_k \leqslant \left(\sqrt{B_l} + \sqrt{2}\zeta_1\cdot\frac{l^2}{96n^2\rho_nL_2}\right)^2, \; l=1,\ldots,N
    \end{array}
\end{equation} 
Since $V[z](x^*) \geqslant 0$, by inequality \eqref{theo_basic_inequality_mini_gr_free} for $l=N$ and the definition of $B_l$, 
we have
\begin{equation}\label{theo_pre_final_mini_gr_free}
    \begin{array}{rl}
        \frac{(N+1)^2}{192n^2\rho_nL_2}\left(\EE[f(y_N)] - f(x^*)\right) \leqslant \left(\sqrt{B_N} + \sqrt{2}\zeta_1\cdot\frac{N^2}{96n^2\rho_nL_2}\right)^2
        \overset{\circledOne}{\leqslant} 2B_N + 4\zeta_1^2\cdot\frac{N^4}{(96n^2\rho_nL_2)^2}\\
        = 2\zeta_2\sum\limits_{k=0}^{l-1}\alpha_{k+1}^2 + 2\Theta + \frac{\sqrt{2\Theta}\zeta_1}{24n^2\rho_nL_2} + 4\zeta_1^2\cdot\frac{N^4}{(96n^2\rho_nL_2)^2}\\
        \overset{\circledTwo}{\leqslant} 2\Theta + \frac{\sqrt{2\Theta}\zeta_1}{24n^2\rho_nL_2} + \frac{2\zeta_2(N+1)^3}{(96n^2\rho_nL_2)^2}  + 4\zeta_1^2\cdot \frac{N^4}{(96n^2\rho_nL_2)^2}
    \end{array}
\end{equation}
where $\circledOne$ is due to the fact that $\forall a,b\in\R\quad (a+b)^2 \leqslant 2a^2+2b^2$ and $\circledTwo$ is because $\sum\limits_{k=0}^{N-1}\alpha_{k+1}^2 = \frac{1}{(96n^2\rho_nL_2)^2}\sum\limits_{k=2}^{N+1}k^2 \leqslant \frac{1}{(96n^2\rho_nL_2)^2} \cdot \frac{(N+1)(N+2)(2N+3)}{6} \leqslant \frac{1}{(96n^2\rho_nL_2)^2} \cdot \frac{(N+1)2(N+1)3(N+1)}{6} = \frac{(N+1)^3}{(96n^2\rho_nL_2)^2}$. Dividing \eqref{theo_pre_final_mini_gr_free} by $\frac{(N+1)^2}{192n^2\rho_nL_2}$ and substituting $\zeta_1,\zeta_2$ from \eqref{eq:zeta1zeta2Def}, we obtain
\begin{equation*}
    \begin{array}{rl}
        \EE[f(y_N)] - f(x^*) \leqslant \frac{384\Theta n^2\rho_nL_2}{(N+1)^2} + \frac{12\sqrt{2\Theta}}{(N+1)^2}\zeta_1 + \frac{(N+1)\zeta_2}{24n^2\rho_nL_2} + \frac{N^4\zeta_1^2}{12n^2\rho_nL_2(N+1)^2}\\
        \leqslant \frac{384\Theta n^2\rho_nL_2}{N^2} + \frac{12n\sqrt{2\Theta}}{N^2}\delta_1+ \frac{N}{24\rho_nL_2} \delta_2 + \frac{N^2}{12\rho_nL_2} \delta_1^2.
    \end{array}
\end{equation*}
\end{proof}

\subsection{Proof Lemma \ref{Lm:ARDFDSDeltaMeaning}}
\label{S:ARDD_pr_2}
We start with the following technical result which connects our noisy  approximation \eqref{eq:MiniBatcStocGrad} of the stochastic gradient with the stochastic gradient itself and also with $\nabla f$.
\begin{lemma}
\label{Lm:MinBatcToGrad}
    For all $x, s \in\R^n$, we have 
    \begin{equation}
		\label{eq:tnfmq2}
			\begin{array}{rl}
			\EE_{e}\| \tnmf(x)\|_q^2 \leqslant \frac{12\rho_n}{n}\|g^m(x,\vec{\xi_{m}})\|_2^2+\frac{\rho_n}{m}\sum\limits_{i=1}^m\zeta(x,\xi_i)^2+16\rho_n \Delta_\eta^2,
			\end{array}
		\end{equation}
		\begin{equation}
		\label{eq:tnfm22}
		\EE_{e}\|\tnmf(x)\|_2^2 \geqslant \frac{1}{2n}\|g^m(x,\vec{\xi_{m}})\|_2^2 - \frac{1}{2m}\sum\limits_{i=1}^m\zeta(x,\xi_i)^2- 8\Delta_\eta^2,
		\end{equation}		
		\begin{equation}
		\label{eq:tnfms}
			\begin{array}{rl}
				\EE_{e} \langle \tnmf(x), s\rangle \geqslant \frac{1}{n}\la g^m(x,\vec{\xi_{m}}),s\ra - \frac{\|s\|_p}{2m\sqrt{n}}\sum\limits_{i=1}^m |\zeta(x,\xi_i)| - \frac{2\Delta_\eta \|s\|_p}{\sqrt{n}},
			\end{array}
		\end{equation}
		\begin{equation}
		\label{eq:nf-tnfm22}
			\begin{array}{rl}
				\EE_{e} \|\langle \nabla f(x),e \ra e - \tnmf(x)\|_2^2 \leqslant \frac{2}{n}\|\nabla f(x) -  g^m(x,\vec{\xi_{m}})\|_2^2 + \frac{1}{m}\sum\limits_{i=1}^m\zeta(x,\xi_i)^2+ 16\Delta_\eta^2,
			\end{array}
		\end{equation}
		where $g^m(x,\vec{\xi_{m}}) := \frac{1}{m}\sum\limits_{i=1}^mg(x,\xi_i)$, $\zeta(x,\xi_i)$ and $\Delta_\eta$ are defined in \eqref{eq:tf_def}.
\end{lemma}
\begin{proof}
First of all, we rewrite $\tnmf(x)$ as follows
\begin{equation}
\notag
\tnmf(x) = \left(\left\la g^m(x,\vec{\xi_{m}}),e \right\ra + \frac{1}{m}\sum\limits_{i=1}^m\theta(x,\xi_i,e)\right)e,
\end{equation}
where
\begin{equation}
\notag
\theta(x,\xi_i,e) = \zeta(x,\xi_i) + \eta(x,\xi_i,e), \quad i=1,...,m. 
\end{equation}
By \eqref{eq:tf_def}, we have
\begin{equation}
\label{eq:thetaEst}
|\theta(x,\xi_i,e)| \leq |\zeta(x,\xi_i)| + \Delta_\eta.
\end{equation}

\textbf{Proof of \eqref{eq:tnfmq2}}.
\begin{equation}
	\begin{array}{rl}
		\EE_{e}\| \tnmf(x)\|_q^2
		= \EE_{e}\Big\|\left(\left\la g^m(x,\vec{\xi_{m}}),e \right\ra + \frac{1}{m}\sum\limits_{i=1}^m\theta(x,\xi_i,e)\right)e\Big\|_q^2 \\
		\overset{\circledOne}{\leqslant} 2\EE_{e}\|\la g^m(x,\vec{\xi_{m}}),e \ra e\|_q^2 + 2\EE_{e}\left\|\frac{1}{m}\sum\limits_{i=1}^m\theta(x,\xi_i,e)e\right\|_q^2 \\
		\overset{\circledTwo}{\leqslant} \frac{12\rho_n}{n}\|g^m(x,\vec{\xi_{m}})\|_2^2+\frac{2\rho_n}{m}\sum\limits_{i=1}^m\left(|\zeta(x,\xi_i)| + \Delta_\eta\right)^2 \leqslant \frac{12\rho_n}{n}\|g^m(x,\vec{\xi_{m}})\|_2^2+\frac{\rho_n}{m}\sum\limits_{i=1}^m\zeta(x,\xi_i)^2+16\rho_n \Delta_\eta^2,
	\end{array}
\end{equation}
where $\circledOne$ holds since $\|x+y\|_q^2 \leqslant 2\|x\|_q^2 + 2\|y\|_q^2, \forall x,y\in\R^n$; $\circledTwo$ follows from inequalities \eqref{assumption_e_norm},\eqref{assumption_e_product}, \eqref{eq:thetaEst} and the fact that, for any $a_1,a_2,\ldots,a_m > 0$, it holds that $\left(\sum\limits_{i=1}^ma_i\right)^2 \leqslant m\sum\limits_{i=1}^ma_i^2$.

\textbf{Proof of \eqref{eq:tnfm22}}.
\begin{equation}
	\begin{array}{rl}
		\EE_{e}\| \tnmf(x)\|_2^2
		= \EE_{e}\Big\|\left(\left\la g^m(x,\vec{\xi_{m}}),e \right\ra + \frac{1}{m}\sum\limits_{i=1}^m\theta(x,\xi_i,e)\right)e\Big\|_2^2 \\
		\overset{\circledOne}{\geqslant} \frac{1}{2}\EE_{e}\|\la g^m(x,\vec{\xi_{m}}),e \ra e\|_2^2 - \frac{1}{m}\sum\limits_{i=1}^m\left(|\zeta(x,\xi_i)| + \Delta_\eta\right)^2 
		\overset{\circledTwo}{\geqslant} \frac{1}{2n}\|g^m(x,\vec{\xi_{m}})\|_2^2 - \frac{1}{2m}\sum\limits_{i=1}^m\zeta(x,\xi_i)^2-8 \Delta_\eta^2,
	\end{array}
\end{equation}
where $\circledOne$ follows from \eqref{eq:thetaEst} and inequality $\|x+y\|_2^2 \geqslant \frac{1}{2}\|x\|_2^2 - \|y\|_2^2, \forall x,y\in\R^n$; $\circledTwo$ follows from $e \in S_2(1)$ and Lemma~B.10 in \cite{bogolubsky2016learning}, stating that, for any $s \in \R^n$, $\EE\la s,e\ra^2 = \frac{1}{n} \|s\|_2^2$.

\textbf{Proof of \eqref{eq:tnfms}}.
\begin{equation}
	\begin{array}{rl}
		\EE_{e} \langle \tnmf(x), s\rangle  = \EE_{e} \langle \la g^m(x,\vec{\xi_{m}}),e\ra e, s\rangle + \EE_{e}  \frac{1}{m}\sum\limits_{i=1}^m\theta(x,\xi_i,e) \la e, s \ra \\
		\overset{\circledOne}{\geqslant} \frac{1}{n}\la g^m(x,\vec{\xi_{m}}),s\ra - \frac{1}{m}\sum\limits_{i=1}^m\left(|\zeta(x,\xi_i)| + \Delta_\eta\right) \EE_{e}  |\la e,s \ra | \\
		\overset{\circledTwo}{\geqslant} \frac{1}{n}\la g^m(x,\vec{\xi_{m}}),s\ra - \frac{\|s\|_p}{2m\sqrt{n}}\sum\limits_{i=1}^m |\zeta(x,\xi_i)| - \frac{2\Delta_\eta \|s\|_p}{\sqrt{n}}
	\end{array}
\end{equation}
where $\circledOne$ follows from $\EE_{e}[n\la g, e \ra e] = g,\, \forall g\in\R^n$ and \eqref{eq:thetaEst}; 
$\circledTwo$ follows from Lemma~B.10 in \cite{bogolubsky2016learning}, since $\EE|\la s,e\ra| \leq \sqrt{\EE\la s,e\ra^2}$, and the fact that $\|x\|_2 \leqslant \|x\|_p$ for $p\leqslant2$.

\textbf{Proof of \eqref{eq:nf-tnfm22}}.
\begin{equation}
	\begin{array}{rl}
		\EE_{e} \|\langle \nabla f(x),e \ra e - \tnmf(x)\|_2^2 = \EE_{e} \left\|\la \nabla f(x),e \ra e - \la g^m(x,\vec{\xi_{m}}),e \ra e - \frac{1}{m}\sum\limits_{i=1}^m\theta(x,\xi_i,e)  e\right\|_2^2 \\
		\overset{\circledOne}{\leqslant}  2\EE_{e} \left\|\la \nabla f(x) -  g^m(x,\vec{\xi_{m}}),e \ra e \right\|_2^2 + 
		2\EE_{e} \left\| \frac{1}{m}\sum\limits_{i=1}^m\theta(x,\xi_i,e)  e\right\|_2^2  \\
		\overset{\circledTwo}{\leqslant} \frac{2}{n}\|\nabla f(x) -  g^m(x,\vec{\xi_{m}})\|_2^2 + \frac{1}{m}\sum\limits_{i=1}^m\zeta(x,\xi_i)^2+16 \Delta_\eta^2,
	\end{array}
\end{equation}
where $\circledOne$ holds since $\|x+y\|_2^2 \leqslant 2\|x\|_2^2 + 2\|y\|_2^2, \forall x,y\in\R^n$; 
$\circledTwo$ follows from $e \in S_2(1)$ and Lemma~B.10 in \cite{bogolubsky2016learning}, and \eqref{eq:thetaEst}.

\end{proof}

We continue by proving the following lemma which estimates the progress in step 7 of ARDD, which is a gradient step.
\begin{lemma}
\label{Lm:GradStep}
    Assume that $y = x - \frac{1}{2L_2}\tnmf(x)$.
		Then, 
		\begin{equation}		
		\label{eq:GradStepProgr}
		\begin{array}{rl}
		    \| g^m(x,\vec{\xi_{m}})\|_2^2 \leq 8nL_2 (f(x) - \EE_ef(y)) + 8 \|\nabla f(x) -  g^m(x,\vec{\xi_{m}})\|_2^2\\ + \frac{5n}{m}\sum\limits_{i=1}^m\zeta(x,\xi_i)^2+ 80n \Delta_\eta^2,
		\end{array}
		\end{equation}	
		where $g^m(x,\vec{\xi_{m}})$ is defined in Lemma \ref{Lm:MinBatcToGrad}, $\zeta(x,\xi_i)$ and $\Delta_\eta$ are defined in \eqref{eq:tf_def}.
\end{lemma}
\begin{proof}
Since $\tnmf(x)$ is collinear to $e$, we have that, for some $\gamma \in \R$, $y-x = \gamma e$. Then, since $\|e\|_2=1$,
\begin{equation*}
    \begin{array}{rl}
	    \langle \nabla f(x), y-x \rangle = \langle \nabla f(x), e \rangle \gamma = \langle \nabla f(x), e \rangle \langle e, y-x \rangle = \langle \langle \nabla f(x),e \rangle e, \, y-x\rangle.
	\end{array}
\end{equation*}
From this and $L_2$-smoothness of $f$ we obtain
\begin{equation*}
	\begin{array}{rl}
	    f(y) \leqslant f(x) + \langle \langle \nabla f(x),e \rangle e, \, y-x\rangle + \frac{L_2}{2}||y-x||_2^2 \\
	    \leqslant f(x) + \langle  \tnmf(x), \, y-x\rangle + L_2||y-x||_2^2  + \langle \langle \nabla f(x),e \rangle e - \tnmf(x), \, y-x\rangle - \frac{L_2}{2}||y-x||_2^2 \\
        \overset{\circledOne}{\leqslant} f(x) + \langle \tnmf(x), \, y-x\rangle + L_2||y-x||_2^2  + \frac{1}{2L_2}\| \langle \nabla f(x),e \rangle e - \tnmf(x)\|_2^2, 
	\end{array}
\end{equation*}
where $\circledOne$ follows form the Fenchel inequality $\la s,z \ra - \frac{\zeta}{2}\|z\|_2^2 \leq \frac{1}{2\zeta}\|s\|_2^2$.
Using $y = x-\frac{1}{2L_2}\tnmf(x)$, we get
\begin{equation*}
    \begin{array}{rl}
        \frac{1}{4L_2} \|\tnmf(x)\|_2^2 \leqslant f(x)-f(y) + \frac{1}{2L_2}\|\langle \nabla f(x),e \rangle e - \tnmf(x)\|_2^2
    \end{array}
\end{equation*}
Taking the expectation in $e$ and applying \eqref{eq:tnfm22}, \eqref{eq:nf-tnfm22}, we obtain
\begin{equation*}
    \begin{array}{rl}
    \frac{1}{4L_2} \left(\frac{1}{2n}\|g^m(x,\vec{\xi_{m}})\|_2^2 - \frac{1}{2m}\sum\limits_{i=1}^m\zeta(x,\xi_i)^2-8 \Delta_\eta^2\right)   \leqslant \frac{1}{4L_2} \EE_e\|\tnmf(x)\|_2^2 \\
		\leqslant f(x)-\EE_ef(y) + \frac{1}{2L_2}\EE_e\|\langle \nabla f(x),e \rangle e - \tnmf(x)\|_2^2 \\
		 \leqslant f(x)-\EE_ef(y) + \frac{1}{2L_2} \left( \frac{2}{n}\|\nabla f(x) -  g^m(x,\vec{\xi_{m}})\|_2^2 + \frac{t^2}{m}\sum\limits_{i=1}^m\zeta(x,\xi_i)^2+16 \Delta_\eta^2\right),
    \end{array}
\end{equation*}
Rearranging the terms, we obtain the statement of the lemma.
\end{proof}

We are now ready to finish the proof of Lemma \ref{Lm:ARDFDSDeltaMeaning}.
\begin{proof}[Proof of Lemma \ref{Lm:ARDFDSDeltaMeaning}.]
    Taking the expectation w.r.t. all randomness\footnote{Note that we use $s = z_k-x_*$ which does not depend on $\xi_1,\xi_2,\ldots,\xi_m$ from the $(k+1)$-th iterate and it does not depend on $e_{k+1}$. Therefore we can use tower property of mathematical expectation and take firstly conditional expectation w.r.t. $\xi_1,\ldots,\xi_m$ and after that take full expectation.} of \eqref{eq:tnfms} and using inequality $$\EE[|\zeta(x,\xi_i)|] \leqslant \sqrt{\EE[|\zeta(x,\xi_i)|^2]} \overset{\eqref{eq:tf_def}}{\leqslant} \sqrt{\Delta_\zeta},$$ we obtain inequality \eqref{eq:AssumInPr} with $\delta_1 = \frac{\sqrt{\Delta_\zeta}}{2\sqrt{n}} + \frac{2\Delta_\eta}{\sqrt{n}}$. Combining \eqref{eq:tnfmq2} and \eqref{eq:GradStepProgr}, taking the full  expectation and using $\EE[\|\nabla f(x) - g^m(x,\xi)\|_2^2] \leqslant\frac{\sigma^2}{m}$, which follows from \eqref{stoch_assumption_on_variance}, we obtain \eqref{eq:AssumQNorm} with $\delta_2 = \frac{96\rho_n}{n}\cdot\frac{\sigma^2}{m}+ 61\rho_n\Delta_\zeta + 976\rho_n\Delta_\eta^2$.
\end{proof}

\section{Proof of main result for RDD method}
\label{sec:rdfds}
As in the previous section, we divide the proof of Theorem \ref{theorem_convergence_mini_gr_free_non_acc} into large steps. First, to simplify the derivations, we prove this theorem assuming two additional inequalities which connect or noisy stochastic approximation of the gradient \eqref{eq:MiniBatcStocGrad} with the true gradient and function values. Then we show that our approximation of the gradient \eqref{eq:MiniBatcStocGrad} indeed satisfies these two inequalities.

\begin{lemma}
\label{Lm:RDSLikeConv}
    Let $\{x_k,y_k,z_k\}$, $k \geqslant 0$ be generated by RDD method. Assume that there exist numbers $\delta_1>0$,$\delta_2>0$ such that, for all $k\geqslant 0$
    \begin{equation}\label{eq:AssumDelta1_nonAcc}
        \EE\left[\left\la\tnmf(x_{k}), x_k-x_*\right\ra\right] \geqslant\frac{1}{n}\EE\left[\left\la\nabla f(x_k),\, x_k-x_* \right\ra\right] - \delta_1\EE\left[\|x_k-x_*\|_p\right]
    \end{equation}
    \begin{equation}\label{eq:AssumDelta2_nonAcc}
        \EE\left[\|\tnmf(x_k)\|_q^2\right] \leqslant \frac{48\rho_nL_2}{n}\left(\EE\left[f(x_{k})\right] - f(x_*)\right) + \delta_2,
    \end{equation}
    where expectation is taken w.r.t. all randomness and $x_*$ is a solution to \eqref{eq:PrSt}.
    Then
    \begin{equation}\label{eq:RDSLikeConv}
        \begin{array}{rl}
        \EE[f(\bar{x}_N)] - f(x_*) \leqslant \frac{384n\rho_nL_2\Theta_p}{N} + \frac{n}{12\rho_nL_2}\delta_2 + \frac{8n\sqrt{2\Theta_p}}{N}\delta_1 + \frac{nN}{3L_2\rho_n}\delta_1^2,
        \end{array}
    \end{equation}
    where $\Theta_p = V[z_0](x^*)$ is defined by the chosen proximal setup and the expectation is taken w.r.t. all randomness.
\end{lemma}
This result is proved below in subsection \ref{S:RDD_pr_1}.

\begin{lemma}
\label{Lm:RDFDSDeltaMeaning}
    Let $\{x_k,y_k,z_k\}$, $k \geqslant 0$ be generated by RDD method.
    Then \eqref{eq:AssumDelta1_nonAcc} and \eqref{eq:AssumDelta2_nonAcc} hold with
    \begin{equation}\label{eq:RDFDSDelta1}
        \delta_1 = \frac{\sqrt{\Delta_\zeta}}{2\sqrt{n}} + \frac{2\Delta_\eta}{\sqrt{n}}
    \end{equation}
    and
    \begin{equation}\label{eq:RDFDSDelta2}
        \delta_2 = \frac{24\rho_n}{n}\cdot\frac{\sigma^2}{m}+ \rho_n\Delta_\zeta + 16\rho_n\Delta_\eta^2.
    \end{equation}
\end{lemma}
This result is proved below in subsection \ref{S:RDD_pr_2}.

\begin{proof}[Proof of Theorem \ref{theorem_convergence_mini_gr_free_non_acc}.]
    Combining Lemma~\ref{Lm:RDSLikeConv} and Lemma~\ref{Lm:RDFDSDeltaMeaning}, we obtain \eqref{theo_main_result_mini_gr_free}.
\end{proof}

\subsection{Proof Lemma \ref{Lm:RDSLikeConv}}
\label{S:RDD_pr_1}



    Combining \eqref{eq:MDStepProgr}, \eqref{eq:AssumDelta1_nonAcc} and \eqref{eq:AssumDelta2_nonAcc} we get
    \begin{equation*}
        \begin{array}{rl}
            \alpha\EE\left[\left\la\nabla f(x_k),\, x_k-x_*\right\ra\right] \leqslant 24\alpha^2n\rho_nL_2\left(\EE\left[f(x_{k})\right] - f(x_*)\right) + \alpha\delta_1n\EE\left[\|x_k-x_*\|_p\right] + \frac{\alpha^2n^2}{2}\delta_2\\
            + \EE\left[V[x_k](x_*)\right] - \EE\left[V[x_{k+1}](x_*)\right],
        \end{array}
    \end{equation*}
whence due to convexity of $f$ we have
\begin{equation}\label{theo_ds_der_free:pre_final}
    \begin{array}{rl}
        \underbrace{(\alpha - 24\alpha^2n\rho_nL_2)}_{\frac{\alpha}{4}}\left(\EE[f(x_k)] - f(x_*)\right) \leqslant \alpha\delta_1n\EE\left[\|x_k-x_*\|_p\right] + \frac{\alpha^2n^2}{2}\delta_2\\
        + \EE[V[x_k](x_*)] - \EE[V[x_{k+1}](x_*)],
    \end{array} 
\end{equation}
because $\alpha = \frac{1}{48n\rho_nL_2}$. Summing \eqref{theo_ds_der_free:pre_final} for $k=0,\ldots,l-1$, where $l \leqslant N$ we get
\begin{equation}\label{theo_ds_der_free:after_telescoping}
    \begin{array}{rl}
        0 \leqslant \frac{N\alpha}{4}\left(\EE[f(\bar{x}_l)] - f(x_*)\right) \leqslant \frac{\alpha^2n^2l}{2}\delta_2 + \alpha\delta_1n\sum\limits_{k=0}^{l-1}\EE[\|x_k-x_*\|_p]\\
        + \underbrace{V[x_0](x_*)}_{\Theta_p} - \EE[V[x_{l}](x^*)],
    \end{array}
\end{equation}
where $\bar{x}_l \overset{\text{def}}{=} \frac{1}{l}\sum\limits_{k=0}^{l-1}x_k$.
From the previous inequality we get
\begin{equation}\label{theo_ds_der_free:distance_estimation}
    \begin{array}{rl}
        \frac{1}{2}\left(\EE[\|x_l-x_*\|_p]\right)^2 \leqslant \frac{1}{2}\EE[\|x_l-x_*\|_p^2] \leqslant \EE[V[x_l](x_*)]\\
        \leqslant \Theta_p + l\cdot\frac{\alpha^2n^2}{2}\delta_2
        +\alpha\delta_1n\sum\limits_{k=0}^{l-1}\EE[\|x_k-x_*\|_p],
    \end{array}
\end{equation}
whence $\forall l\leqslant N$ we obtain
\begin{equation}
    \EE[\|x_k-x_*\|_p] \leqslant \sqrt{2}\sqrt{\Theta_p + l\cdot\frac{\alpha^2n^2}{2}\delta_2
    +\alpha\delta_1n\sum\limits_{k=0}^{l-1}\EE[\|x_k-x_*\|_p]}.
\end{equation}
Denote $R_k = \EE[\|x^*-x_k\|_p]$ for $k=0,\ldots,N$. Applying Lemma~\ref{stoh:technical_lemma_non_acc} for $a_0 = \Theta_p + \alpha\delta_1n\EE[\|x_0-x_*\|_p] \leqslant \Theta_p + \alpha n\sqrt{2\Theta_p}\delta_1, a_{k} = \frac{\alpha^2n^2}{2}\delta_2, b = n\delta_1$ for $k=1,\ldots,N-1$ we have for $l=N$
\begin{equation*}
    \begin{array}{rl}
        \frac{N\alpha}{4}\left(\EE[f(\bar{x}_N)] - f(x_*)\right)\\ 
        \leqslant \left(\sqrt{\Theta_p+N\cdot\frac{\alpha^2n^2}{2}\delta_2 + \alpha n\sqrt{2\Theta_p}\delta_1} + \sqrt{2}n\delta_1\alpha N\right)^2\\
        \overset{\circledOne}{\leqslant} 2\Theta_p + N\alpha^2n^2\delta_2 + 2\alpha n\sqrt{2\Theta_p}\delta_1 + 4n^2\delta_1^2\alpha^2N^2,
    \end{array}
\end{equation*}
whence
\begin{equation*}
    \begin{array}{rl}
        \EE[f(\bar{x}_N)] - f(x_*) \leqslant \frac{384n\rho_nL_2\Theta_p}{N} + \frac{n}{12\rho_nL_2}\delta_2 + \frac{8n\sqrt{2\Theta_p}}{N}\delta_1 + \frac{nN}{3L_2\rho_n}\delta_1^2,
    \end{array}
\end{equation*}
because $\alpha = \frac{1}{48n\rho_nL_2}$. 
\qed

\subsection{Proof Lemma \ref{Lm:RDFDSDeltaMeaning}}
\label{S:RDD_pr_2}

    Taking mathematical expectation w.r.t. all randomness from the \eqref{eq:tnfms} we obtain\footnote{Note that we use $s = x_k-x_*$ which does not depend on $\xi_1,\xi_2,\ldots,\xi_m$ from the $(k+1)$-th iterate and it does not depend on $e_{k+1}$. Therefore we can use tower property of mathematical expectation and take firstly conditional expectation w.r.t. $\xi_1,\ldots,\xi_m$ and after that take full expectation.} inequality \eqref{eq:AssumDelta1_nonAcc} with $\delta_1 = \frac{\sqrt{\Delta_\zeta}}{2\sqrt{n}} + \frac{2\Delta_\eta}{\sqrt{n}}$, because $\EE[|\zeta(x,\xi_i)|] \leqslant \sqrt{\EE[|\zeta(x,\xi_i)|^2]} \overset{\eqref{eq:tf_def}}{\leqslant} \sqrt{\Delta_\zeta}$. Combining \eqref{eq:tnfmq2} and
    \begin{equation*}
        \begin{array}{rl}
            \|g^m(x,\vec{\xi}_m)\|_2^2 \leqslant 2\|\nabla f(x)\|_2^2 + 2\|\nabla f(x) - g^m(x,\vec{\xi}_m)\|_2^2 \leqslant 4L_2\left(\EE[f(x)] - f(x_*)\right) + 2\|\nabla f(x) - g^m(x,\vec{\xi}_m)\|_2^2,\\
            \EE[\|\nabla f(x) - g^m(x,\vec{\xi}_m)\|_2^2] \leqslant \frac{\sigma^2}{m}
        \end{array}
    \end{equation*}
    and taking full mathematical expectation we obtain \eqref{eq:AssumDelta2_nonAcc} with $\delta_2 = \frac{24\rho_n}{n}\cdot\frac{\sigma^2}{m}+ \rho_n\Delta_\zeta + 16\rho_n\Delta_\eta^2$.
    \qed


\section{Proofs for strongly convex problems}
\label{S:SCProofs}



\subsection{Accelerated algorithm}

\begin{lemma}
\label{Lm:ACDS_sc_auxil}
Assume that we start ARDD Algorithm \ref{Alg:ARDS} from a random point $x_0$ such that $\EE_{x_0} \|x^*-x_0\|_p^2 \leqslant R_p^2$,
use the function $R_p^2 d \left( \frac{x-x_0}{R_p} \right)$ as the prox-function and run ARDD for $N_0$ iterations. 
Then
\begin{align}
& \EE[f(y_{N_0})] - f^* \leqslant \frac{aL_2R_p^2\Omega_p  }{N_0^2} + \frac{b \sigma^2 N_0}{mL_2} +  \Delta, \notag
\end{align}
where $a= 384 n^2 \rho_n$, $b=\frac{4}{n}$, $$\Delta = \frac{61N_0}{24L_2}\Delta_\zeta + \frac{122N_0}{3L_2}\Delta_\eta^2
        + \frac{12\sqrt{2nR_p^2\Omega_p}}{N_0^2}\left(\frac{\sqrt{\Delta_\zeta}}{2}+ 2\Delta_\eta\right)
        + \frac{N_0^2}{12n\rho_nL_2} \left(\frac{\sqrt{\Delta_\zeta}}{2} 
        + 2\Delta_\eta\right)^2$$ and the expectation is taken with respect to all the randomness.
\end{lemma}

\begin{proof}
Note that $R_p^2 d \left( \frac{x-x_0}{R_p} \right)$ is strongly convex with constant 1 w.r.t $\|\cdot\|_p$.
Since $0 = \arg \min d(x)$, we have, for the prox-function $\bar{d}(x) = R_p^2 d \left( \frac{x-x_0}{R_p} \right)$ and corresponding Bregman divergence $\bar{V}[x_0](x)$,
$$
\Theta_p = \bar{V}[x_0](x_*) = \bar{d}(x_{*}) - \bar{d}(x_0) - \la \nabla \bar{d}(x_0), x_{*} - x_0 \ra = \bar{d}(x_{*}) \leq \frac{R_p^2\Omega_p}{2}.
$$
Applying Theorem \ref{Th:ARDFDSConv} an taking additional expectation w.r.t to $x_0$, we finish the proof of the lemma.

\end{proof}

\begin{proof}[Proof of Theorem \ref{Th:ACDS_sc_rate}]
We prove by induction that
\begin{equation}
\EE \|u_{k}- x^*\|_p^2 \leq R_k^2 = R_p^2 2^{-k} + \frac{4 \Delta}{\mu_p}\left(1-2^{-k}\right).
\end{equation}
For $k=0$, this inequality obviously holds. Let us assume that it holds for some $k \geq 0$ and prove the induction step.
Applying Lemma \ref{Lm:ACDS_sc_auxil} at the step $k$ of Algorithm \ref{ACDS_sc}, we obtain that
$$
\EE f(u_{k+1}) - f^* = \EE f(y_{N_0}) - f^* \leqslant \frac{aL_2R_k^2\Omega_p  }{N_0^2} + \frac{b \sigma^2 N_0}{m_kL_2} +  \Delta.
$$
By definition of $N_0$, we have
$$
\frac{aL_2R_k^2\Omega_p  }{N_0^2} \leqslant \frac{aL_2R_k^2\Omega_p}{\frac{8 a L_2\Omega_p }{\mu_p}} = \frac{\mu_p R_k^2}{8}.
$$
By definition of $m_k$, we have
\begin{align}
& m_{k} \geqslant \frac{8b \sigma^2 N_0}{L_2 \mu_p R_p^2 2^{-k}} \geqslant \frac{8b \sigma^2 N_0}{L_2 \mu_p \left( R_p^2 2^{-k} + \frac{4 \Delta}{\mu_p}\left(1-2^{-k}\right) \right) }   = \frac{8b \sigma^2 N_0}{L_2 \mu_p R_k^2 } \notag
\end{align}
and
$$
\frac{b \sigma^2 N_0}{m_kL_2} \leqslant \frac{b \sigma^2 N_0}{L_2\frac{8b \sigma^2 N_0}{L_2 \mu_p R_k^2 }} = \frac{\mu_p R_k^2}{8}.
$$
Hence,
\begin{equation*}
    \begin{array}{rl}
        \EE f(u_{k+1}) - f^* \leqslant \frac{\mu_p R_k^2}{4} + \Delta = \frac{\mu_p }{4} \left( R_p^2 2^{-k} + \frac{4 \Delta}{\mu_p}\left(1-2^{-k}\right) \right) + \Delta\\
        = \frac{\mu_p }{2} \left( R_p^2 2^{-(k+1)} + \frac{4 \Delta}{\mu_p}\left(1-2^{-(k+1)}\right) \right) = \frac{\mu_p R_{k+1}^2 }{2}.
    \end{array}
\end{equation*}
Since $f$ is strongly convex, we have
$$
\EE \|u_{k+1}- x^*\|_p^2 \leqslant \frac{2}{\mu_p} (\EE f(u_{k+1}) - f^* ) \leqslant R_{k+1}^2.
$$
This finishes the induction step and, as a byproduct, we obtain inequality \eqref{eq:ARDFDSSConv}.

It remains to estimate the complexity. To make the right hand side of \eqref{eq:ARDFDSSConv} smaller than $\e$ it is sufficient to choose  $K=\left\lceil\log_2 \frac{\mu_p R_p^2}{\e}\right\rceil$.
To estimate the total number of oracle calls, we write
\begin{align}
\text{Number of calls} = \sum_{k=0}^{K-1}{N_0 m_k} &\leqslant \sum_{k=0}^{K-1}{N_0 \left(1 + \frac{8b \sigma^2 N_0 2^{k}}{L_2\mu_p R_p^2 } \right)}  \leqslant K N_0+\frac{8b \sigma^2 N_0^2 2^{K}}{L_2\mu_p R_p^2  }  \notag \\
& \leqslant \sqrt{\frac{8 a L_2\Omega_p }{\mu_p}} \log_2 \frac{\mu_p R_p^2 }{\e} + \frac{8b \sigma^2 }{L_2\mu_p R_p^2 } \cdot \frac{8 a L_2\Omega_p }{\mu_p} \cdot  \frac{\mu_p R_p^2 }{ \e} \notag \\
& \leqslant\sqrt{\frac{8 a L_2\Omega_p }{\mu_p}}\log_2 \frac{\mu_p R_p^2 }{ \e} + \frac{64 ab \sigma^2 \Omega_p  }{\mu_p \e }\notag \\
&= \widetilde{O}\left(\max\left\{n^{\frac12+\frac{1}{q}}\sqrt{\frac{L_2\Omega_p }{\mu_p}}\log_2 \frac{\mu_p R_p^2 }{ \e},\frac{n^{\frac{2}{q}}\sigma^2 \Omega_p}{\mu_p \e}\right\}\right), \notag
\end{align}
where we used that $a= 384 n^2 \rho_n$, $b=\frac{4}{n}$ and $\rho_n$ is given in Lemma \ref{Lm:MainTechLM}.

\end{proof}

\subsection{Non-accelerated algorithm}
\begin{lemma}
\label{Lm:CDS_sc_auxil}
Assume that we start RDD Algorithm \ref{Alg:RDFDS} from a random point $x_0$ such that $\EE_{x_0} \|x^*-x_0\|_p^2 \leqslant R_p^2$,
use the function $R_p^2 d \left( \frac{x-x_0}{R_p} \right)$ as the prox-function and run RDD for $N_0$ iterations. 
Then
\begin{align}
& \EE[f(y_{N_0})] - f^* \leqslant \frac{aL_2R_p^2\Omega_p  }{N_0} + \frac{b \sigma^2 }{mL_2} +  \Delta, \notag
\end{align}
where $a= 192 n \rho_n$, $b=2$, $\Delta = \frac{n}{12L_2}\Delta_\zeta + \frac{4n}{3L_2}\Delta_\eta^2 + \frac{8\sqrt{2nR_p^2\Omega_p}}{N_0}\left(\frac{\sqrt{\Delta_\zeta}}{2} + 2\Delta_\eta\right) + \frac{N_0}{3L_2\rho_n}\left(\frac{\sqrt{\Delta_\zeta}}{2} + 2\Delta_\eta\right)^2$ and the expectation is taken with respect to all the randomness.
\end{lemma}

\begin{proof}
Note that $R_p^2 d \left( \frac{x-x_0}{R_p} \right)$ is strongly convex with constant 1 w.r.t $\|\cdot\|_p$.
Since $0 = \arg \min d(x)$, we have, for the prox-function $\bar{d}(x) = R_p^2 d \left( \frac{x-x_0}{R_p} \right)$ and corresponding Bregman divergence $\bar{V}[x_0](x)$,
$$
\Theta_p = \bar{V}[x_0](x_*) = \bar{d}(x_{*}) - \bar{d}(x_0) - \la \nabla \bar{d}(x_0), x_{*} - x_0 \ra = \bar{d}(x_{*}) \leq \frac{R_p^2\Omega_p}{2}.
$$
Applying Theorem \ref{theorem_convergence_mini_gr_free_non_acc} an taking additional expectation w.r.t to $x_0$, we finish the proof of the lemma.

\end{proof}

\begin{proof}[Proof of Theorem \ref{Th:CDS_sc_rate}]
We prove by induction that
\begin{equation}
\EE \|u_{k}- x^*\|_p^2 \leq R_k^2 = R_p^2 2^{-k} + \frac{4 \Delta}{\mu_p}\left(1-2^{-k}\right).
\end{equation}
For $k=0$, this inequality obviously holds. Let us assume that it holds for some $k \geq 0$ and prove the induction step.
Applying Lemma \ref{Lm:CDS_sc_auxil} at the step $k$ of Algorithm \ref{CDS_sc}, we obtain that
$$
\EE f(u_{k+1}) - f^* = \EE f(y_{N_0}) - f^* \leqslant \frac{aL_2R_k^2\Omega_p  }{N_0} + \frac{b \sigma^2 }{m_kL_2} +  \Delta.
$$
By definition of $N_0$, we have
$$
\frac{aL_2R_k^2\Omega_p  }{N_0} \leqslant \frac{aL_2R_k^2\Omega_p}{\frac{8 a L_2\Omega_p }{\mu_p}} = \frac{\mu_p R_k^2}{8}.
$$
By definition of $m_k$, we have
\begin{align}
& m_{k} \geqslant \frac{8b \sigma^2 }{L_2 \mu_p R_p^2 2^{-k}} \geqslant \frac{8b \sigma^2 }{L_2 \mu_p \left( R_p^2 2^{-k} + \frac{4 \Delta}{\mu_p}\left(1-2^{-k}\right) \right) }   = \frac{8b \sigma^2 }{L_2 \mu_p R_k^2 } \notag
\end{align}
and
$$
\frac{b \sigma^2 }{m_kL_2} \leqslant \frac{b \sigma^2 }{L_2\frac{8b \sigma^2 }{L_2 \mu_p R_k^2 }} = \frac{\mu_p R_k^2}{8}.
$$
Hence,
\begin{equation*}
    \begin{array}{rl}
        \EE f(u_{k+1}) - f^* \leqslant \frac{\mu_p R_k^2}{4} + \Delta = \frac{\mu_p }{4} \left( R_p^2 2^{-k} + \frac{4 \Delta}{\mu_p}\left(1-2^{-k}\right) \right) + \Delta\\
        = \frac{\mu_p }{2} \left( R_p^2 2^{-(k+1)} + \frac{4 \Delta}{\mu_p}\left(1-2^{-(k+1)}\right) \right) = \frac{\mu_p R_{k+1}^2 }{2}.
    \end{array}
\end{equation*}
Since $f$ is strongly convex, we have
$$
\EE \|u_{k+1}- x^*\|_p^2 \leqslant \frac{2}{\mu_p} (\EE f(u_{k+1}) - f^* ) \leqslant R_{k+1}^2.
$$
This finishes the induction step and, as a byproduct, we obtain inequality \eqref{eq:RDFDSSConv}.

It remains to estimate the complexity. To make the right hand side of \eqref{eq:RDFDSSConv} smaller than $\e$ it is sufficient to choose  $K=\left\lceil\log_2 \frac{\mu_p R_p^2}{\e}\right\rceil$.
To estimate the total number of oracle calls, we write
\begin{align}
\text{Number of calls} = \sum_{k=0}^{K-1}{N_0 m_k} &\leqslant \sum_{k=0}^{K-1}{N_0 \left(1 + \frac{8b \sigma^2  2^{k}}{L_2\mu_p R_p^2 } \right)}   \leqslant K N_0+\frac{8b \sigma^2 N_0 2^{K}}{L_2\mu_p R_p^2  }  \notag \\
& \leqslant \frac{8 a L_2\Omega_p }{\mu_p} \log_2 \frac{\mu_p R_p^2 }{\e} + \frac{8b \sigma^2 }{L_2\mu_p R_p^2 } \cdot \frac{8 a L_2\Omega_p }{\mu_p} \cdot  \frac{\mu_p R_p^2 }{ \e} \notag \\
& \leqslant\frac{8 a L_2\Omega_p }{\mu_p}\log_2 \frac{\mu_p R_p^2 }{ \e} + \frac{64 ab \sigma^2 \Omega_p  }{\mu_p \e }\notag\\
&= \widetilde{O}\left(\max\left\{\frac{n^{\frac{2}{q}}L_2\Omega_p }{\mu_p}\log_2 \frac{\mu_p R_p^2 }{ \e},\frac{n^{\frac{2}{q}}\sigma^2 \Omega_p}{\mu_p \e}\right\}\right), \notag
\end{align}
where we used that $a= 192 n \rho_n$, $b=2$ and $\rho_n$ is given in Lemma \ref{Lm:MainTechLM}.

\end{proof}

{
\section{Numerical experiments}\label{sec:numerical_experiments}
In this section we numerically test our methods on the ``worst in the world'' function from \cite{nesterov2004introduction} and least squares problem. In these problems there is no noise of type $\eta(x,\xi,e)$ from \eqref{eq:tf_def} since one can compute directional derivatives with machine precision. Moreover, for both examples one can compute exact functional values, therefore, using small enough smoothing parameter $t$ (see \eqref{eq:FinDiffStocGrad}) it is possible to approximate directional derivatives via finite differences with high enough accuracy. That is, for the problems we consider in this section the difference between directional derivative oracle and derivative-free oracle is negligible to influence the behaviour of our methods. Taking it into account we consider only derivative-free oracle in the experiments and compare our methods with RSGF from \cite{ghadimi2013stochastic}.

\subsection{Nesterov's function}\label{sec:nesterov_func}
We start with numerical tests on Nesterov's function
\begin{equation}
    f(x) = \frac{L}{8}\left(x_1^2+\sum\limits_{i=0}^{n-1}\left(x_i - x_{i+1}\right)^2 + x_i^2\right) - \frac{L}{4}x_1\label{eq:nesterov_func}
\end{equation}
which is convex, $L$-smooth and attains its minimal value $f^* = \frac{L}{8}\left(-1+\frac{1}{n+1}\right)$ at such $x^* = (x_1^*,\ldots, x_n^*)^\top$ that $x_i^* = 1 - \frac{i}{n+1}$ for $i=1,\ldots,n$ \cite{nesterov2004introduction}. We take the starting point $x_0$ such that all coordinates expect the first one coincides with corresponding coordinates of $x^*$ and we take $10$ as the first coordinate of $x_0$. We also choose $L = 10$, $t = 10^{-8}$ and consider $n=100,1000,5000$. The results can be found in Figure~\ref{fig:nesterov_comparison}.
\begin{figure}[h]
    \centering
    \includegraphics[width=0.45\textwidth]{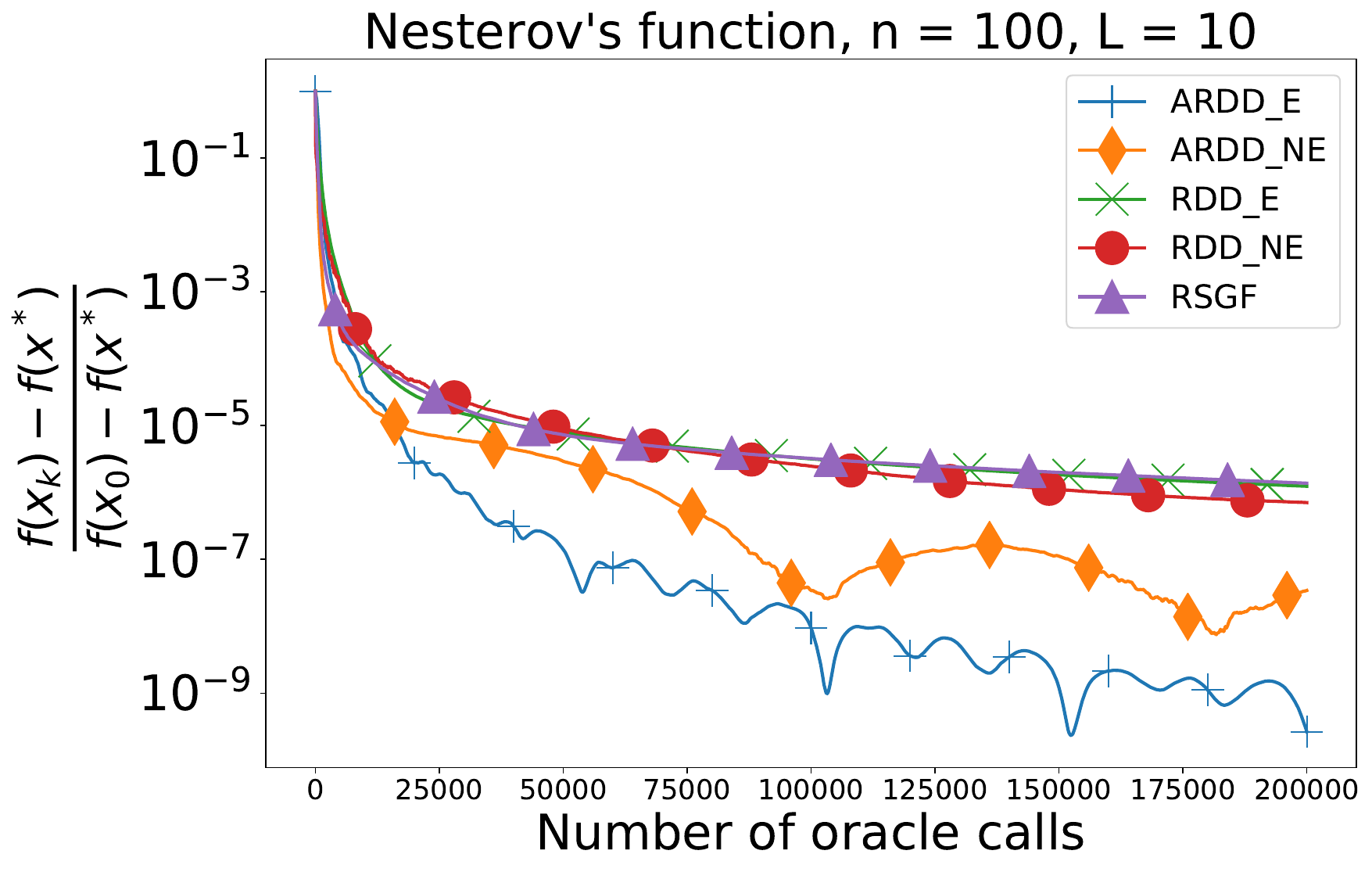}
    \includegraphics[width=0.45\textwidth]{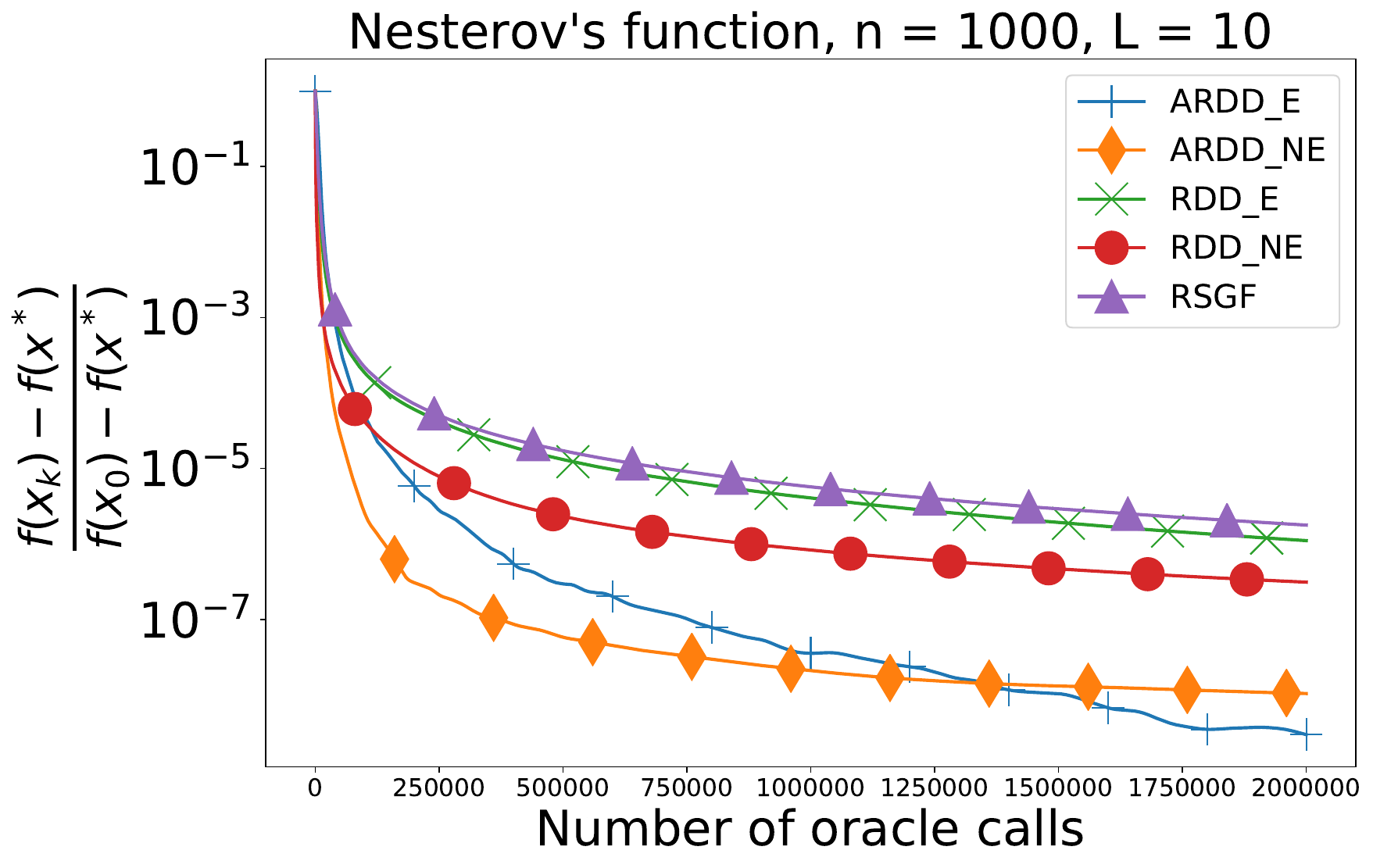}
    \includegraphics[width=0.45\textwidth]{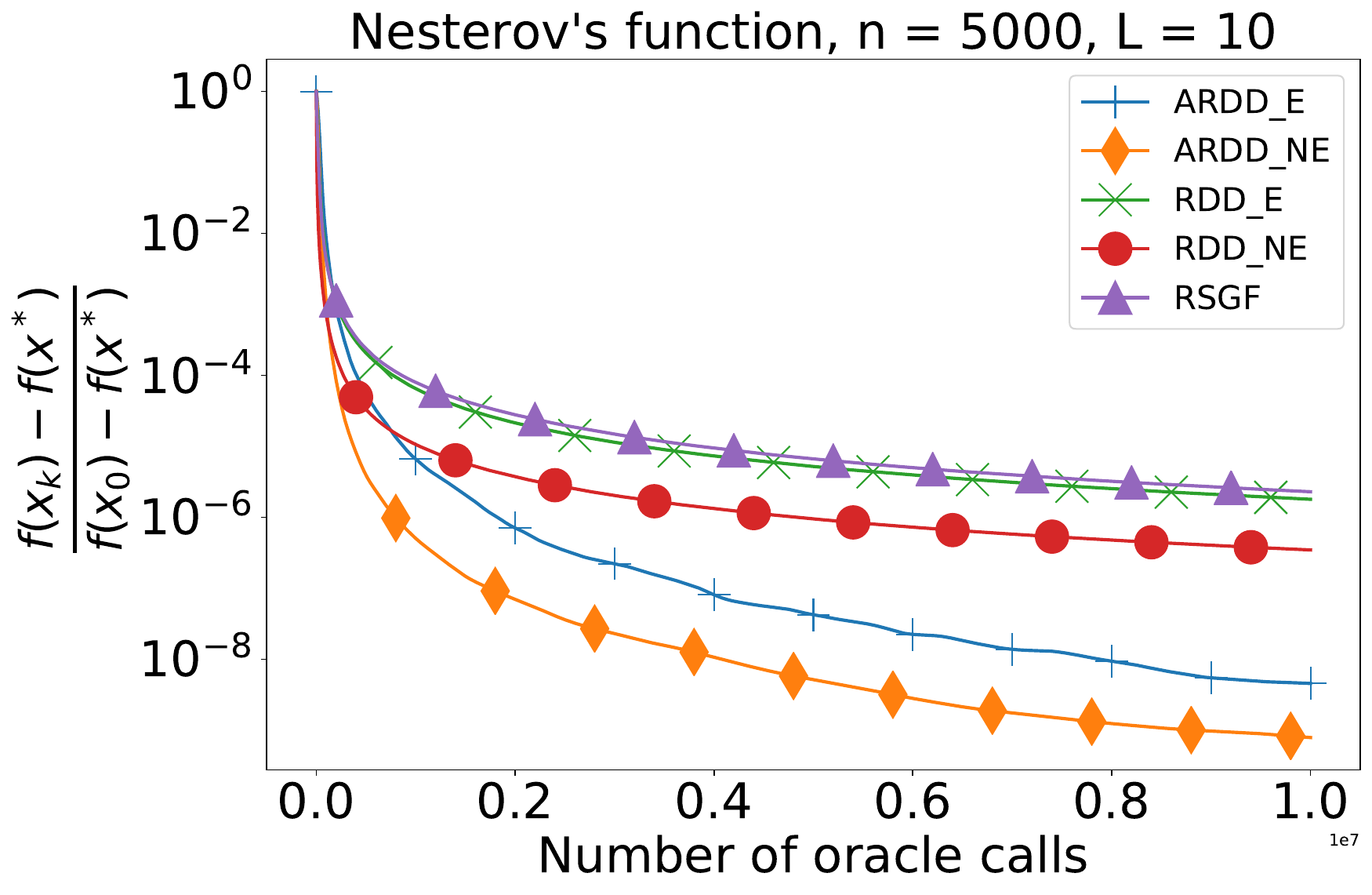}
    \caption{ARDD, RDD and RSGF applied to minimize Nesterov's function \eqref{eq:nesterov_func}. We use {\_}E and {\_}NE to define $\ell_2$ and $\ell_1$ proximal setups respectively (see \eqref{eq:dp1} and \eqref{eq:dp2} for the details). In the plot for $n=5000$ number of oracle calls is divided by $10^7$.}
    \label{fig:nesterov_comparison}
\end{figure}
In these settings $\|x_0 - x^*\|_1 = \|x_0 - x^*\|_2$ and our theory establishes (see Tables~\ref{tab:ARDD}~and~\ref{tab:RDD}) better complexity bounds for the case when $p=1$ then for the Euclidean case especially for big $n$. The experiments confirm this claim: as one can see in Figure~\ref{fig:nesterov_comparison}, the choice of $\ell_1$ proximal setup becomes more beneficial than standard Euclidean setup for $n = 1000$ and $n=5000$ to reach good enough accuracy. Indeed, our choice of the starting point and $L$ implies that $f(x_0) - f(x^*) \approx 200$ and for $n=1000$ and $n=5000$ ARDD with $\ell_1$ proximal setup (ARDD{\_}NE in Figure~\ref{fig:nesterov_comparison}) make $f(x_N) - f(x^*)$ of order $10^{-3}-10^{-5}$ faster than ARDD with $p=2$ (ARDD{\_}E in Figure~\ref{fig:nesterov_comparison}) and RDD with $p=1$ (RDD{\_}NE in Figure~\ref{fig:nesterov_comparison}) finds such $x^N$ that $f(x^N) - f(x^*)$ is of order $10^{-3}$ faster than its Euclidean counterpart (RDD{\_}E in Figure~\ref{fig:nesterov_comparison}). Finally, all of our methods outperform RSGF on the considered problem.

To perform mirror descent step for $p=1$ we apply relations obtained in Appendix B from \cite{dvurechensky2018accelerated}. See other details connected with parameters tuning in \ref{sec:param_tunning} of this work.

\subsection{Least squares problem}\label{sec:least_squares}
In this subsection we consider least squares problem:
\begin{equation}
    \min\limits_{x\in\R^n}\left\{f(x) = \frac{1}{2r}\|Ax - b\|_2^2 = \frac{1}{r}\sum\limits_{i=1}^r\frac{1}{2}(A_ix - b_i)^2\right\}. \label{eq:least_squares_pr}
\end{equation}
Here $A$ is $r\times n$ real matrix, $b\in\R^r$ and $A_i$ denotes the $i$-th row of $A$. Clearly, $f(x)$ is convex and smooth function. Moreover, each summand $f_i(x) = \frac{1}{2}(A_ix - b_i)^2$ is also convex and $L_{2.i}$-smooth function with $L_{2,i} = \|A_i\|_2^2$. One can consider \eqref{eq:least_squares_pr} as \eqref{eq:PrSt} with $F(x,\xi) = f_{\xi}(x) = \frac{1}{2}(A_{\xi}x - b_{\xi})^2$ where $\xi$ is uniformly distributed on $\{1,2,\ldots,r\}$. Then, by definition of $L_2$ we have
\begin{equation}
    L_2 = \sqrt{\EE_\xi L_{2,\xi}^2} = \sqrt{\frac{1}{r}\sum\limits_{i=1}^r \|A_i\|_2^2} = \frac{\|A\|_F}{\sqrt{r}}\label{eq:L_2_least_squares}
\end{equation}
where $\|A\|_F$ denotes Frobenius norm of matrix $A$.

In our preliminary experiments elements of $A$ and $b$ were sampled independently from the standard normal distribution and then matrix $A$ was normalized by its $\ell_2$-norm. In particular, we choose $r=300$ and $n=400$ which implies that $f(x)$ is just convex but not strongly convex and $f(x^*) = 0$. Moreover, we compute the solution $x^*$ as $A^{+}b$ where $A^{+}$ denotes Moore-Penrose inverse of $A$ and choose the starting point $x_0$ as $x^*$ and $100$ to the first component. In our tests the suboptimality of the starting point, i.e.\ $f(x_0) - f(x^*)$, was approximately $3$. The results can be found in Figure~\ref{fig:ls_comparison}.
\begin{figure}[h]
    \centering
    \includegraphics[width=0.7\textwidth]{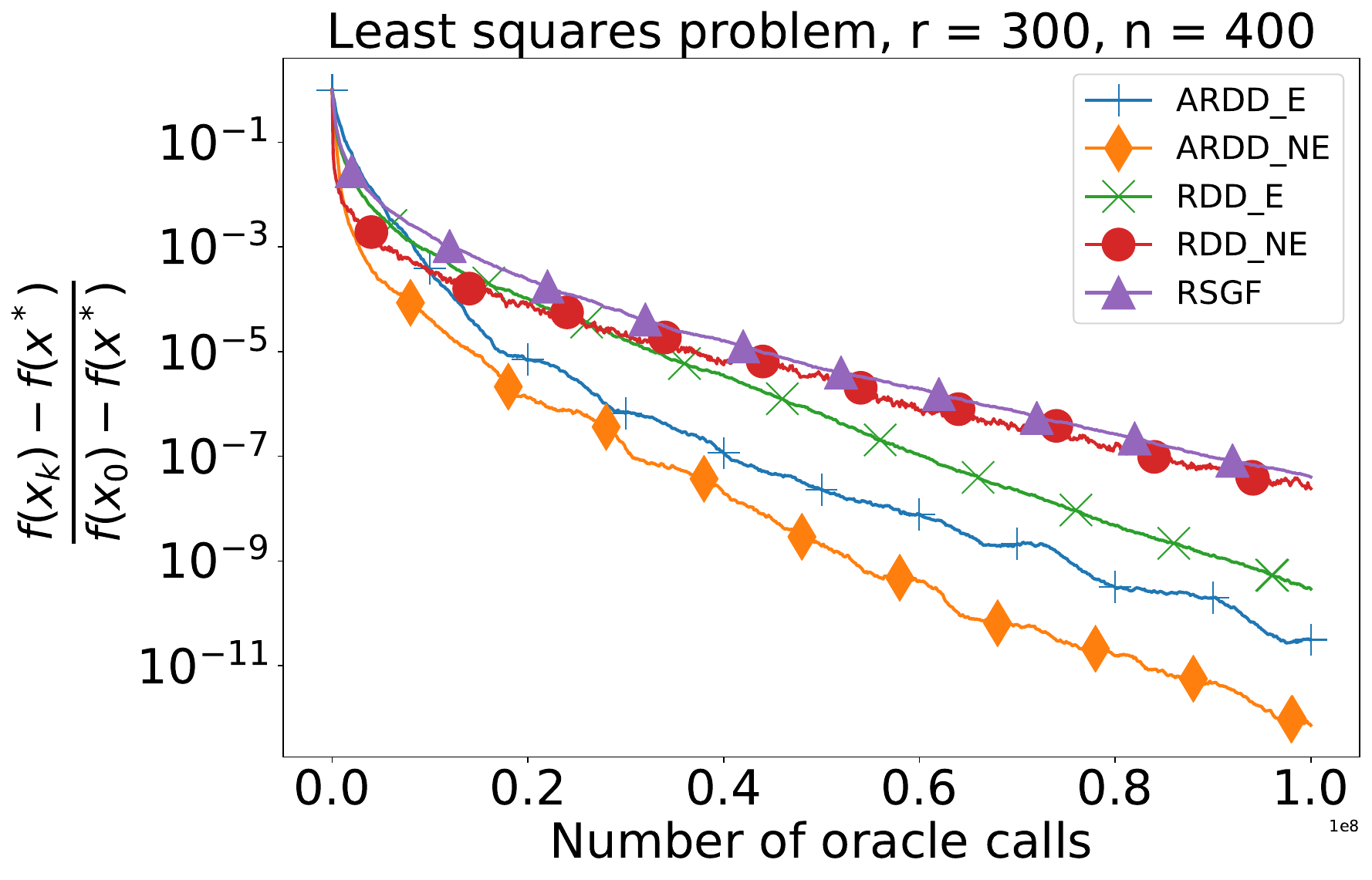}
    \caption{ARDD, RDD and RSGF applied to solve least squares problem \eqref{eq:least_squares_pr}. We use {\_}E and {\_}NE to define $\ell_2$ and $\ell_1$ proximal setups respectively (see \eqref{eq:dp1} and \eqref{eq:dp2} for the details). For all methods batch size $m$ equals $50$. By oracle call we mean one computation of functional value of a summand. Number of oracle calls is divided by $10^8$.}
    \label{fig:ls_comparison}
\end{figure}
We want to notice that in these preliminary experiments with stochasticity in functional values in experiments with ARDD it was needed to tune not only $\alpha_{k+1}$ that appears in the mirror descent step, but also the stepsize for the gradient step, see the details in \ref{sec:param_tunning}.

\section{Conclusion}
In this paper we propose four novel directional derivative methods for smooth stochastic convex and strongly convex optimization with corollaries for derivative-free optimization. These methods are able to work with Euclidean and non-Euclidean proximal setups. We prove complexity results showing that in non-Euclidean case complexities of our methods outperform state-of-the-art results for directional derivative and derivative-free methods in terms of the dependence on the dimension of the problem under assumption that $\ell_1$ and $\ell_2$ norms of $x_0 - x^*$ are close to each other, e.g.\ when $x_0 = 0$ and $x^*$ is sparse. Moreover, we analyze our methods under general assumptions on the noisy oracle and provide bounds for the admissible noise levels. Since we use mini-batches, we are able to separate iteration complexity and sample complexity, the former being up to a dimension-dependent factor the same as for accelerated gradient method in the standard deterministic full-gradient setting. This makes our methods amenable to parallel computation setting \cite{dvurechensky2018parallel} and leads to acceleration in this setting compared to standard stochastic gradient methods \cite{duchi2015optimal}. Finally, we conduct several experiments providing numerical justifications of the obtained results.

\rev2{Using an additional ``light-tail'' assumption that $ \EE_{\xi}[\exp(\|g(x,\xi) - \nabla f(x)\|_2^2/\sigma^2)] \leqslant \exp(1)$ and techniques of \cite{gorbunov2019optimal} our algorithms and analysis can be extended to obtain results in terms of probability of large deviations. For example, in the case of controlled noise levels $\Delta_\zeta, \Delta_\eta$ this means that an algorithm outputs a point $\hat{x}$ which satisfies $\mathbb{P} \{f(\hat{x}) - f(x^*) \leqslant \e\} \geqslant 1-\delta$, where $\delta \in (0,1)$ is the confidence level, for the price of extra $\ln \frac{1}{\delta}$ factor in $N$ and $m$.}
As directions of future research we would like to point a primal-dual extension for problems with linear constraints in the spirit of \cite{dvurechensky2016primal-dual,chernov2016fast,anikin2017dual,bayandina2018mirror,dvurechensky2018decentralize,dvinskikh2019primal,nesterov2020primal-dual}, an extension with line-search to adapt to an unknown value of $L_2$ using the techniques in \cite{cartis2018global,berahas2019global,dvinskikh2020line-search}, an extension for the case of intermediate smoothness \cite{nesterov2015universal,kamzolov2020universal} or interpolation between accelerated and non-accelerated methods \cite{gasnikov2016stochasticInter,dvurechensky2016stochastic}, as well as extension to a more general type of inexactness called inexact model of the objective \cite{stonyakin2019gradient,stonyakin2020inexact}.
}

\section*{Acknowledgements}
The research is supported by the Ministry of Science and Higher Education of the Russian Federation (Goszadaniye) No. 075-00337-20-03, project No. 0714-2020-0005.

 \appendix

\section{Proof of Lemma~\ref{Lm:MainTechLM}}\label{sec:MainTechLM_proof}
Here we prove that, for $e \in RS_2\left( 1 \right)$ 
\begin{equation}\label{lemm1:expect_q_norm}
        \EE[\|e\|_q^2] \leqslant \min\{q-1,\,16\ln n - 8\}n^{\frac{2}{q}-1},
    \end{equation}
    \begin{equation}\label{lemm1:expect_inner_product}
        \EE[\langle s,\, e\rangle^2\|e\|_q^2] \leqslant 6\|s\|_2^2\min\{q-1,16\ln n -8\}n^{\frac{2}{q}-2}.
\end{equation}

We start with proving the following inequality which could be rough for big $q$:
\begin{equation}\label{lemm1:rough_estimation_of_expectation_q_norm}
    \EE[\|e\|_q^2] \leqslant (q-1)n^{\frac{2}{q}-1},\quad 2\leqslant q < \infty.
\end{equation}
We have
\begin{equation}\label{lemm1:jensen}
    \begin{array}{rl}
        \EE[\|e\|_q^2] = \EE\left[\left(\sum\limits_{k=1}^{n}|e_k|^q\right)^{\frac{2}{q}} \right] \overset{\circledOne}{\leqslant} \left(\EE\left[\sum\limits_{k=1}^{n}|e_k|^q\right]\right)^{\frac{2}{q}} \overset{\circledTwo}{=} \left(n\EE[|e_2|^q]\right)^{\frac{2}{q}},
    \end{array}
\end{equation}
where $\circledOne$ is due to probabilistic version of Jensen's inequality (function $\varphi(x) = x^{\frac{2}{q}}$ is concave, because $q\geqslant2$) and $\circledTwo$ is because mathematical expectation is linear and components of vector $e$ are identically distributed.

Moreover, due to Poincare lemma, we have
\begin{equation}
    e \overset{d}{=} \frac{\xi}{\sqrt{\xi_1^2 + \dots + \xi_n^2}},
\end{equation}
where $\xi$ is Gaussian random vector which mathematical expectation is zero vector and covariance matrix is identical. Then
\begin{equation*}
    \begin{array}{rl}
        \EE[|e_2|^q] = \EE\left[\frac{|\xi_2|^q}{\left(\xi_1^2+\ldots+\xi_n^2\right)^{\frac{q}{2}}}\right]\\
        = \idotsint\limits_{\R^n}|x_2|^q\left(\sum\limits_{k=1}^{n}x_k^2\right)^{-\frac{q}{2}}\cdot\frac{1}{(2\pi)^{\frac{n}{2}}}\cdot \exp\left(-\frac{1}{2}\sum\limits_{k=1}^{n}x_k^2\right)dx_1\ldots dx_n.
    \end{array}
\end{equation*}
Consider spherical coordinates:
\begin{equation*}
    \begin{array}{rl}
        x_1 = r\cos\varphi\sin\theta_1\ldots\sin\theta_{n-2},\\
        x_2 = r\sin\varphi\sin\theta_1\ldots\sin\theta_{n-2},\\
        x_3 = r\cos\theta_1\sin\theta_2\ldots\sin\theta_{n-2},\\
        x_4 = r\cos\theta_2\sin\theta_3\ldots\sin\theta_{n-2},\\
        \ldots\\
        x_n = r\cos\theta_{n-2},\\
        r>0,\,\varphi \in [0,2\pi),\, \theta_i \in [0,\pi],\, i = \overline{1,n-2}. 
    \end{array}
\end{equation*}
The Jacobian of mapping is
\begin{equation*}
    \det\left(\frac{\partial(x_1,\ldots,x_n)}{\partial(r,\varphi,\theta_1,\theta_2,\ldots,\theta_{n-2})}\right) = r^{n-1}\sin\theta_1(\sin\theta_2)^2\ldots(\sin\theta_{n-2})^{n-2}.
\end{equation*}
Then mathematical expectation $\EE[|e_2|^q]$ could be rewritten in the following form:
\begin{equation*}
    \begin{array}{rl}
        \EE[|e_2|^q]\\
        = \idotsint\limits_{\substack{r>0,\,\varphi \in [0,2\pi),\\ \theta_i \in [0,\pi],\, i = \overline{1,n-2}}}r^{n-1}|\sin\varphi|^q|\sin\theta_1|^{q+1}|\sin\theta_2|^{q+2}\ldots|\sin\theta_{n-2}|^{q+n-2}\\
        \cdot\frac{e^{-\frac{r^2}{2}}}{(2\pi)^{\frac{n}{2}}}dr\ldots d\theta_{n-2}\\
        =\frac{1}{(2\pi)^{\frac{n}{2}}} I_r\cdot I_\varphi \cdot I_{\theta_1}\cdot I_{\theta_2}\cdot\ldots\cdot I_{\theta_{n-2}},
    \end{array}
\end{equation*}
where
\begin{equation*}
    \begin{array}{rl}
        I_r = \int\limits_{0}^{+\infty}r^{n-1}e^{-\frac{r^2}{2}}dr,\\
        I_\varphi = \int\limits_{0}^{2\pi}|\sin\varphi|^qd\varphi = 2\int\limits_{0}^{\pi}|\sin\varphi|^qd\varphi,\\
        I_{\theta_i} = \int\limits_{0}^{\pi}|\sin\theta_i|^{q+i}d\theta_i,\, i=\overline{1,n-2}.
    \end{array}
\end{equation*}
Now we are going to compute these integrals. Start with $I_r$:
\begin{equation*}
    \begin{array}{rl}
        I_r = \int\limits_{0}^{+\infty}r^{n-1}e^{-\frac{r^2}{2}}dr = /r = \sqrt{2t}/ = \int\limits_{0}^{+\infty}(2t)^{\frac{n}{2}-1}e^{-t}dt = 2^{\frac{n}{2}-1}\Gamma(\frac{n}{2}).
    \end{array}
\end{equation*}
To compute other integrals it is useful to consider the following integral $(\alpha > 0)$:
\begin{equation*}
    \begin{array}{rl}
        \int\limits_{0}^{\pi} |\sin\varphi|^\alpha d\varphi = 2\int\limits_{0}^{\frac{\pi}{2}}|\sin\varphi|^\alpha d\varphi = 2\int\limits_{0}^{\frac{\pi}{2}}(\sin^2\varphi)^{\frac{\alpha}{2}}d\varphi = /t = \sin^2\varphi /\\ 
        = \int\limits_{0}^{1}t^{\frac{\alpha-1}{2}}(1-t)^{-\frac{1}{2}}dt = B(\frac{\alpha+1}{2},\,\frac{1}{2}) = \frac{\Gamma(\frac{\alpha+1}{2})\Gamma(\frac{1}{2})}{\Gamma(\frac{\alpha+2}{2})} = \sqrt{\pi} \frac{\Gamma(\frac{\alpha+1}{2})}{\Gamma(\frac{\alpha+2}{2})}.
    \end{array}
\end{equation*}
From this we obtain
\begin{equation}\label{lemm1:expectation_component}
    \begin{array}{rl}
        \EE[|e_2|^q] = \frac{1}{(2\pi)^{\frac{n}{2}}} I_r\cdot I_\varphi \cdot I_{\theta_1}\cdot I_{\theta_2}\cdot\ldots\cdot I_{\theta_{n-2}}\\
        = \frac{1}{(2\pi)^{\frac{n}{2}}}\cdot 2^{\frac{n}{2}-1}\Gamma(\frac{n}{2})\cdot 2\sqrt{\pi}\frac{\Gamma(\frac{q+1}{2})}{\Gamma(\frac{q+2}{2})}\cdot\sqrt{\pi}\frac{\Gamma(\frac{q+2}{2})}{\Gamma(\frac{q+3}{2})}\cdot\sqrt{\pi}\frac{\Gamma(\frac{q+3}{2})}{\Gamma(\frac{q+4}{2})}\cdot\ldots\cdot\sqrt{\pi}\frac{\Gamma(\frac{q+n-1}{2})}{\Gamma(\frac{q+n}{2})}\\
        = \frac{1}{\sqrt{\pi}}\cdot\frac{\Gamma(\frac{n}{2})\Gamma(\frac{q+1}{2})}{\Gamma(\frac{q+n}{2})}.
    \end{array}
\end{equation}
Now, we want to show that $\forall\, q\geqslant 2$
\begin{equation}\label{lemm1:key_estimation}
    \frac{1}{\sqrt{\pi}}\cdot\frac{\Gamma(\frac{n}{2})\Gamma(\frac{q+1}{2})}{\Gamma(\frac{q+n}{2})} \leqslant \left(\frac{q-1}{n}\right)^{\frac{q}{2}}.
\end{equation}
At the beginning show that \eqref{lemm1:key_estimation} holds for $q=2$ (and arbitrary $n$):
\begin{equation*}
    \frac{1}{\sqrt{\pi}}\cdot\frac{\Gamma(\frac{n}{2})\Gamma(\frac{2+1}{2})}{\Gamma(\frac{2+n}{2})} - \frac{1}{n} = \frac{1}{\sqrt{\pi}} \cdot \frac{\Gamma(\frac{n}{2})\cdot\frac{1}{2}\Gamma(\frac{1}{2})}{\frac{n}{2}\Gamma(\frac{n}{2})} - \frac{1}{n} = \frac{1}{n}-\frac{1}{n} = 0 \leqslant 0.
\end{equation*}
Consider the function
\begin{equation*}
    f_n(q) = \frac{1}{\sqrt{\pi}}\cdot\frac{\Gamma(\frac{n}{2})\Gamma(\frac{q+1}{2})}{\Gamma(\frac{q+n}{2})} - \left(\frac{q-1}{n}\right)^{\frac{q}{2}}
\end{equation*}
where $q \geqslant 2$. Also consider $\psi(x) = \frac{d(\ln(\Gamma(x)))}{dx}$ with $x > 0$ which is called (\textit{digamma function}). For gamma function it holds
\begin{equation*}
    \Gamma(x+1) = x\Gamma(x),\, x>0.
\end{equation*}
Taking natural logarithm from it and taking derivative w.r.t. $x$:
\begin{equation*}
    \begin{array}{rl}
        \ln\Gamma(x+1) = \ln\Gamma(x) + \ln x,\\
        \frac{d(\ln(\Gamma(x+1)))}{dx} = \frac{d(\ln(\Gamma(x)))}{dx} + \frac{1}{x},
    \end{array}
\end{equation*}
which could be written in digamma-function-notation:
\begin{equation}\label{lemm1:digamma_recurrence}
    \psi(x+1) = \psi(x) + \frac{1}{x}.
\end{equation}
One can show that digamma function is monotonically increases when $x>0$. To prove this fact we are going to show that
\begin{equation}\label{lemm1:digamma_decrease}
    \left(\Gamma'(x)\right)^2 < \Gamma(x)\Gamma''(x).
\end{equation}
That is,
\begin{equation*}
    \begin{array}{rl}
        \left(\Gamma'(x)\right)^2 = \left(\int\limits_0^{+\infty}e^{-t}\ln t\cdot t^{x-1}dt\right)^2\\
        \overset{\circledOne}{<} \int\limits_0^{+\infty}\left(e^{-\frac{t}{2}}t^{\frac{x-1}{2}}\right)^2dt\cdot\int\limits_0^{+\infty}\left(e^{-\frac{t}{2}}t^{\frac{x-1}{2}}\ln t\right)^2dt = \underbrace{\int\limits_0^{+\infty}e^{-t}t^{x-1}dt}_{\Gamma(x)}\cdot\underbrace{\int\limits_{0}^{+\infty}e^tt^{x-1}\ln^2 tdt}_{\Gamma''(x)},
    \end{array}
\end{equation*}
where $\circledOne$ follows from Cauchy-Schwartz inequality (the equality cannot occur because functions $e^{-\frac{t}{2}}t^{\frac{x-1}{2}}$ and $e^{-\frac{t}{2}}t^{\frac{x-1}{2}}\ln t$ are linearly independent). From \eqref{lemm1:digamma_decrease} follows that
\begin{equation*}
    \frac{d^2(\ln\Gamma(x))}{dx^2} = \left(\frac{\Gamma'(x)}{\Gamma(x)}\right)' = \frac{\Gamma''(x)}{\Gamma(x)} - \frac{\left(\Gamma'(x)\right)^2}{\left(\Gamma(x)\right)^2} \overset{\eqref{lemm1:digamma_decrease}}{>} 0, 
\end{equation*}
which shows that digamma function increases.

Now we show that $f_n(q)$ decreases on the interval $[2,+\infty)$. To obtain it is sufficient to consider $\ln(f(q))$:
\begin{equation*}
    \begin{array}{rl}
        \ln(f_n(q))\\
        = \ln\left(\frac{\Gamma(\frac{n}{2})}{\sqrt{\pi}}\right) + \ln\left(\Gamma\left(\frac{q+1}{2}\right)\right) - \ln\left(\Gamma\left(\frac{q+n}{2}\right)\right) - \frac{q}{2}\left(\ln(q-1)-\ln n\right),\\
        \frac{d(\ln(f_n(q)))}{dq} = \frac{1}{2}\psi\left(\frac{q+1}{2}\right)-\frac{1}{2}\psi\left(\frac{q+n}{2}\right)-\frac{1}{2}\ln(q-1)-\frac{q}{2(q-1)} + \frac{1}{2}\ln n.
    \end{array}
\end{equation*}
We are going to show that $\frac{d(\ln(f_n(q)))}{dq} < 0$ for $q\geqslant 2$. Let $k = \lfloor\frac{n}{2}\rfloor$ (the closest integer which is no greater than $\frac{n}{2}$). Then $\psi\left(\frac{q+n}{2}\right) > \psi\left(k-1+\frac{q+1}{2}\right)$ and $\ln n \leqslant \ln(2k+1)$, whence
\begin{equation*}
    \begin{array}{rl}
        \frac{d(\ln(f_n(q)))}{dq}\\
        < \frac{1}{2}\left(\psi\left(\frac{q+1}{2}\right)-\psi\left(k-1+\frac{q+1}{2}\right)\right)-\frac{1}{2}\ln(q-1)-\frac{q}{2(q-1)} + \frac{1}{2}\ln(2k+1)\\
        \overset{\eqref{lemm1:digamma_recurrence}}{=}\frac{1}{2} \left(\psi\left(\frac{q+1}{2}\right) - \sum\limits_{i=1}^{k-1}\frac{1}{\frac{q+1}{2}+ k - i - 1} - \psi\left(\frac{q+1}{2}\right)\right) - \frac{q}{2(q-1)} + \frac{1}{2}\ln\left(\frac{2k+1}{q-1}\right)\\
        \overset{\circledOne}{\leqslant} -\frac{1}{2}\sum\limits_{i=1}^{k-1}\frac{2}{q-1+2k-2i} - \frac{1}{q-1} + \frac{1}{2}\ln\left(\frac{2k+1}{q-1}\right)\\
        = -\frac{1}{2}\left(\frac{2}{q-1} + \frac{2}{q+1} + \frac{2}{q+3}+ \ldots + \frac{2}{q+2k-3}\right) + \frac{1}{2}\ln\left(\frac{2k+1}{q-1}\right)\\
        \overset{\circledTwo}{<} -\frac{1}{2}\ln\left(\frac{q+2k-1}{q-1}\right)+\frac{1}{2}\ln\left(\frac{2k+1}{q-1}\right) \overset{\circledThree}{\leqslant}-\frac{1}{2}\ln\left(\frac{2k+1}{q-1}\right)+\frac{1}{2}\ln\left(\frac{2k+1}{q-1}\right) = 0,
    \end{array}
\end{equation*}
where $\circledOne$ and $\circledThree$ is because $q\geqslant2$, $\circledTwo$ is due to estimation of integral of $\frac{1}{x}$ by integral of $g(x) = \frac{1}{q-1+2i},\, x\in[q-1+2i,q-1+2i+2],\, i=\overline{0,2k-1}$ which is no less than $f(x)$: 
\begin{equation*}
    \frac{2}{q-1} + \frac{2}{q+1} + \frac{2}{q+3} + \ldots + \frac{2}{q+2k-3} > \int\limits_{q-1}^{q+2k-1}\frac{1}{x}dx = \ln\left(\frac{q+2k-1}{q-1}\right).
\end{equation*}

So, we shown that $\frac{d(\ln(f_n(q)))}{dq} < 0$ for $q\geqslant2$ arbitrary natural number $n$. Therefore for any fixed number $n$ the function $f_n(q)$ decreases as $q$ increase, which means that $f_n(q)\leqslant f_n(2) = 0$, i.e., \eqref{lemm1:key_estimation} holds. From this and \eqref{lemm1:jensen},\eqref{lemm1:expectation_component} we obtain that $\forall\, q\geqslant2$
\begin{equation}\label{lemm1:pre-final}
    \EE[||e||_q^2] \overset{\eqref{lemm1:jensen}}{\leqslant} \left(n\EE[|e_2|^q]\right)^{\frac{2}{q}} \overset{\eqref{lemm1:expectation_component},\eqref{lemm1:key_estimation}}{\leqslant} (q-1)n^{\frac{2}{q}-1}.
\end{equation}
However, inequality \eqref{lemm1:pre-final} is useless when $q$ is big (with respect to $n$). Consider left hand side of \eqref{lemm1:pre-final} as function of $q$ and find its minimum for $q\geqslant2$. Consider $h_n(q) = \ln(q-1) + \left(\frac{2}{q}-1\right)\ln n$ (it is logarithm of the right hand side of \eqref{lemm1:pre-final}). Derivative of $h(q)$ is
\begin{equation*}
    \begin{array}{rl}
        \frac{dh(q)}{dq} = \frac{1}{q-1} -\frac{2\ln n}{q^2},\\
        \frac{1}{q-1} -\frac{2\ln n}{q^2} = 0,\\
        q^2-2q\ln n + 2\ln n = 0.
    \end{array}
\end{equation*}
If $n \geqslant 8$, then the point where the function obtains its minimum on the set $[2,+\infty)$ is $q_0 = \ln n\left(1+\sqrt{1-\frac{2}{\ln n}}\right)$ (for the case $n\leqslant 7$ it turns out that $q_0 = 2$; further without loss of generality we assume $n\geqslant8$). Therefore for all $q>q_0$ it is more useful to use the following estimation:
\begin{equation}\label{lemm1:pre_final_big_q}
    \begin{array}{rl}
        \EE[||e||_q^2] \overset{\circledOne}{<} \EE[||e||_{q_0}^2] \overset{\eqref{lemm1:pre-final}}{\leqslant}(q_0-1)n^{\frac{2}{q_0}-1}\overset{\circledTwo}{\leqslant}(2\ln n -1)n^{\frac{2}{\ln n}-1}\\
        = (2\ln n -1)e^2\frac{1}{n}\leqslant(16\ln n -8)\frac{1}{n}\leqslant(16\ln n -8)n^{\frac{2}{q}-1},
    \end{array}
\end{equation}
where $\circledOne$ is due to $\|e\|_q < \|e\|_{q_0}$ for $q > q_0$, $\circledTwo$ follows from $q_0 \leqslant 2\ln n,\, q_0 \geqslant \ln n$. Putting estimations \eqref{lemm1:pre-final} and \eqref{lemm1:pre_final_big_q} together we obtain \eqref{lemm1:expect_q_norm}. 

Now we are going to prove \eqref{lemm1:expect_inner_product}. Firstly, we want to estimate $\sqrt{\EE[\|e\|_q^4]}$. Due to probabilistic Jensen's inequality ($q\geqslant2$)
\begin{equation*}
    \begin{array}{rl}
        \EE[||e||_q^4] = \EE\left[\left(\left(\sum\limits_{k=1}^{n}|e_k|^q\right)^2\right)^{\frac{2}{q}}\right] \leqslant \left(\EE\left[\left(\sum\limits_{k=1}^{n}|e_k|^q\right)^2\right]\right)^{\frac{2}{q}}\\
        \overset{\circledOne}{\leqslant} \left(\EE\left[\left(n\sum\limits_{k=1}^{n}|e_k|^{2q}\right)\right]\right)^{\frac{2}{q}} \overset{\circledTwo}{=} \left(n^2\EE[|e_2|^{2q}]\right)^{\frac{2}{q}}\\
        \overset{\eqref{lemm1:expectation_component},\eqref{lemm1:key_estimation}}{\leqslant} n^{\frac{4}{q}}\left(\left(\frac{2q-1}{n}\right)^{\frac{2q}{2}}\right)^{\frac{2}{q}} = (2q-1)^{2}n^{\frac{4}{q}-2},
    \end{array}
\end{equation*}
where $\circledOne$ is because $\left(\sum\limits_{k=1}^{n} x_k\right)^2 \leqslant n\sum\limits_{k=1}^{n} x_k^2$ for $x_1,x_2,\ldots,x_n\in\R$ and $\circledTwo$ follows from that mathematical expectation is linear and components of the random vector $e$ are identically distributed. From this we obtain
\begin{equation}\label{lemm1:pre_final_q_norm_power4}
    \sqrt{\EE[||e||_q^4]} \leqslant (2q-1)n^{\frac{2}{q}-1}.
\end{equation}
Consider the right hand side of the inequality \eqref{lemm1:pre_final_q_norm_power4} as a function of $q$ and find its minimum for $q\geqslant2$. Consider $h_n(q) = \ln(2q-1) + \left(\frac{2}{q}-1\right)\ln n$ (logarithm of the right hand side \eqref{lemm1:pre_final_q_norm_power4}). Derivative of $h(q)$ is
\begin{equation*}
    \begin{array}{rl}
        \frac{dh(q)}{dq} = \frac{2}{2q-1} -\frac{2\ln n}{q^2},\\
        \frac{2}{2q-1} -\frac{2\ln n}{q^2} = 0,\\
        q^2-2q\ln n + \ln n = 0.
    \end{array}
\end{equation*}
If $n \geqslant 3$, the the point where the function obtains its minimum on the set $[2,+\infty)$ is $q_0 = \ln n\left(1+\sqrt{1-\frac{1}{\ln n}}\right)$ (for the case $n\leqslant 2$ it turns out that $q_0 = 2$; further without loss of generality we assume that $n\geqslant3$). Therefore for all $q>q_0$:
\begin{equation}\label{lemm1:pre_final_big_q_power4}
    \begin{array}{rl}
        \sqrt{\EE[||e||_q^4]} \overset{\circledOne}{<} \sqrt{\EE[||e||_{q_0}^4]} \overset{\eqref{lemm1:pre_final_q_norm_power4}}{\leqslant}(2q_0-1)n^{\frac{2}{q_0}-1}\overset{\circledTwo}{\leqslant}(4\ln n -1)n^{\frac{2}{\ln n}-1}\\
        = (4\ln n -1)e^2\frac{1}{n}\leqslant(32\ln n -8)\frac{1}{n}\leqslant(32\ln n -8)n^{\frac{2}{q}-1},
    \end{array}
\end{equation}
where $\circledOne$ is due to $\|e\|_q < \|e\|_{q_0}$ for $q > q_0$, $\circledTwo$ follows from $q_0 \leqslant 2\ln n,\, q_0 \geqslant \ln n$. Putting estimations \eqref{lemm1:pre_final_q_norm_power4} and \eqref{lemm1:pre_final_big_q_power4} together we get inequality
\begin{equation}\label{lemm1:expect_q_norm_power4}
    \sqrt{\EE[||e||_q^4]}\leqslant \min\{2q-1,32\ln n -8\}n^{\frac{2}{q}-1}.
\end{equation}

Now we are going to find $\EE[\langle s,\,e\rangle^4]$, where $s\in\R^n$ is some vector. Let $S_n(r)$ be a surface area of $n$-dimensional Euclidean sphere with radius $r$ and $d\sigma(e)$ be unnormalized uniform measure on $n$-dimensional Euclidean sphere. From this it follows that $S_n(r) = S_n(1)r^{n-1},\, \frac{S_{n-1}(1)}{S_n(1)} = \frac{n-1}{n\sqrt{\pi}}\frac{\Gamma(\frac{n+2}{2})}{\Gamma(\frac{n+1}{2})}$. Besides, let $\varphi$ be the angle between $s$ and $e$.
Then
\begin{equation}\label{lemm1:inner_product_power4}
    \begin{array}{rl}
        \EE[\langle s,\, e\rangle^4] = \frac{1}{S_n(1)}\int\limits_{S}\langle s,\, e\rangle^4d\sigma(\varphi) = \frac{1}{S_n(1)}\int\limits_0^\pi||s||_2^4\cos^3\varphi S_{n-1}(\sin\varphi)d\varphi\\
        = ||s||_2^4\frac{S_{n-1}(1)}{S_n(1)}\int\limits_{0}^\pi\cos^4\varphi\sin^{n-2}\varphi d\varphi = ||s||_2^4\cdot\frac{n-1}{n\sqrt{\pi}}\frac{\Gamma(\frac{n+2}{2})}{\Gamma(\frac{n+1}{2})}\int\limits_{0}^\pi\cos^4\varphi\sin^{n-2}\varphi d\varphi.
    \end{array}
\end{equation}
Compute the integral:
\begin{equation*}
    \begin{array}{rl}
        \int\limits_0^\pi\cos^4\varphi\sin^{n-2}\varphi d\varphi = 2\int\limits_0^{\frac{\pi}{2}}\cos^4\varphi\sin^{n-2}\varphi d\varphi = / t=\sin^2\varphi/\\
        = \int\limits_0^{\frac{\pi}{2}}t^{\frac{n-3}{2}}(1-t)^{\frac{3}{2}}dt = B(\frac{n-1}{2},\frac{5}{2}) = \frac{\Gamma(\frac{5}{2})\Gamma(\frac{n-1}{2})}{\Gamma(\frac{n+4}{2})} = \frac{\frac{3}{2}\cdot\frac{1}{2}\Gamma(\frac{1}{2})\Gamma(\frac{n-1}{2})}{\frac{n+2}{2}\cdot\Gamma(\frac{n+2}{2})} = \frac{3}{n+2}\cdot\frac{\sqrt{\pi}\Gamma(\frac{n-1}{2})}{2\Gamma(\frac{n+2}{2})}.
    \end{array}
\end{equation*}
From this and \eqref{lemm1:inner_product_power4} we obtain
\begin{equation}\label{lemm1:inner_product_power4_final}
    \begin{array}{rl}
        \EE[\langle s,\, e\rangle^4] = ||s||_2^4\cdot\frac{n-1}{n\sqrt{\pi}}\frac{\Gamma(\frac{n+2}{2})}{\Gamma(\frac{n+1}{2})}\cdot\frac{3}{n+2}\cdot\frac{\sqrt{\pi}\Gamma(\frac{n-1}{2})}{2\Gamma(\frac{n+2}{2})}\\
        = ||s||_2^4\cdot\frac{3(n-1)}{2n(n+2)}\cdot\frac{\Gamma(\frac{n-1}{2})}{\frac{n-1}{2}\Gamma(\frac{n-1}{2})}
        = \frac{3||s||_2^4}{n(n+2)} \overset{\circledOne}{\leqslant} \frac{3||s||_2^4}{n^2}.
    \end{array}
\end{equation}

To prove \eqref{lemm1:expect_inner_product}, it remains to use \eqref{lemm1:expect_q_norm_power4}, \eqref{lemm1:inner_product_power4_final} and Cauchy-Schwartz inequality ($(\EE[XY])^2\leqslant\EE[X^2]\cdot\EE[Y^2]$):
\begin{equation*}
    \begin{array}{rl}
        \EE[\langle s,\,e\rangle^2 ||e||_q^2] \overset{\circledOne}{\leqslant} \sqrt{\EE[\langle s,\,e\rangle^4]\cdot\EE[||e||_q^4]} \leqslant \sqrt{3}||s||_2^2\min\{2q-1,32\ln n -8\}n^{\frac{2}{q}-2}.
    \end{array}
\end{equation*}

\section{Technical Results}
\begin{lemma}\label{stoh:technical_lemma}
    Let $a_0,\ldots,a_{N-1}, b, R_1,\ldots, R_{N-1}$ be non-negative numbers such that
    \begin{equation}\label{technical_lemma_assumption}
        R_{l} \leqslant \sqrt{2}\cdot\sqrt{\left(\sum\limits_{k=0}^{l-1}a_k + b\sum\limits_{k=1}^{l-1}\alpha_{k+1}R_k \right)}\quad l=1,\ldots,N,
    \end{equation}
    where $\alpha_{k+1} = \frac{k+2}{96n^2\rho_nL_2}$ for all $k\in\NN$. Then for $l=1,\ldots,N$
    \begin{equation}\label{technical_inequality_for_induction}
        \sum\limits_{k=0}^{l-1}a_k + b\sum\limits_{k=1}^{l-1} \alpha_{k+1}R_k \leqslant \left(\sqrt{\sum\limits_{k=0}^{l-1}a_k} + \sqrt{2}b\cdot\frac{l^2}{96n^2\rho_nL_2}\right)^2.
    \end{equation}
\end{lemma}
\begin{proof}
For $l=1$ it is trivial inequality. Assume that \eqref{technical_inequality_for_induction} holds for some $l < N$ and prove it for $l+1$. From the induction assumption and \eqref{technical_lemma_assumption} we obtain
\begin{equation}\label{technical_distance_estimation_3}
    \begin{array}{rl}
        R_l \leqslant \sqrt{2}\left(\sqrt{\sum\limits_{k=0}^{l-1}a_k} + \sqrt{2}b\cdot\frac{l^2}{96n^2\rho_nL_2}\right),
    \end{array}
\end{equation}
whence
\begin{equation*}
    \begin{array}{rl}
        \sum\limits_{k=0}^{l}a_k + b\sum\limits_{k=1}^{l}\alpha_{k+1}R_k = \sum\limits_{k=0}^{l-1}a_k + b\sum\limits_{k=1}^{l-1}\alpha_{k+1}R_k + a_l + b\alpha_{l+1}R_{l}\\
        \overset{\circledOne}{\leqslant} \left(\sqrt{\sum\limits_{k=0}^{l-1}a_k} + \sqrt{2}b\cdot\frac{l^2}{96n^2\rho_nL_2}\right)^2 + a_l
        + \sqrt{2}b\alpha_{l+1}\left(\sqrt{\sum\limits_{k=0}^{l-1}a_k} + \sqrt{2}b\cdot\frac{l^2}{96n^2\rho_nL_2}\right)\\
        = \sum\limits_{k=0}^{l}a_k + 2\sqrt{\sum\limits_{k=0}^{l-1}a_k}\cdot \sqrt{2}b\frac{l^2}{96n^2\rho_nL_2} + 2b^2\frac{l^4}{(96n^2\rho_nL_2)^2}
        + \sqrt{2}b\alpha_{l+1}\left(\sqrt{\sum\limits_{k=0}^{l-1}a_k} + \sqrt{2}b\cdot\frac{l^2}{96n^2\rho_nL_2}\right)\\
        = \sum\limits_{k=0}^{l}a_k + 2\sqrt{\sum\limits_{k=0}^{l-1}a_k}\cdot\sqrt{2}b\left(\frac{l^2}{96n^2\rho_nL_2} + \frac{\alpha_{l+1}}{2}\right) + 2b^2\left(\frac{l^4}{(96n^2\rho_nL_2)^2} + \alpha_{l+1}\cdot\frac{l^2}{96n^2\rho_nL_2}\right)\\
        \overset{\circledTwo}{\leqslant} \sum\limits_{k=0}^{l}a_k + 2\sqrt{\sum\limits_{k=0}^{l}a_k}\cdot \sqrt{2}b\frac{(l+1)^2}{96n^2\rho_nL_2} + 2b^2\frac{(l+1)^4}{(96n^2\rho_nL_2)^2}
        = \left(\sqrt{\sum\limits_{k=0}^{l}a_k} + \sqrt{2}b\cdot\frac{(l+1)^2}{96n^2\rho_nL_2}\right)^2,
    \end{array}    
\end{equation*}
where $\circledOne$ follows from the induction assumption and \eqref{technical_distance_estimation_3}, $\circledTwo$ is because $\sum\limits_{k=0}^{l-1}a_k\leqslant \sum\limits_{k=0}^{l}a_k$ and
\begin{equation*}
    \begin{array}{rl}
        \frac{l^2}{96n^2\rho_nL_2} + \frac{\alpha_{l+1}}{2} = \frac{2l^2 + l + 2}{192n^2\rho_nL_2} \leqslant \frac{(l+1)^2}{96n^2\rho_nL_2},\\
        \frac{l^4}{(96n^2\rho_nL_2)^2} + \alpha_{l+1}\cdot\frac{l^2}{96n^2\rho_nL_2} \leqslant \frac{l^4 + (l+2)l^2}{(96n^2\rho_nL_2)^2} \leqslant \frac{(l+1)^4}{(96n^2\rho_nL_2)^2}.
    \end{array}
\end{equation*}
\end{proof}

\begin{lemma}\label{stoh:technical_lemma_non_acc}
    Let $a_0,\ldots,a_{N-1}, b, R_1,\ldots, R_{N-1}$ be non-negative numbers such that
    \begin{equation}\label{technical_lemma_assumption_non_acc}
        R_{l} \leqslant \sqrt{2}\cdot\sqrt{\left(\sum\limits_{k=0}^{l-1}a_k + b\alpha\sum\limits_{k=1}^{l-1}R_k \right)}\quad l=1,\ldots,N.
    \end{equation}
    Then for $l=1,\ldots,N$
    \begin{equation}\label{technical_inequality_for_induction_non_acc}
        \sum\limits_{k=0}^{l-1}a_k + b\alpha\sum\limits_{k=1}^{l-1} R_k \leqslant \left(\sqrt{\sum\limits_{k=0}^{l-1}a_k} + \sqrt{2}b\alpha l\right)^2.
    \end{equation}
\end{lemma}
\begin{proof}
For $l=1$ it is trivial inequality. Assume that \eqref{technical_inequality_for_induction_non_acc} holds for some $l < N$ and prove it for $l+1$. From the induction assumption and \eqref{technical_lemma_assumption_non_acc} we obtain
\begin{equation}\label{technical_distance_estimation_3_non_acc}
    \begin{array}{rl}
        R_l \leqslant \sqrt{2}\left(\sqrt{\sum\limits_{k=0}^{l-1}a_k} + \sqrt{2}b\alpha l\right),
    \end{array}
\end{equation}
whence
\begin{equation*}
    \begin{array}{rl}
        \sum\limits_{k=0}^{l}a_k + b\alpha\sum\limits_{k=1}^{l}R_k = \sum\limits_{k=0}^{l-1}a_k + b\alpha\sum\limits_{k=1}^{l-1}R_k + a_l + b\alpha R_{l}\\
        \overset{\circledOne}{\leqslant} \left(\sqrt{\sum\limits_{k=0}^{l-1}a_k} + \sqrt{2}b\alpha l\right)^2 + a_l
        + \sqrt{2}b\alpha\left(\sqrt{\sum\limits_{k=0}^{l-1}a_k} + \sqrt{2}b\alpha l\right)\\
        = \sum\limits_{k=0}^{l}a_k + 2\sqrt{\sum\limits_{k=0}^{l-1}a_k}\cdot \sqrt{2}b\alpha l + 2b^2\alpha^2 l^2
        + \sqrt{2}b\alpha\left(\sqrt{\sum\limits_{k=0}^{l-1}a_k} + \sqrt{2}b\alpha l\right)\\
        = \sum\limits_{k=0}^{l}a_k + 2\sqrt{\sum\limits_{k=0}^{l-1}a_k}\cdot\sqrt{2}b\alpha\left(l + \frac{1}{2}\right) + 2b^2\alpha^2\left(l^2 + l\right)\\
        \overset{\circledTwo}{\leqslant} \sum\limits_{k=0}^{l}a_k + 2\sqrt{\sum\limits_{k=0}^{l}a_k}\cdot \sqrt{2}b\alpha (l+1) + 2b^2\alpha^2 (l+1)^2
        = \left(\sqrt{\sum\limits_{k=0}^{l}a_k} + \sqrt{2}b\alpha (l+1)\right)^2,
    \end{array}    
\end{equation*}
where $\circledOne$ follows from the induction assumption and \eqref{technical_distance_estimation_3_non_acc}, $\circledTwo$ is because $\sum\limits_{k=0}^{l-1}a_k\leqslant \sum\limits_{k=0}^{l}a_k$.
\end{proof}

{
\section{Parameters tuning}\label{sec:param_tunning}
In our analysis it is needed to choose $\alpha_{k+1} = \frac{k+2}{96n^2\rho_n L_2}$ for ARDD and $\alpha = \frac{1}{48n\rho_n L_2}$. However, one can tune these parameters in order to achieve better convergence rate in practice. In our experiments we choose $\alpha_{k+1} = \gamma\cdot \frac{k+2}{96n^2\rho_n L_2}$, $\alpha = \gamma\cdot\frac{1}{48n\rho_n L_2}$ and tune numerical factor $\gamma$. In \cite{ghadimi2013stochastic} authors prove convergence results for stepsize\footnote{
If $\sigma = 0$, then one should ignore the second term in the minimum.} $\alpha = \frac{1}{\sqrt{n+4}}\min\left\{\frac{1}{4L\sqrt{n+4}}, \frac{\tilde{D}}{\sigma \sqrt{N}}\right\}$ where $\tilde{D}$ is some numerical constant, therefore, in our experiments with RSGD we use stepsizes $\alpha = \gamma\cdot\frac{1}{\sqrt{n+4}}\min\left\{\frac{1}{4L\sqrt{n+4}}, \frac{1}{\sqrt{N}}\right\}$ where we also tune numerical factor $\gamma$.
\subsection{Nesterov's function}
One can find our numerical results with tuning stepsizes for each method in Figures~\ref{fig:nesterov_tuning_n_100}-\ref{fig:nesterov_tuning_n_5000}.
\begin{figure}[h]
    \centering
    \includegraphics[width=0.32\textwidth]{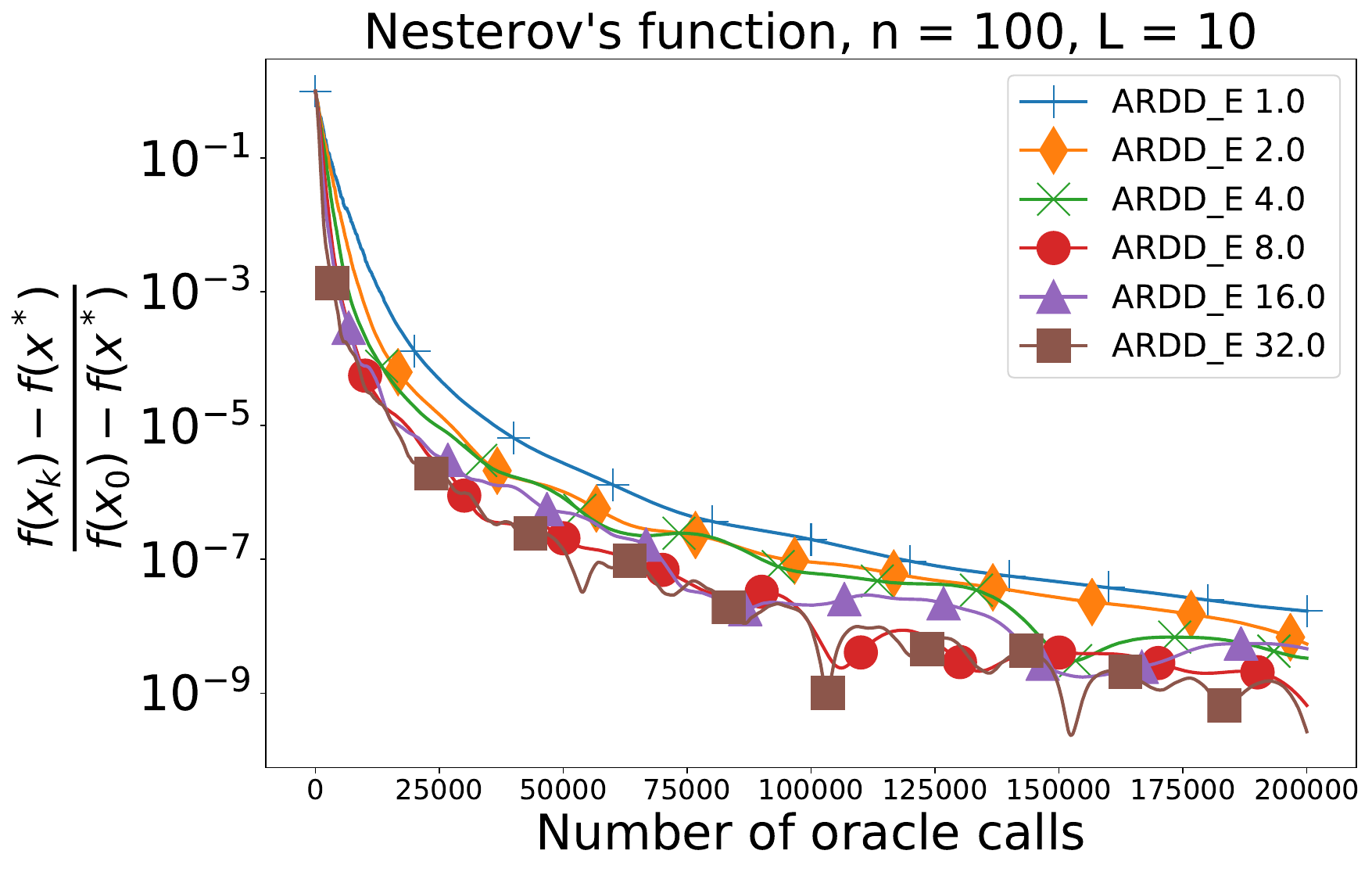}
    \includegraphics[width=0.32\textwidth]{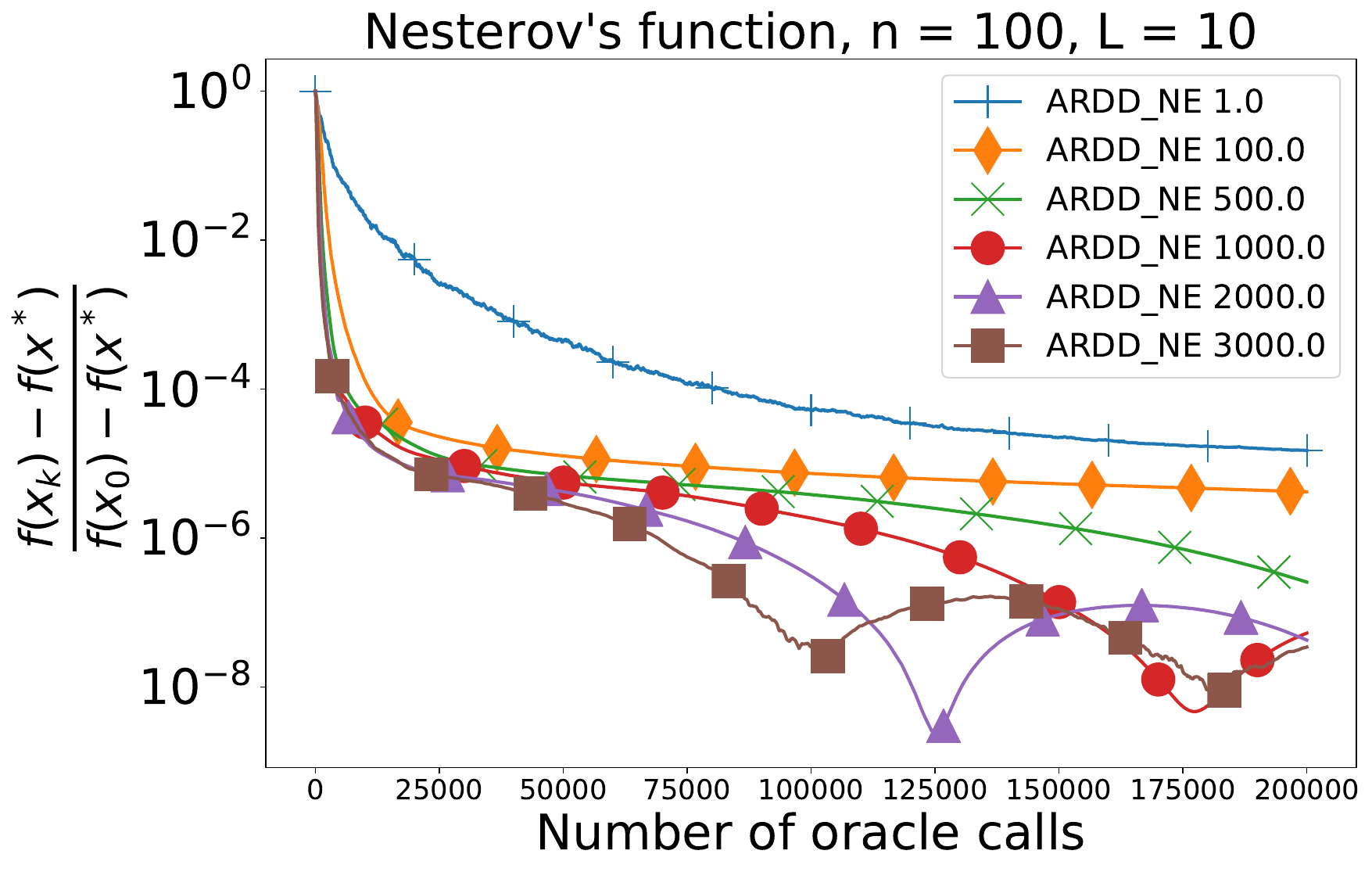}
    \includegraphics[width=0.32\textwidth]{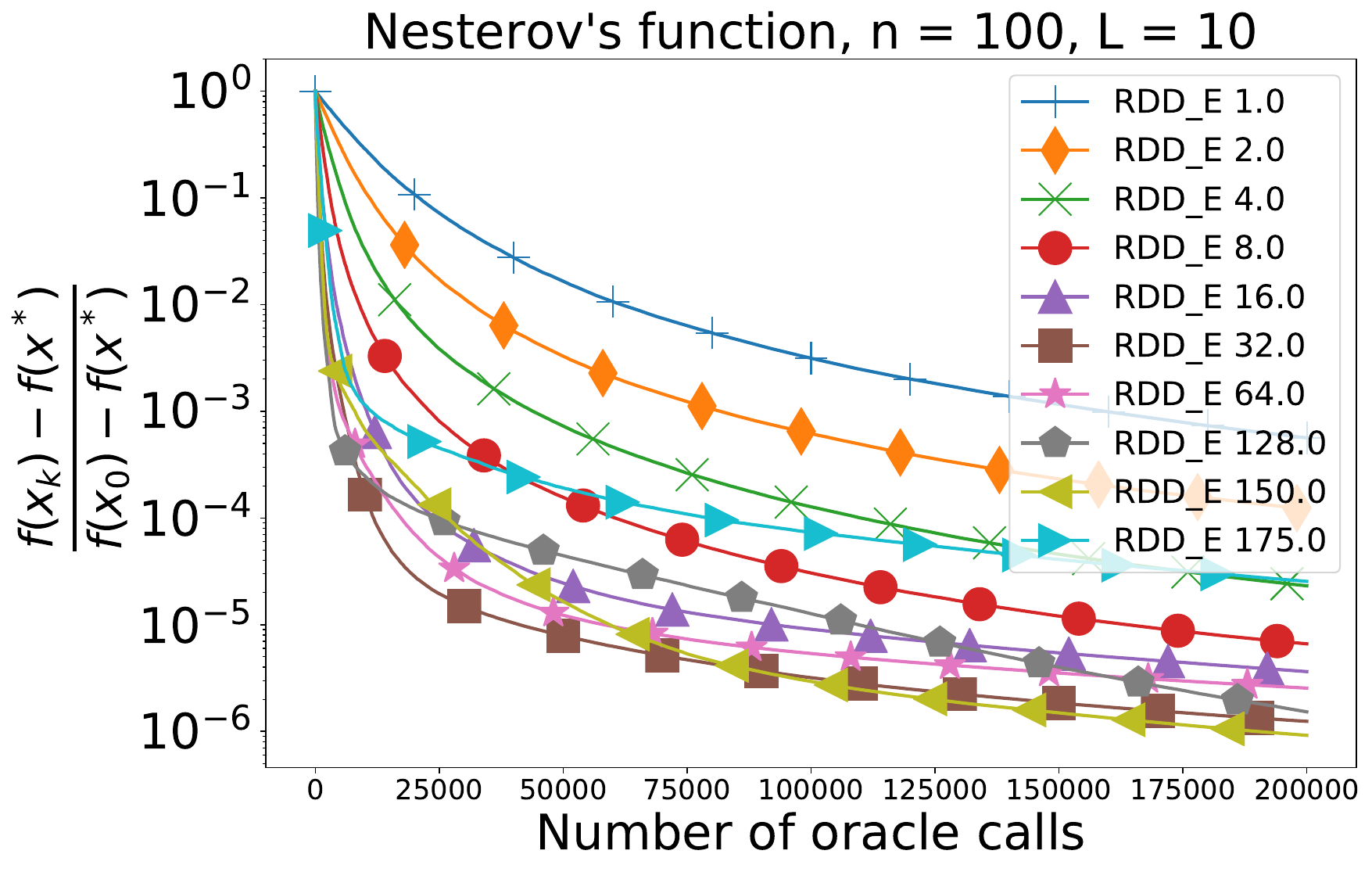}
    \includegraphics[width=0.32\textwidth]{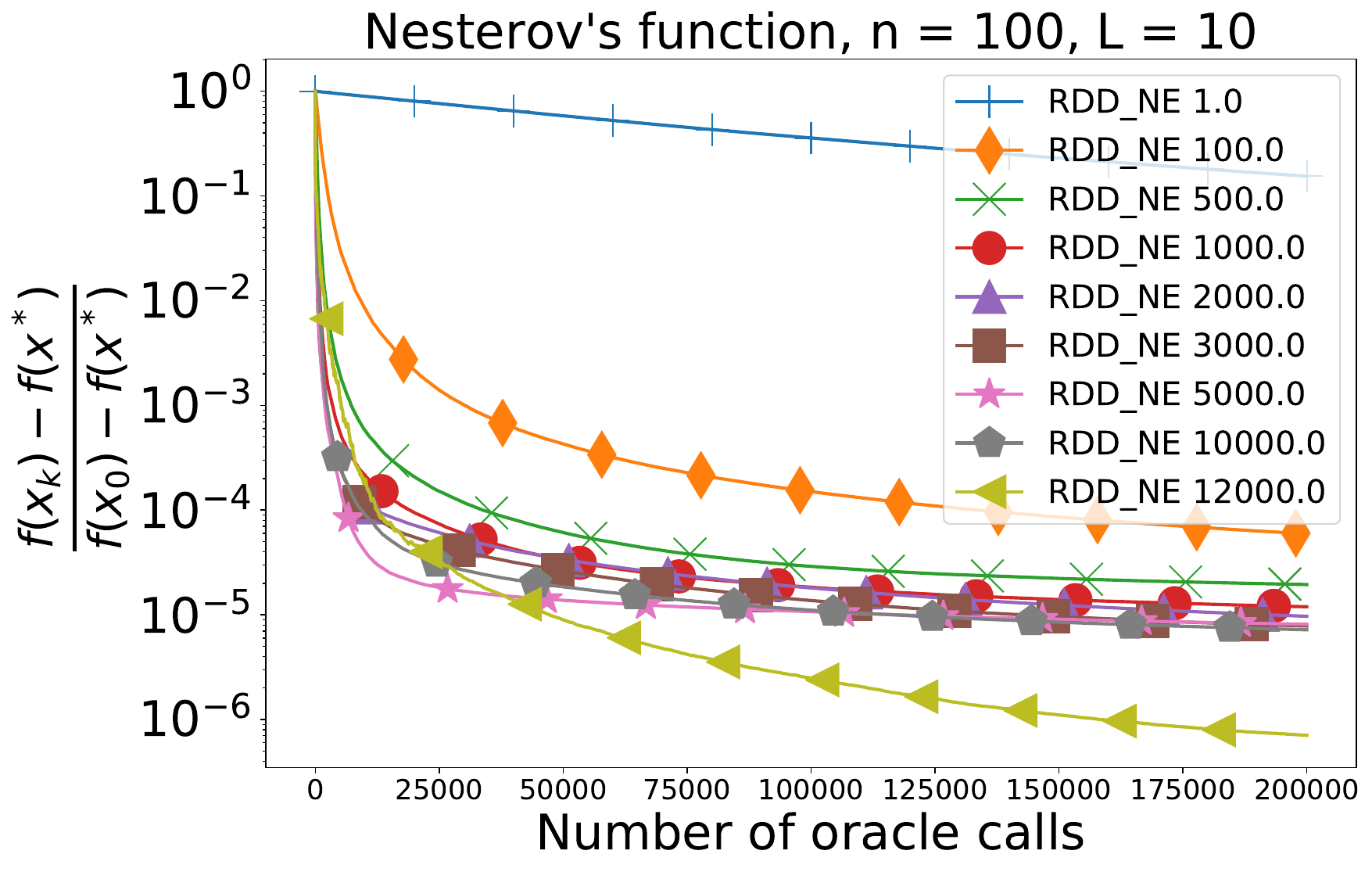}
    \includegraphics[width=0.32\textwidth]{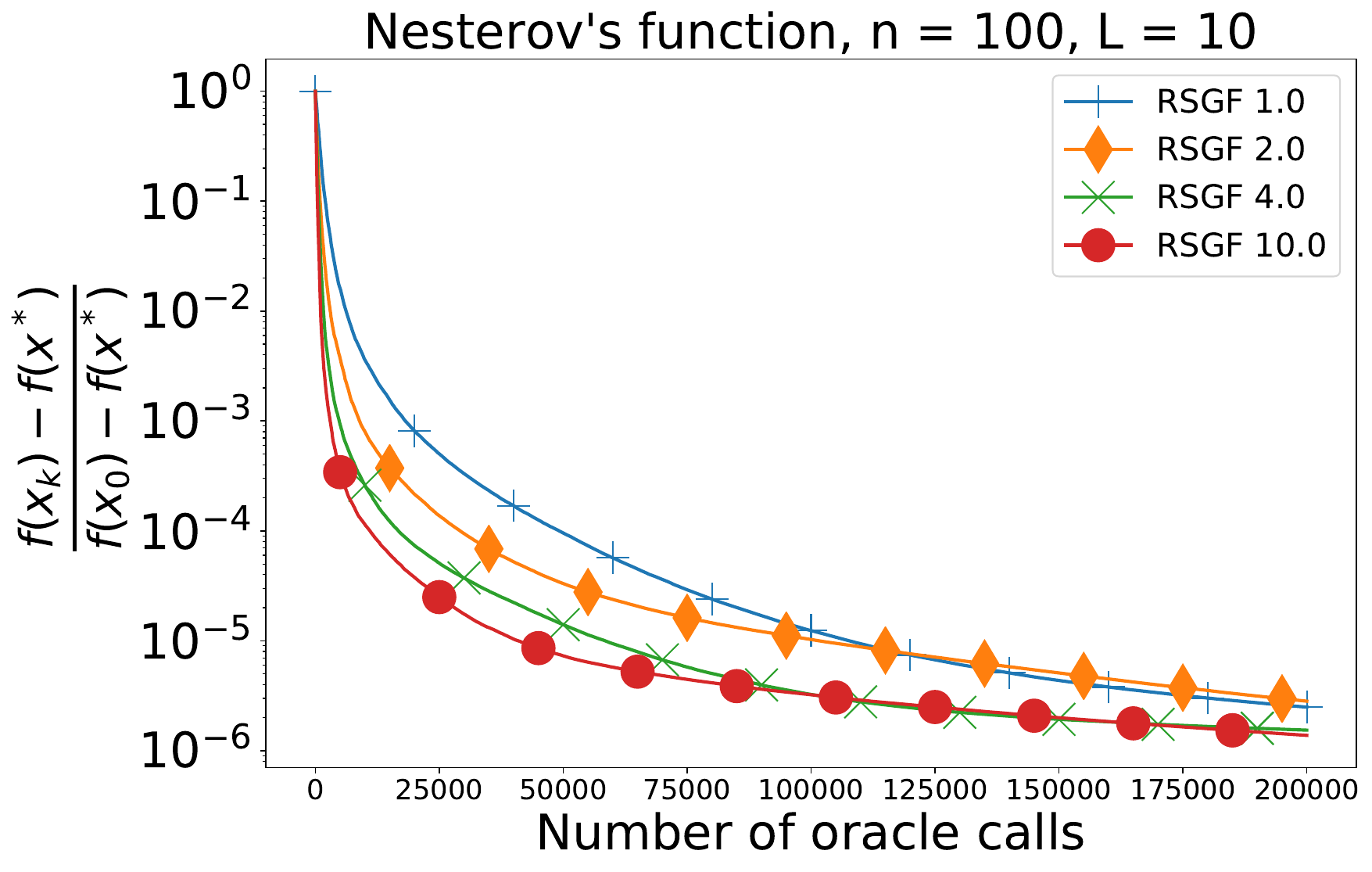}
    \caption{Stepsize tuning for ARDD, RDD and RSGF applied to minimize Nesterov's function \eqref{eq:nesterov_func}. We use {\_}E and {\_}NE to define $\ell_2$ and $\ell_1$ proximal setups respectively (see \eqref{eq:dp1} and \eqref{eq:dp2} for the details). Numbers in labels in upper right corners denote different choices of $\gamma$ that are used.}
    \label{fig:nesterov_tuning_n_100}
\end{figure}
\begin{figure}[h]
    \centering
    \includegraphics[width=0.32\textwidth]{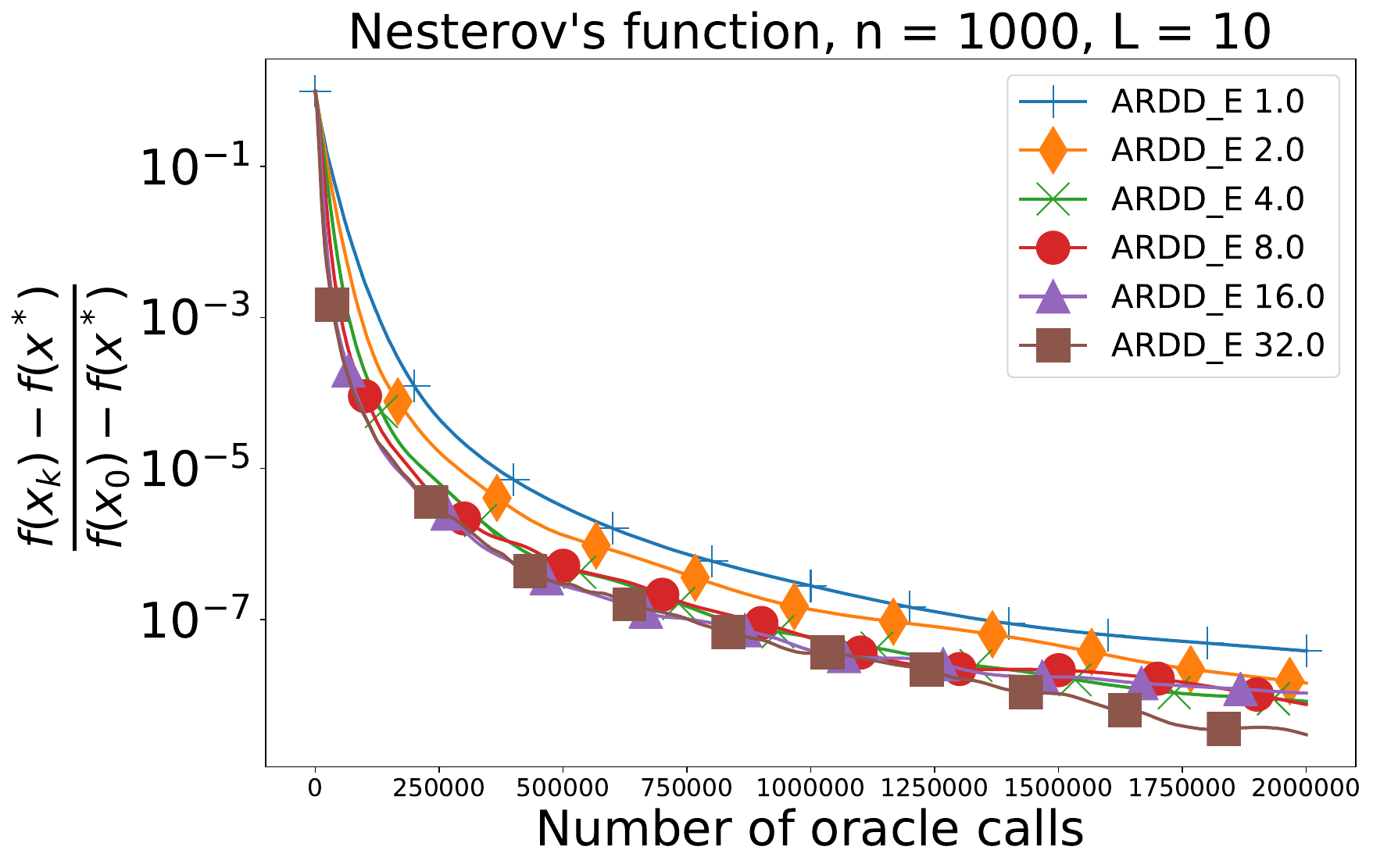}
    \includegraphics[width=0.32\textwidth]{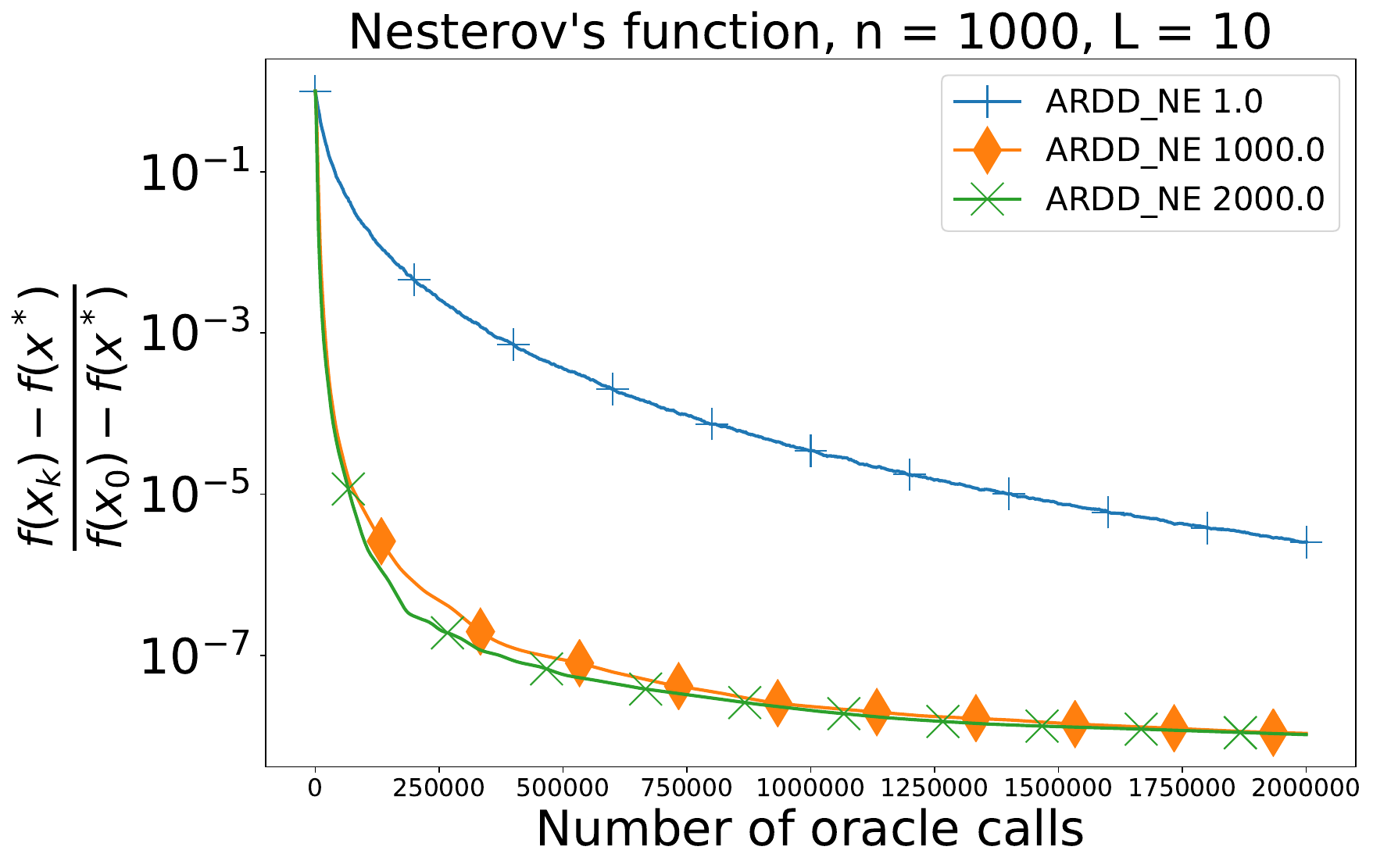}
    \includegraphics[width=0.32\textwidth]{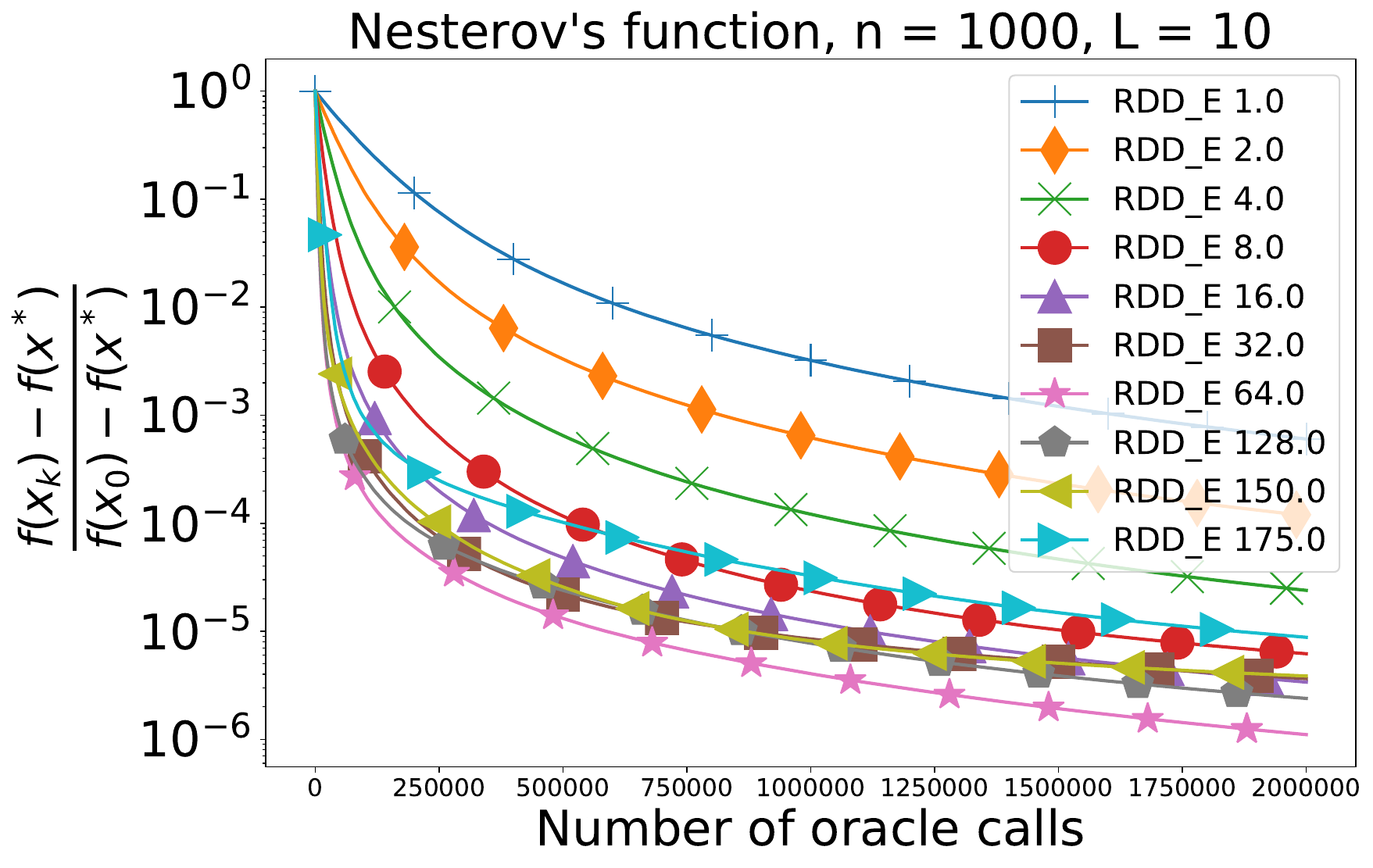}
    \includegraphics[width=0.32\textwidth]{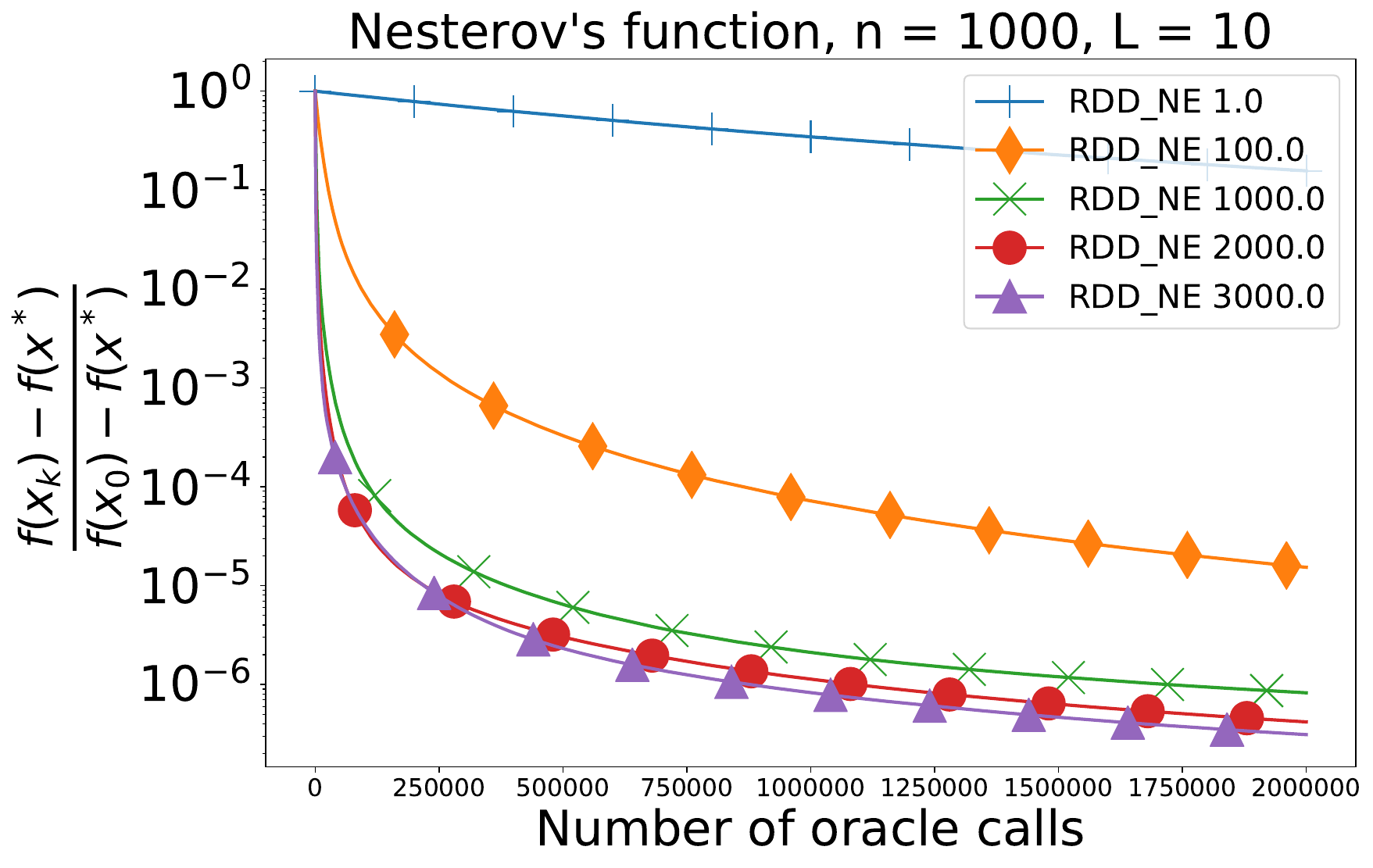}
    \includegraphics[width=0.32\textwidth]{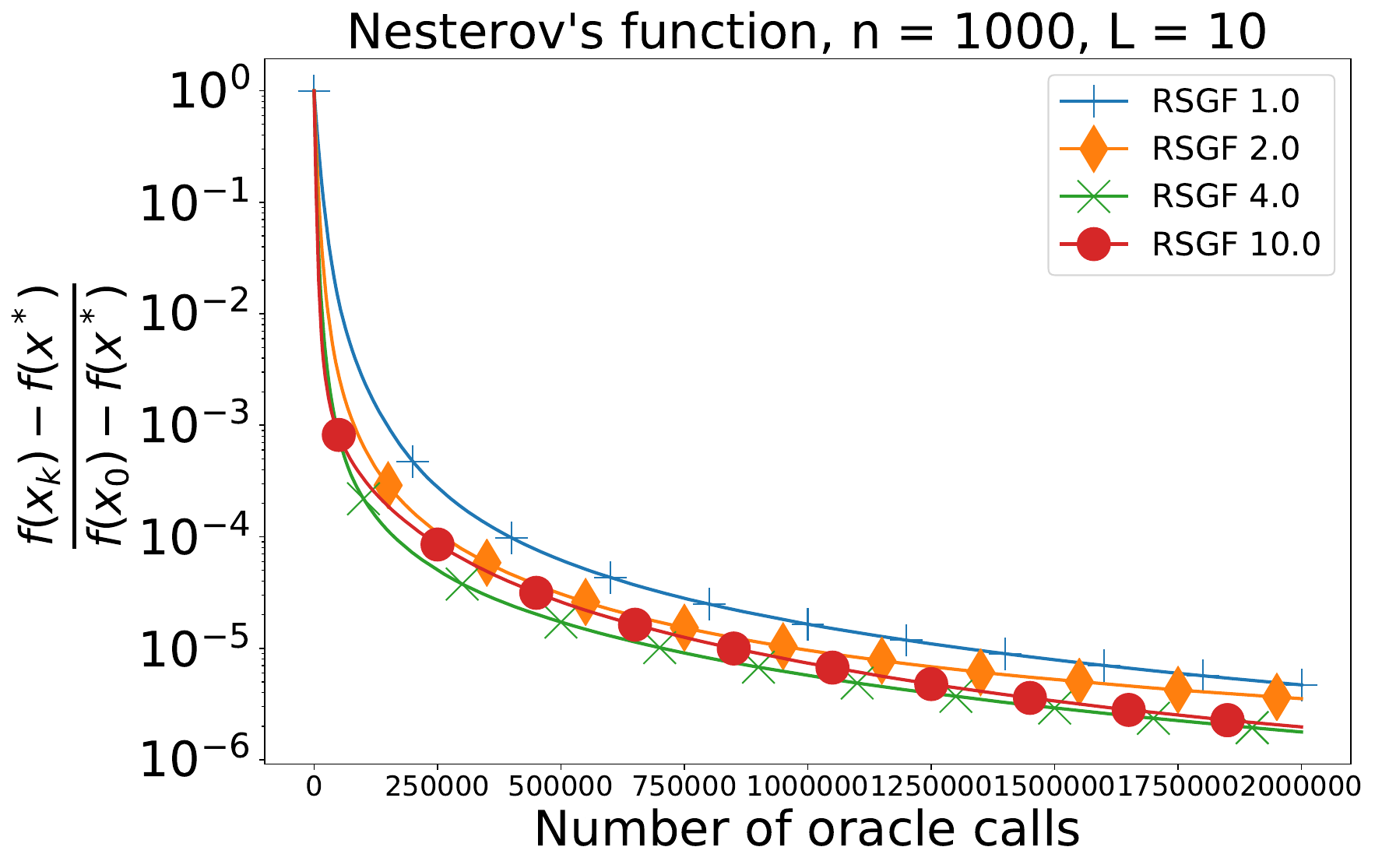}
    \caption{Stepsize tuning for ARDD, RDD and RSGF applied to minimize Nesterov's function \eqref{eq:nesterov_func}. We use {\_}E and {\_}NE to define $\ell_2$ and $\ell_1$ proximal setups respectively (see \eqref{eq:dp1} and \eqref{eq:dp2} for the details). Numbers in labels in upper right corners denote different choices of $\gamma$ that are used.}
    \label{fig:nesterov_tuning_n_1000}
\end{figure}
\begin{figure}[h]
    \centering
    \includegraphics[width=0.32\textwidth]{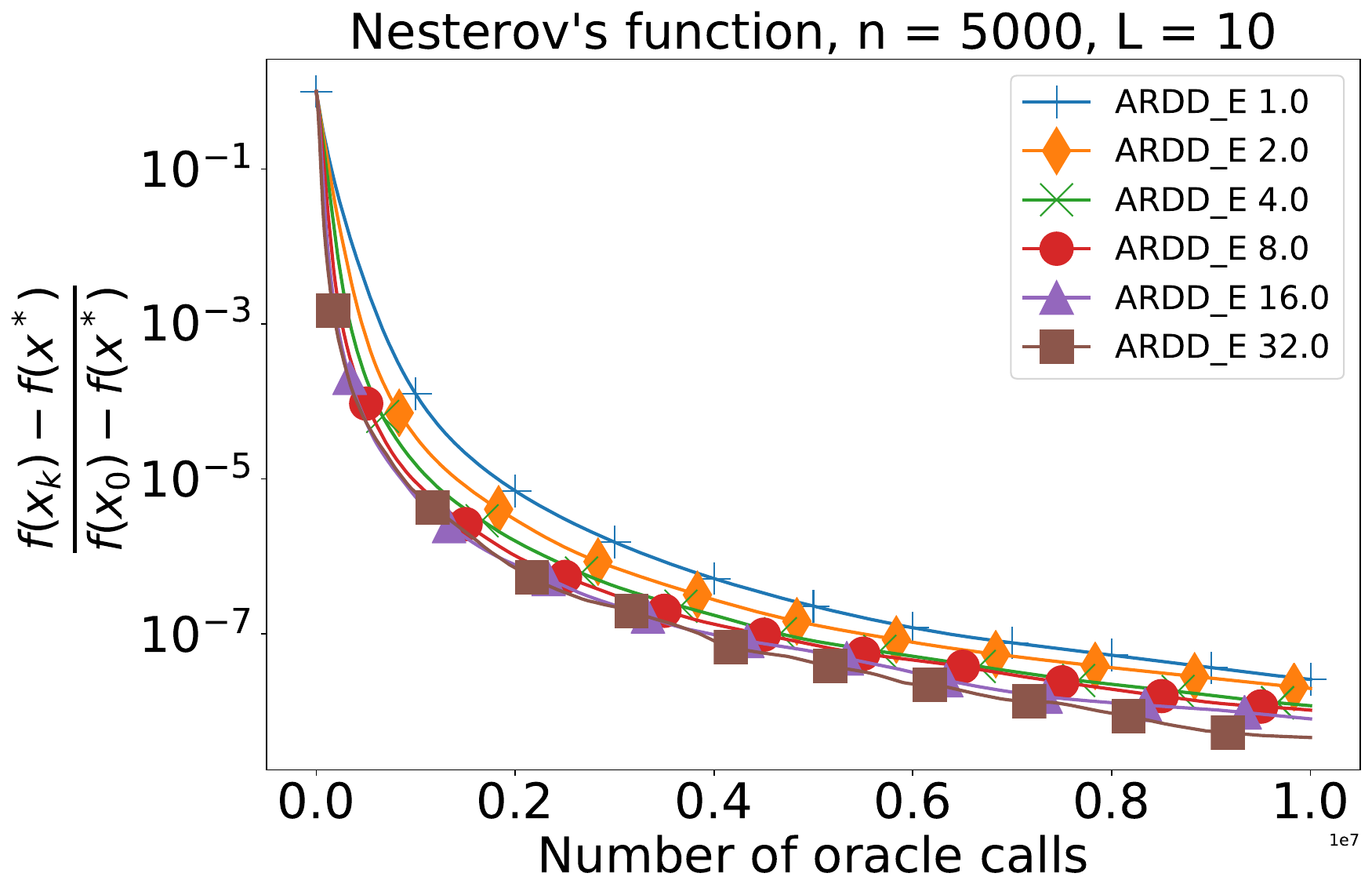}
    \includegraphics[width=0.32\textwidth]{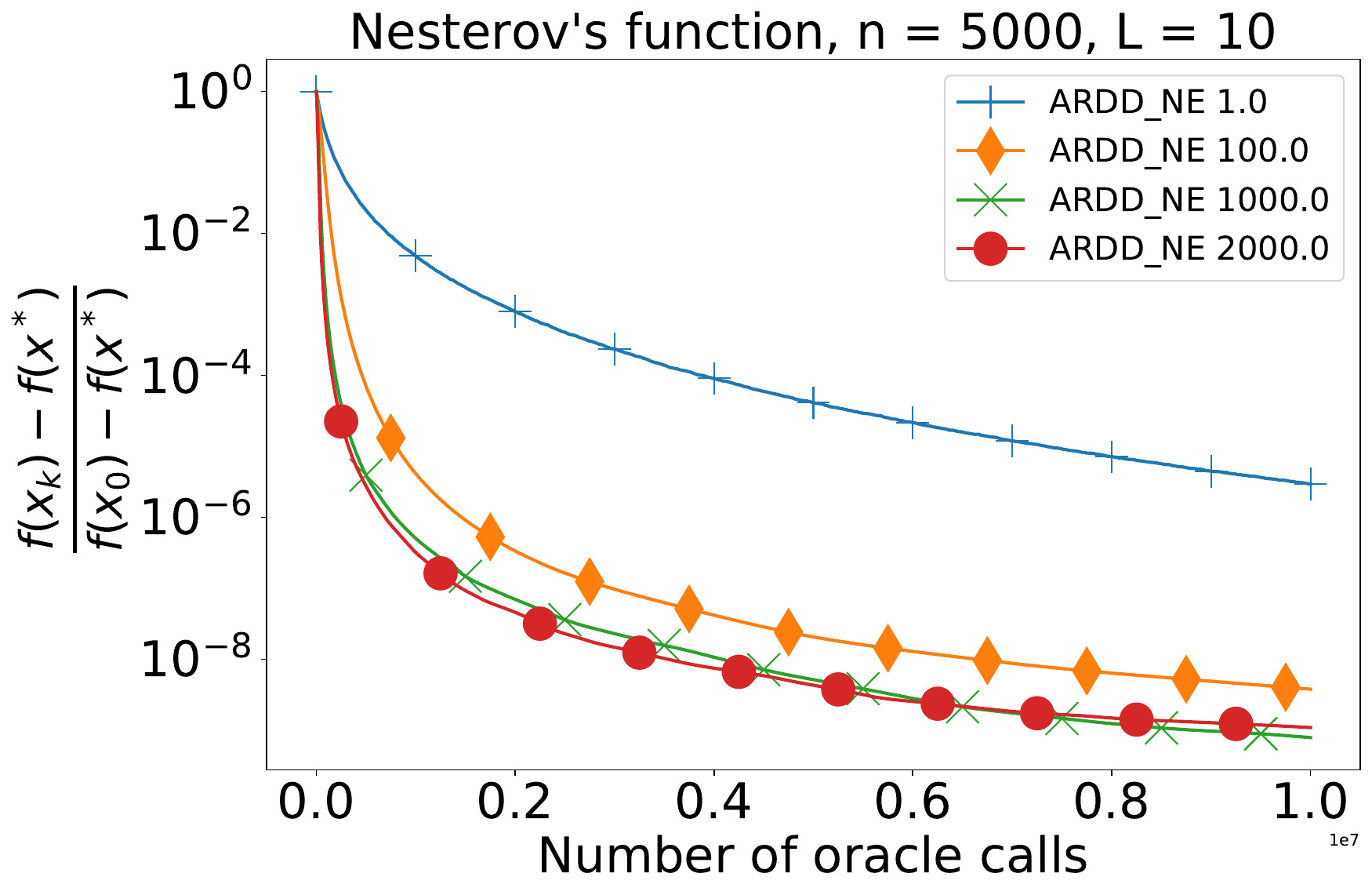}
    \includegraphics[width=0.32\textwidth]{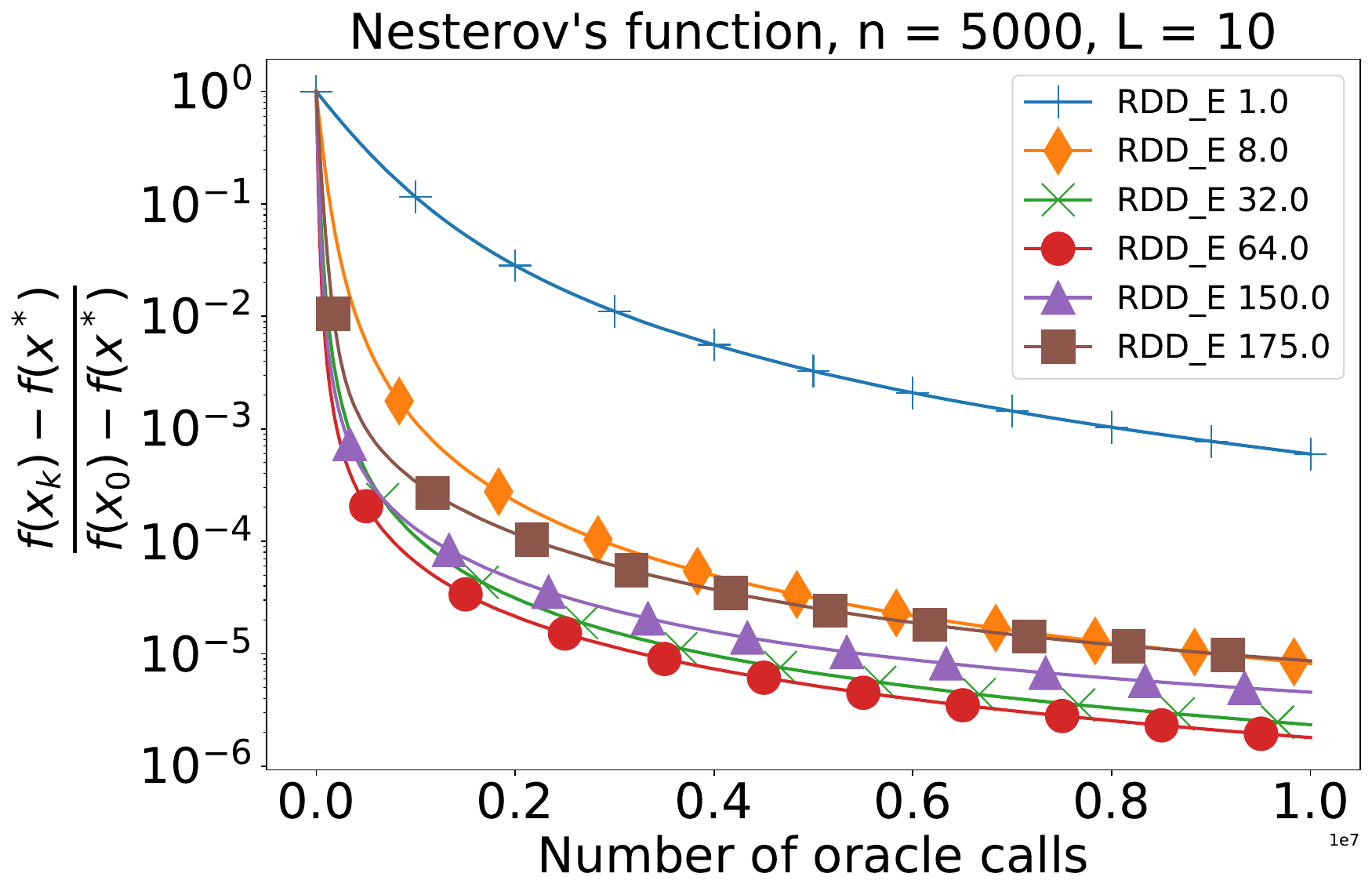}
    \includegraphics[width=0.32\textwidth]{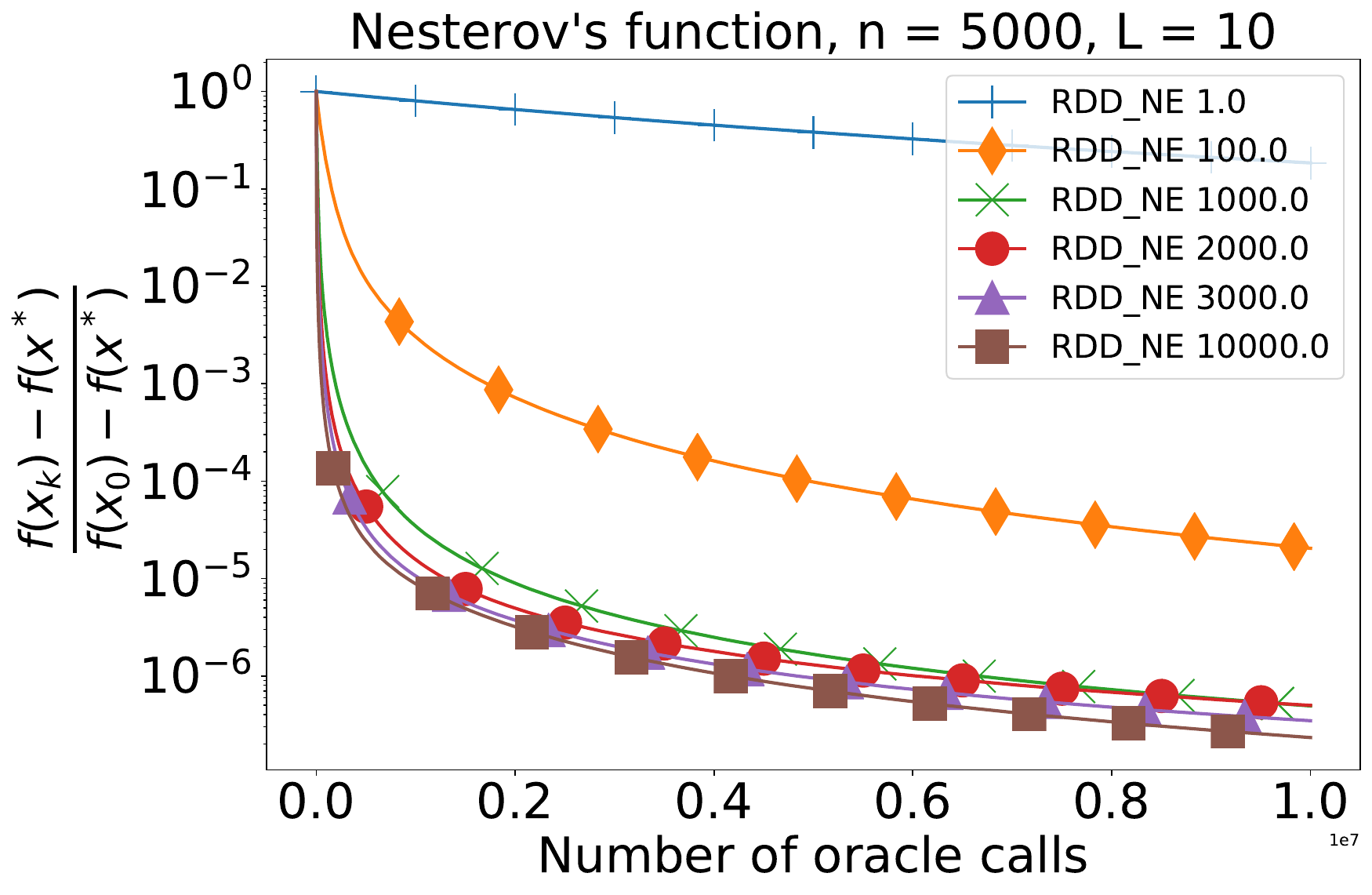}
    \includegraphics[width=0.32\textwidth]{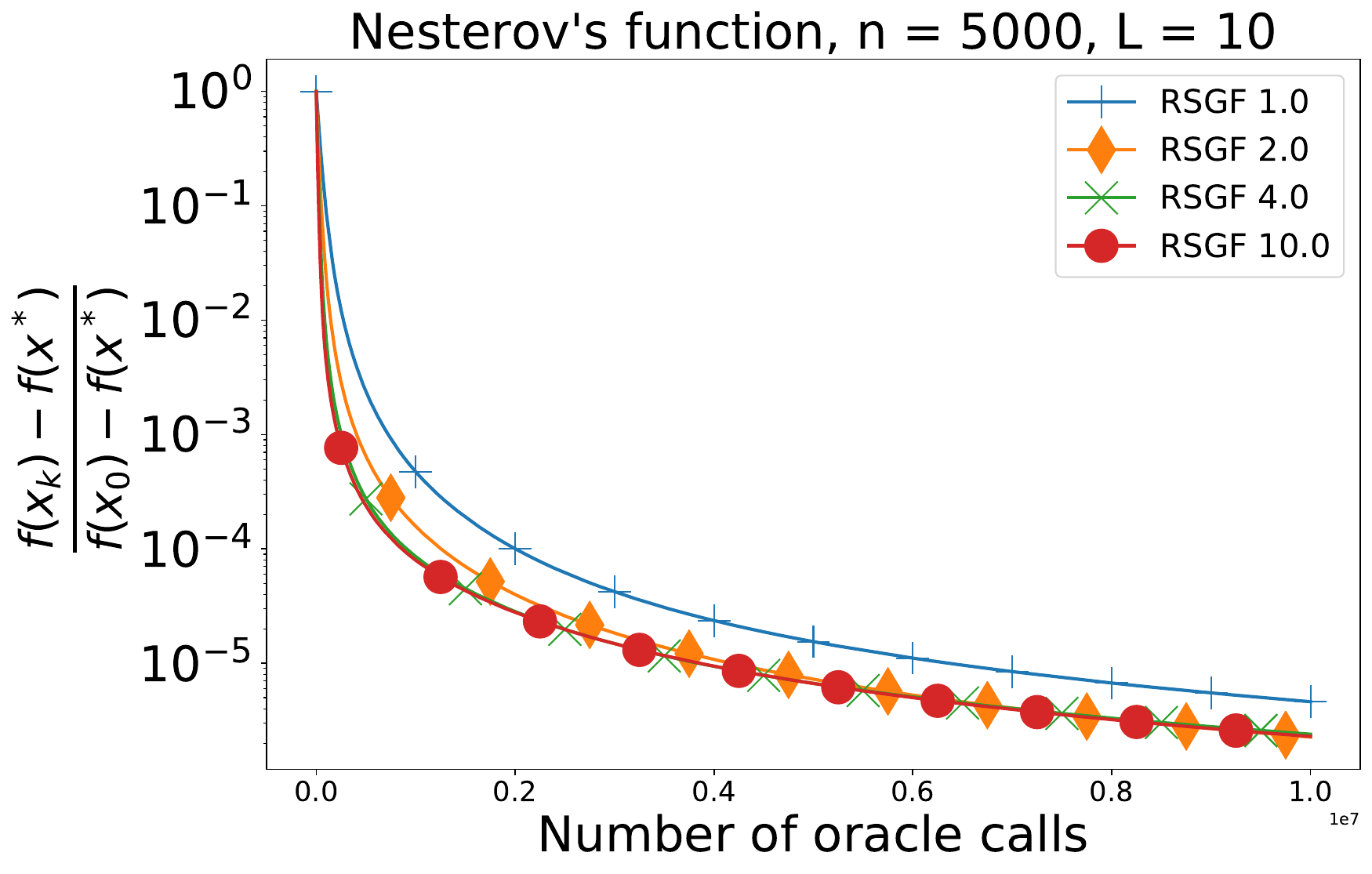}
    \caption{Stepsize tuning for ARDD, RDD and RSGF applied to minimize Nesterov's function \eqref{eq:nesterov_func}. We use {\_}E and {\_}NE to define $\ell_2$ and $\ell_1$ proximal setups respectively (see \eqref{eq:dp1} and \eqref{eq:dp2} for the details). Number of oracle calls is divided by $10^7$. Numbers in labels in upper right corners denote different choices of $\gamma$ that are used.}
    \label{fig:nesterov_tuning_n_5000}
\end{figure}
Our tests with Nesterov's function show that for this problem ARDD{\_}E and RDD{\_} work better with $\gamma \in [32,64]$ and RSGF shows the best performance with $\gamma \in [4,10]$. Interestingly, ARDD and RDD with $p=1$ require to choose $\gamma$ significantly larger (of order $10^{3}-10^4$) than for Euclidean methods in order to get competitive or even better convergence rate. Moreover, ARDD{\_}E, RDD{\_}E and RSGF disconverge for $\gamma \ge 64, 200, 20$ respectively. So, our empirical observation is as follows: ARDD and RDD with non-Euclidean proximal setup are able to converge with significantly larger stepsizes than its Euclidean counterpart.

We summarize best options for $\gamma$ that we use in the experiments presented in Section~\ref{sec:numerical_experiments} in Table~\ref{tab:nesterov_stepsize_tuning}.
\begin{table}[h]
    \centering
    \begin{tabular}{|c|c|c|c|c|c|}
         \hline
         & ARDD{\_}E & ARDD{\_}NE & RDD{\_}E & RDD{\_}NE & RSGF \\
        \hline
        $n=100$ & $32$ & $2000$ & $32$ & $12000$ & $10$ \\
        \hline
        $n=1000$ & $32$ & $2000$ & $64$ & $3000$ & $4$ \\
        \hline
        $n=5000$ & $32$ & $1000$ & $64$ & $3000$ & $10$ \\
        \hline
    \end{tabular}
    \caption{The optimal choices of $\gamma$ for ARDD, RDD and RSGF applied to minimize Nesterov's function \eqref{eq:nesterov_func} for different dimension $n$.}
    \label{tab:nesterov_stepsize_tuning}
\end{table}

\subsection{Least squares problem}
In addition to the tuning of $\gamma$ in ARDD we also tried different options for $L_2$: instead of $L_2$ from \eqref{eq:L_2_least_squares} we tried $\beta \cdot \frac{\|A\|_F}{\sqrt{r}}$ with different $\beta$. We tried $\beta = 0.001, 0.01, 0.1, 1, 2, 5$ and $10$, but the best results were obtained for $\beta = 0.01$. One can find our numerical results with tuning $\gamma$ in Figure~\ref{fig:ls_tuning}.
\begin{figure}[h]
    \centering
    \includegraphics[width=0.32\textwidth]{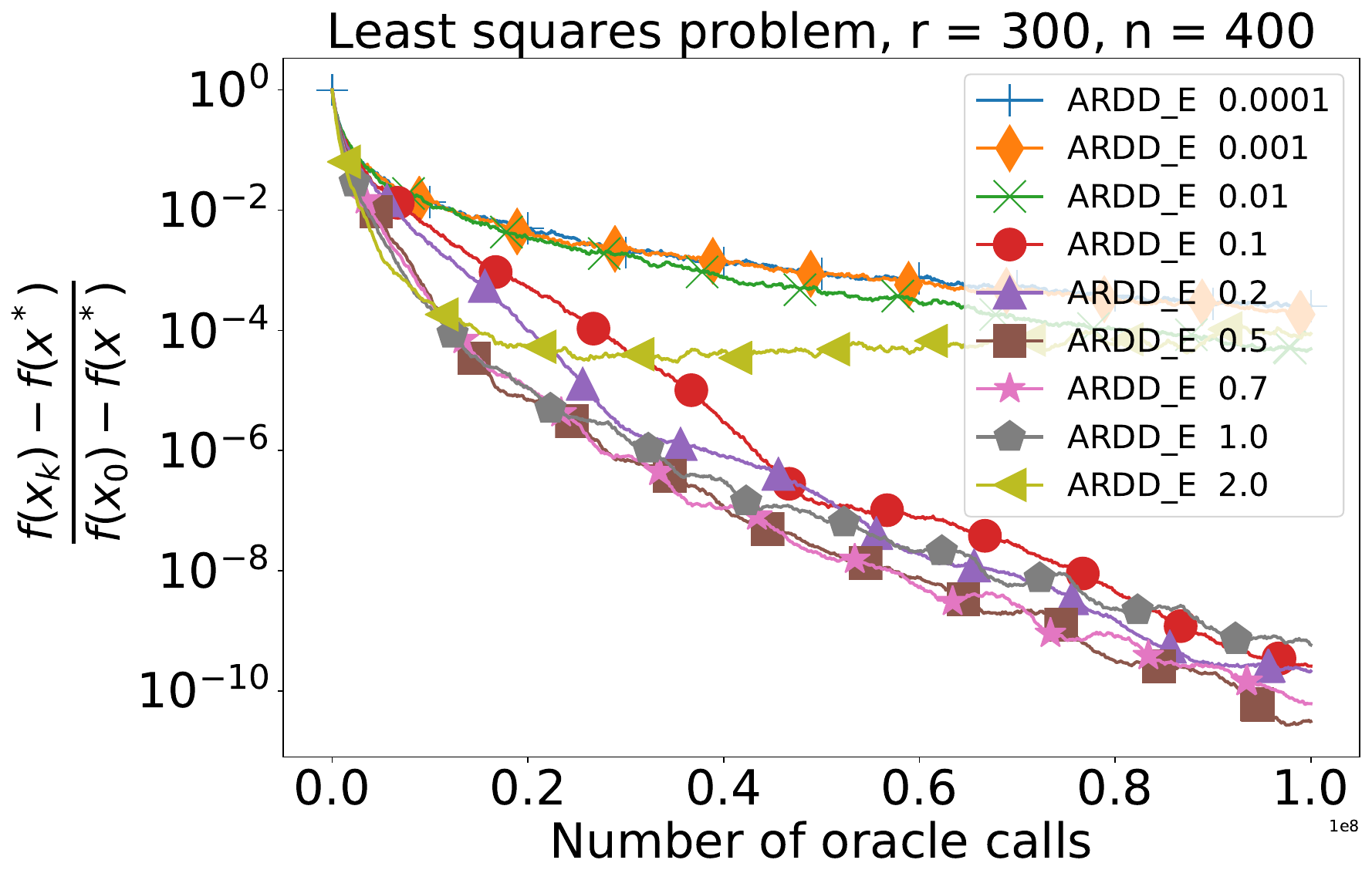}
    \includegraphics[width=0.32\textwidth]{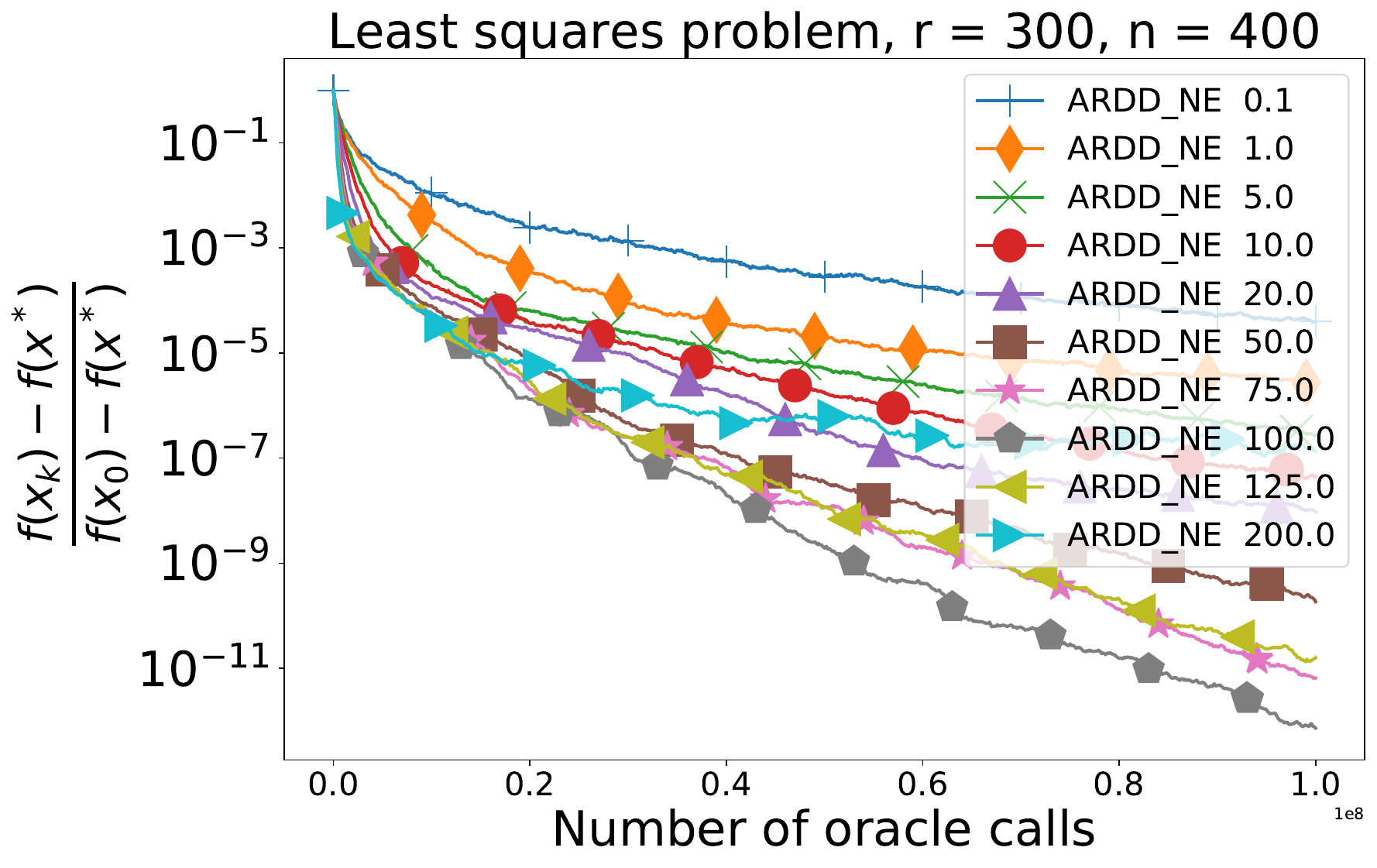}
    \includegraphics[width=0.32\textwidth]{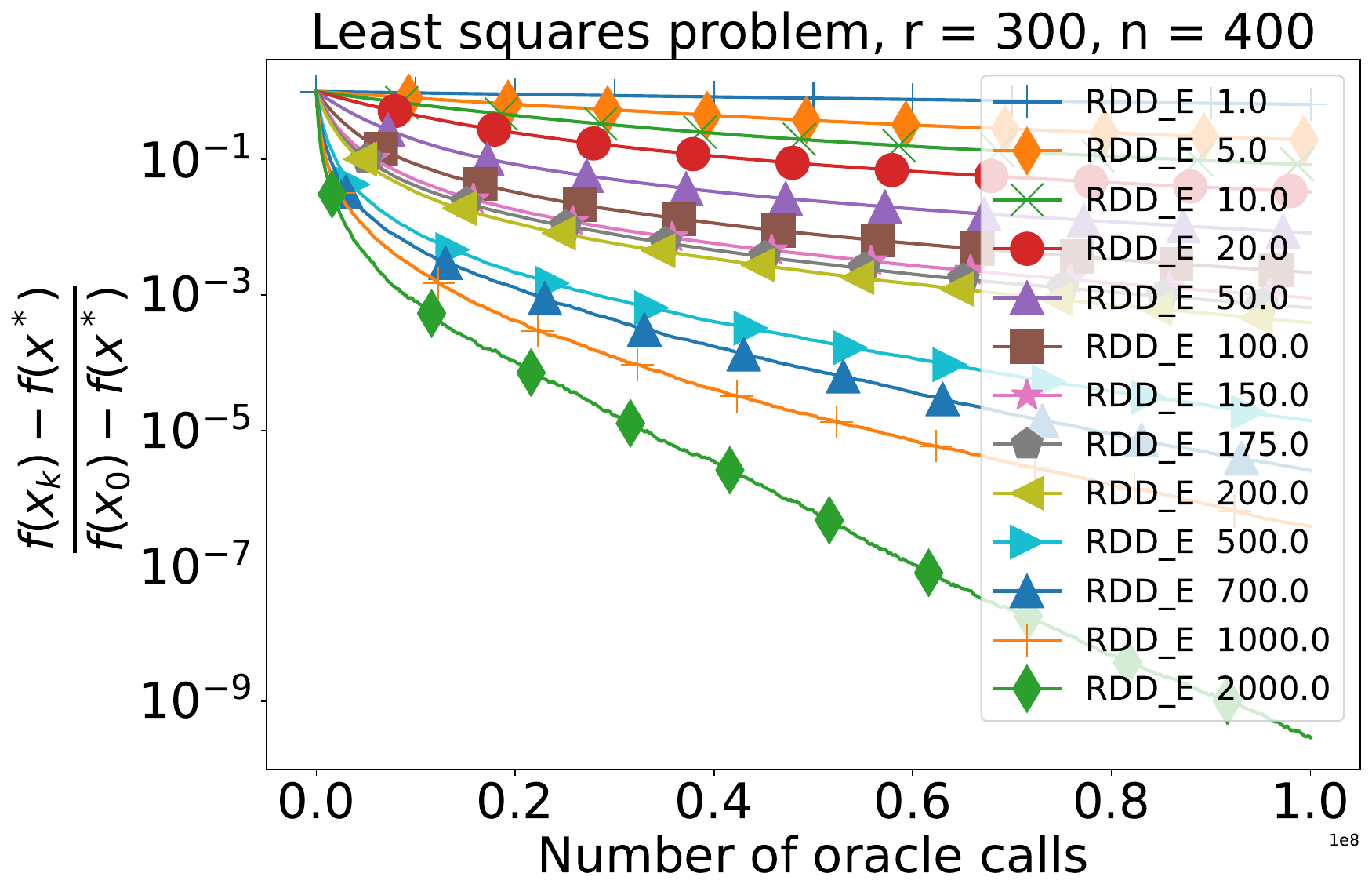}
    \includegraphics[width=0.32\textwidth]{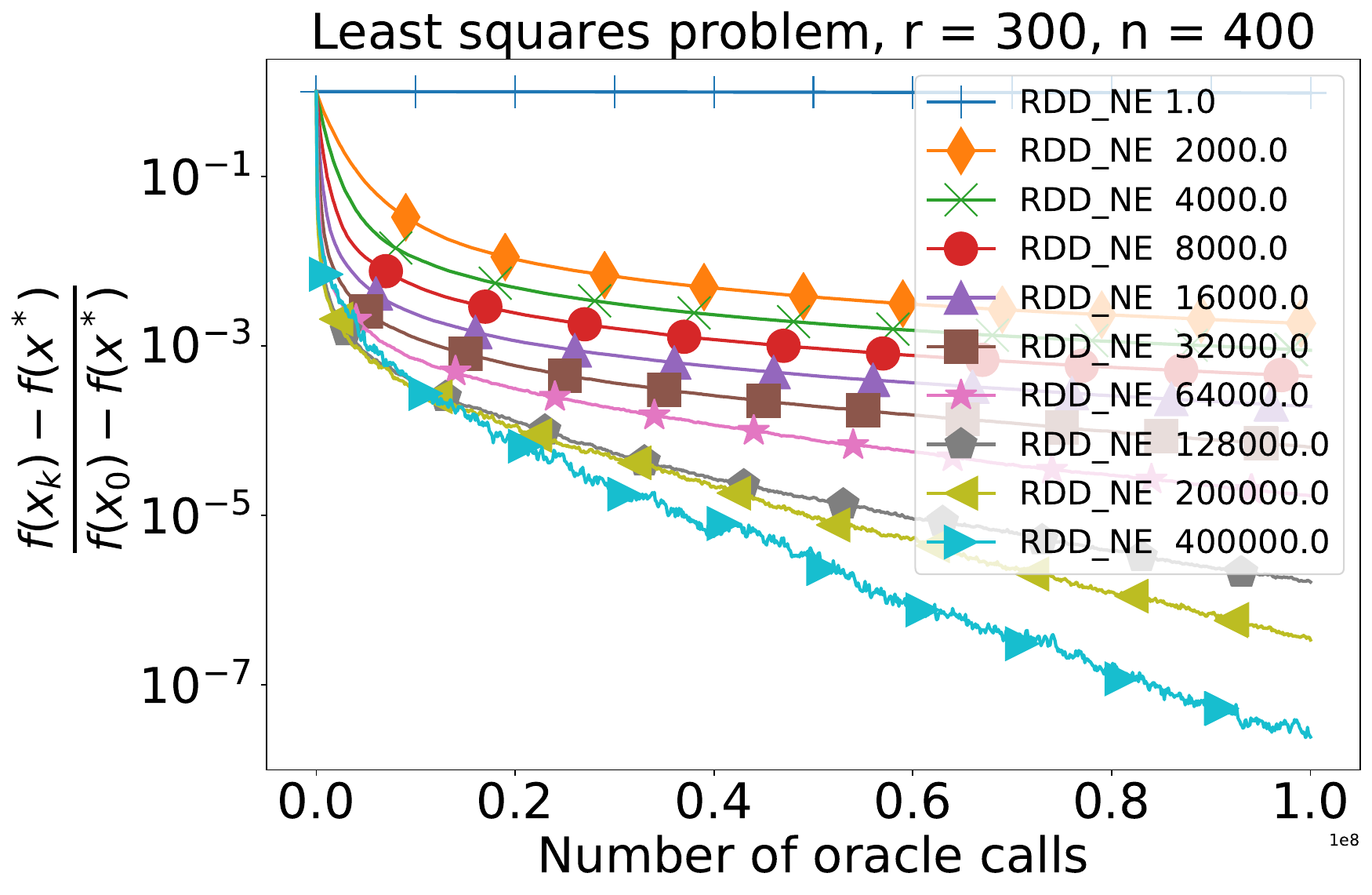}
    \includegraphics[width=0.32\textwidth]{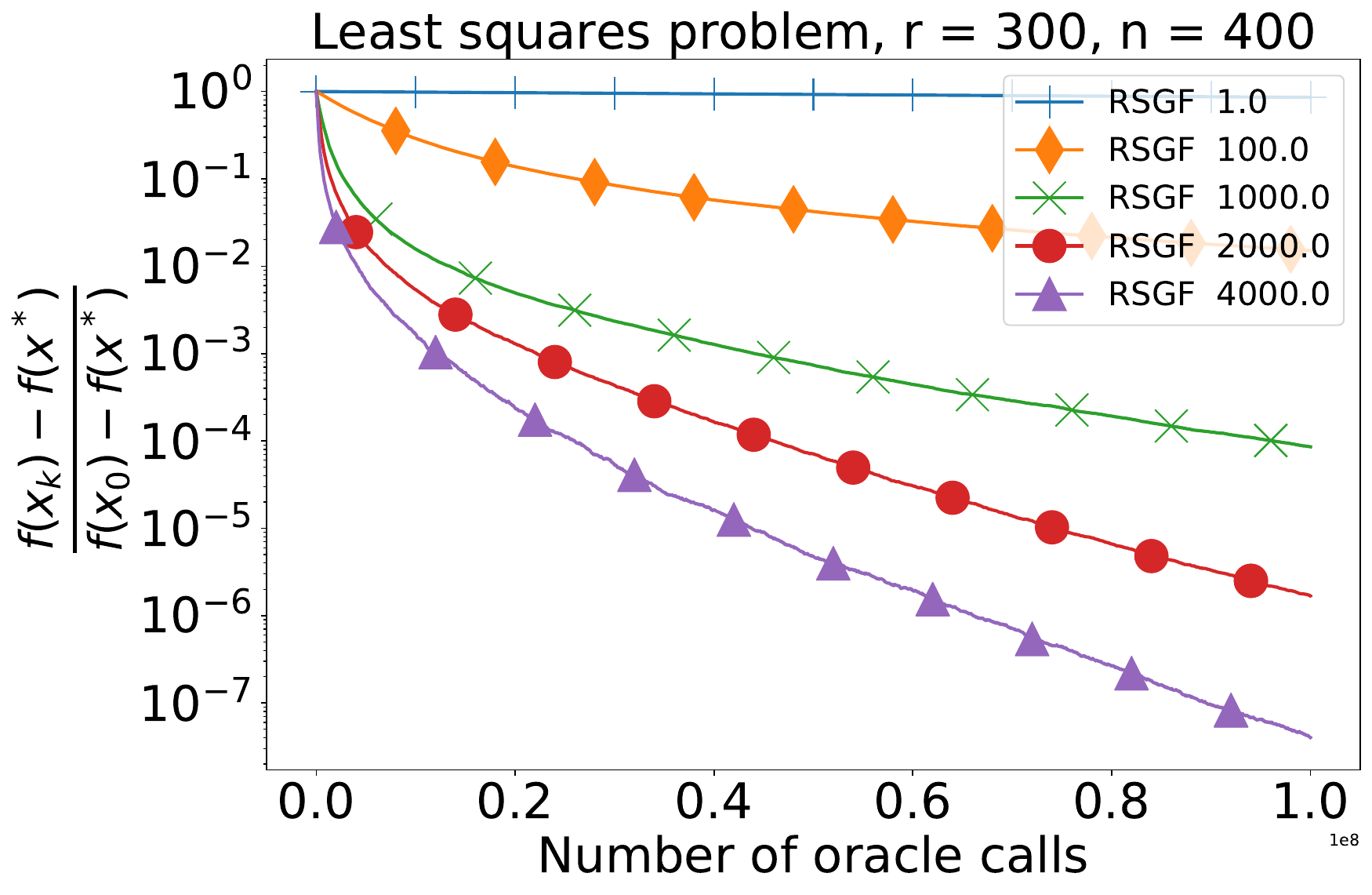}
    \caption{Stepsize tuning for ARDD, RDD and RSGF applied to solve least squares problem \eqref{eq:least_squares_pr}. We use {\_}E and {\_}NE to define $\ell_2$ and $\ell_1$ proximal setups respectively (see \eqref{eq:dp1} and \eqref{eq:dp2} for the details). For all methods batch size $m$ equals $50$. By oracle call we mean one computation of functional value of a summand. Number of oracle calls is divided by $10^8$.}
    \label{fig:ls_tuning}
\end{figure}
Besides $m=50$ we tried different batch sizes. In general, the behaviour of the considered methods was similar after proper parameters tuning.
}



\bibliography{references}

\end{document}